\documentclass[sn-mathphys]{sn-jnl}

\jyear{2021}%

\theoremstyle{thmstyleone}%
\usepackage{subfigure}
\newtheorem{theorem}{Theorem}
\newtheorem{proposition}[theorem]{Proposition}%
\newtheorem{lemma}{Lemma}%
\newtheorem{corollary}{Corollary}%

\theoremstyle{thmstyletwo}%
\newtheorem{example}{Example}%
\newtheorem{remark}{Remark}%

\theoremstyle{thmstylethree}%
\newtheorem{definition}{Definition}%

\usepackage{algorithm,algpseudocode,float}
\usepackage{lipsum}

\makeatletter
\newenvironment{breakablealgorithm}
  {
     \refstepcounter{algorithm}
     \hrule height.8pt depth0pt \kern2pt
     \renewcommand{\caption}[2][\relax]{
       {\raggedright\textbf{\ALG@name~\thealgorithm} ##2\par}%
       \ifx\relax##1\relax 
         \addcontentsline{loa}{algorithm}{\protect\numberline{\thealgorithm}##2}%
       \else 
         \addcontentsline{loa}{algorithm}{\protect\numberline{\thealgorithm}##1}%
       \fi
       \kern2pt\hrule\kern2pt
     }
  }{
     \kern2pt\hrule\relax
  }
\makeatother

\begin{document}

\title[On sketch-and-project methods for solving tensor equations]{On sketch-and-project methods for solving tensor equations}


\author[1]{\fnm{Ling} \sur{Tang}}\email{ltang@cqu.edu.cn}

\author[1]{\fnm{Yanjun} \sur{Zhang}}\email{yjzhang@cqu.edu.cn}

\author*[1]{\fnm{Hanyu} \sur{Li}}\email{lihy.hy@gmail.com or hyli@cqu.edu.cn}

\affil[1]{\orgdiv{College of Mathematics and Statistics}, \orgname{Chongqing University}, \orgaddress{ \city{Chongqing}, \postcode{401331},  \country{P.R. China}}}




\abstract{We ﬁrst propose the regular sketch-and-project method for solving tensor equations with respect to the popular t-product. Then, three adaptive sampling strategies and three corresponding adaptive sketch-and-project methods are derived. We prove that all the proposed methods have linear convergence in expectation. Furthermore, we investigate the Fourier domain versions and some special cases of the new methods, where the latter corresponds to some existing matrix equation methods. Finally, numerical experiments are presented to demonstrate and test the feasibility and effectiveness of the proposed methods for solving tensor equations.}




\maketitle

\section{Introduction}\label{Introduction}

Given the third-order tensors $\mathcal{A}\in\mathbb{R}^{m\times r\times l}$, $\mathcal{B}\in\mathbb{R}^{s\times n \times l}$, and $\mathcal{C}\in\mathbb{R}^{m\times n \times l}$, we consider the following consistent  linear tensor equation
\begin{align}\label{tensorequation}
	\mathcal{A}*\mathcal{X}*\mathcal{B}=\mathcal{C},
\end{align}
where the operator $*$ denotes the t-product introduced by Kilmer and Martin in \cite{kilmer2011factorization} and has been proved to be extremely useful in a variety of  fields, including image and signal processing \cite{kilmer2013third, soltani2016tensor, tarzanagh2018fast}, computer vision \cite{xie2018unifying, yin2018multiview}, data denoising \cite{zhang2018nonlocal}, low-rank tensor completion \cite{semerci2014tensor,  zhang2014novel, zhang2016exact, zhou2017tensor}, and robust tensor PCA \cite{liu2018improved}, among others. One main reason may be that, with t-product, many properties of numerical linear algebra can be extended to third and high order tensors; see \cite{braman2010third, jin2017generalized, lund2020tensor, miao2020generalized, miao2021t,  zheng2021t, qi2021t}. We will go through the fundamentals of t-product in Section \ref{Na_Pre}.


The tensor equation (\ref{tensorequation}) is widely and heavily used in tensor low-rank approximation and decomposition \cite{tarzanagh2018fast}, statistical models \cite{jin2017generalized}, and so on. At present, there are also some work on the computation of this equation. For example, Jin et al. \cite{jin2017generalized} gave a solvability condition and a general solution to (\ref{tensorequation}); El Guide et al. \cite{el2021tensor} generalized the GMRES and Golub-Kahan bidiagonalization methods to solve discrete-ill-posed tensor equations.
Note that, in \cite{el2021tensor}, the authors mainly discussed a special case of (\ref{tensorequation}), i.e., the case of  $\mathcal{A}$ and $\mathcal{B}$ being square tensors. In this paper, we aim to consider the stochastic iterative methods  deeply influenced by the philosophy of the famous Kaczmarz method for solving the general linear tensor equation (\ref{tensorequation}). 

The tensor equation (\ref{tensorequation}) can be regarded as a generalization of the linear matrix equation, that is, the special case for $l=1$ in  (\ref{tensorequation}), whose computation has been dealt with  by some Kaczmarz-type methods. Specifically, Wu et al. \cite{wu2022kaczmarz} presented the randomized Kaczmarz method and combined it with the relaxed greedy selection strategies;  Niu and Zheng  \cite{niu2022global} proposed the randomized block Kaczmarz method and  randomized average block Kaczmarz method. Furthermore, Du et al. \cite{du2022convergence} proposed the randomized block coordinate descent methods for solving the matrix least-squares problem $\min_{X\in\mathbb{R}^{r\times s}}\|C-AXB\|_{F}$, where $\|\cdot\|_{F}$ denotes the Frobenius norm of a matrix. As we know, the randomized Kaczmarz-type and coordinate descent methods can be unified into the sketch-and-project method and its adaptive variants \cite{gower2015randomized,gower2019adaptive,tang2022sketch}. So, more specifically, the stochastic iterative methods we consider for solving (\ref{tensorequation}) in this paper are the sketch-and-project methods. We first propose the regular  sketch-and-project method 
and dub it TESP for short. Then, to further improve the convergence rate, we derive three adaptive sampling strategies and propose the corresponding adaptive variants of the TESP method. 
Meanwhile, we also discuss the Fourier domain versions and some special cases of the proposed methods. 

The remainder of this paper is organized as follows. Section \ref{Na_Pre} introduces some necessary notation and preliminaries. In Section \ref{subTESP}, we present the TESP method and its adaptive variants, and establish their convergence theories. The implementation of the proposed methods in the Fourier domain is described in Section \ref{subTESPF}. In Section \ref{subspecialcase}, we discuss some special cases of the TESP methods.  Section \ref{subexperi} is  devoted to numerical experiments to 
test our methods. Finally, we give the conclusion of the whole paper. The detailed proofs of main lemmas and theorems are provided in the appendix, along with an additional algorithm.

\section{Notation and preliminaries}\label{Na_Pre}
  Throughout this work, scalars, vectors, matrices, and tensors are denoted by lowercase letters, e.g., $a$, boldface lowercase letters, e.g., $\mathbf{a}$, capital letters, e.g., $A$, and calligraphic letters, e.g., $\mathcal{A}$, respectively. For a positive integer $m$, let $[m]:=\{1,\cdots,m\}$.

For a  matrix $A\in \mathbb{R}^{m\times n}$, 
\textbf{Range}($A$)  denotes its 
column space. 
 When $A$ is square, $\lambda_{\max}(A)$, $\lambda_{\min}(A)$, and $\lambda_{\min}^{+}(A)$ represent its largest eigenvalue, smallest eigenvalue, and smallest positive eigenvalue, respectively.

For a third-order tensor $\mathcal{A}\in\mathbb{R}^{m\times n\times l}$, its $(i,j,k)$-th element is denoted as $\mathcal{A}_{(i,j,k)}$. Fibers of $\mathcal{A}$ are defined by fixing two indices, and its $(j,k)$-th column, $(i,k)$-th row and $(i,j)$-th tube fibers are denoted by $\mathcal{A}_{(:,j,k)}$, $\mathcal{A}_{(i,:,k)}$ and $\mathcal{A}_{(i,j,:)}$, respectively. Slices of $\mathcal{A}$ are defined by fixing one index, and its $i$-th horizontal, $j$-th lateral and $k$-th frontal slices are denoted by $\mathcal{A}_{(i,:,:)}$, $\mathcal{A}_{(:,j,:)}$ and $\mathcal{A}_{(:,:,k)}$, respectively. For convenience, the frontal slice $\mathcal{A}_{(:,:,k)}$ is written as $\mathcal{A}_{(k)}$, and $(\mathcal{A}_{(i,:,:)})^{T}$, $(\mathcal{A}_{(:,j,:)})^{T}$ and $(\mathcal{A}_{(k)})^{T}$ are simply denoted  as $\mathcal{A}_{(i,:,:)}^{T}$, $\mathcal{A}_{(:,j,:)}^{T}$ and $\mathcal{A}_{(k)}^{T}$, respectively.

 In this paper, we also refer to third-order tensors as tubal matrices. The details are described in the following definition.
 \begin{definition}[\cite{kilmer2013third}]\rm
 	An element $\mathbf{a}\in \mathbb{R}^{1\times 1\times l}$ is called a tubal scalar of length $l$ and the set consisting of all tubal scalars of length $l$ is denoted by $\mathbb{K}_{l}$; an element $\overrightarrow{\mathcal{A}} \in \mathbb{R}^{m\times 1\times l}$ 
 	is  called a vector of tubal scalars of length $l$ with size $m$ and the corresponding set is denoted by $\mathbb{K}^{m}_{l}$; 
 	an element $\mathcal{A} \in \mathbb{R}^{m\times n\times l}$ is  called a matrix of tubal scalars of length $l$ with size $m\times n$ and the corresponding set is denoted by $\mathbb{K}^{m\times n}_{l}$.
 \end{definition}

Now, we introduce the definition of the t-product. 
\begin{definition}[t-product \cite{kilmer2011factorization}] \rm
	Given $\mathcal{A}\in \mathbb{K}^{m\times n}_{l}$ and $\mathcal{B}\in \mathbb{K}^{n\times r}_{l}$, the t-product $\mathcal{A}*\mathcal{B}\in \mathbb{K}^{m\times r}_{l}$ is defined as
	$$\mathcal{A}*\mathcal{B}=\text{fold}(\text{bcirc}(\mathcal{A})\text{unfold}(\mathcal{B})),$$
	where 
	$$\text{bcirc}(\mathcal{A}):=
	\begin{bmatrix}
		\mathcal{A}_{(1)} & \mathcal{A}_{(l)} & \cdots & \mathcal{A}_{(2)}\\
		\mathcal{A}_{(2)} & \mathcal{A}_{(1)} & \cdots & \mathcal{A}_{(3)}  \\
		\vdots & \vdots & \ddots & \vdots \\
		\mathcal{A}_{(l)} & \mathcal{A}_{(l-1)} & \cdots & \mathcal{A}_{(1)} \\
	\end{bmatrix}
	,~~\text{unfold}(\mathcal{A}):=\begin{bmatrix}
		\mathcal{A}_{(1)} \\
		\mathcal{A}_{(2)} \\
		\vdots \\
		\mathcal{A}_{(l)}
	\end{bmatrix},$$
 and  $\text{fold}(\text{unfold}(\mathcal{A})):=\mathcal{A}$.
\end{definition}

Using the Matlab commands $\texttt{fft}$ and $\texttt{ifft}$, the t-product can be computed by the discrete Fourier transform (DFT), 
as shown in Algorithm \ref{t-product}.
	\begin{algorithm}[htbp]
	\caption{t-product $\mathcal{C}=\mathcal{ A}*\mathcal{ B}$ in the Fourier domain \cite{kilmer2011factorization}} \label{t-product}
	\begin{algorithmic}[1]
 \State  \textbf{Input:} $\mathcal{A}\in \mathbb{K}^{m\times n }_{l}$, $\mathcal{B}\in \mathbb{K}^{n\times r }_{l}$
		\State $\widehat{\mathcal{ A}}=\texttt{fft}(\mathcal{ A},[~],3)$ and $\widehat{\mathcal{B}}=\texttt{fft}(\mathcal{ B},[~],3)$
		\For{$k=1,\cdots,l$}
		\State $\widehat{\mathcal{ C}}_{(k)}=\widehat{\mathcal{ A}}_{(k)}\widehat{\mathcal{ B}}_{(k)}$
		\EndFor
		\State $\mathcal{ C}=\texttt{ifft}(\widehat{\mathcal{ C}},[~],3)$. 
  \State  \textbf{Output:} $\mathcal{C}$
	\end{algorithmic}
\end{algorithm}
Moreover, due to the special structure of the DFT, the $l$ matrix-matrix multiplications in Algorithm \ref{t-product} can be reduced to $\lceil\frac{l+1}{2}\rceil$, where $\lceil n \rceil$ means the nearest integer number larger than or equal to $n$. That is, we can replace the ``for" loop in Algorithm \ref{t-product} with the following computations \cite{lu2019tensor}:
 \begin{equation*}
 	\left\{
 	\begin{array}{lcl}
 		\widehat{\mathcal{ C}}_{(k)}=\widehat{\mathcal{ A}}_{(k)}\widehat{\mathcal{ B}}_{(k)},&&\text{for}~k=1,\cdots,\lceil\frac{l+1}{2}\rceil,\\
 		\widehat{\mathcal{C}}_{(k)}=\text{conj}(\widehat{\mathcal{C}}_{(l-k+2)}), & &\text{for}~ k=\lceil\frac{l+1}{2}\rceil+1,\cdots,l.
 	\end{array} \right.
 \end{equation*}
We will use them in all our simulations.

Next, we review some definitions and properties related to t-product, which will be necessary later in this paper. For details, refer to \cite{kilmer2011factorization, jin2017generalized, qi2021t, zheng2021t, kilmer2013third}.
\begin{definition}[identity tubal matrix \cite{kilmer2011factorization}]\rm
	The identity tubal matrix $\mathcal{I}_{n}\in \mathbb{K}^{n\times n }_{l}$ is the tubal matrix whose first frontal slice is the $n\times n$ identity matrix, and whose other frontal slices are all zeros. 
\end{definition}

\begin{definition}[inverse \cite{kilmer2011factorization}]\rm
	For a tubal matrix $\mathcal{A}\in \mathbb{K}^{n\times n}_{l}$, if there exists $\mathcal{B} \in \mathbb{K}^{n\times n}_{l}$ such that
	$$\mathcal{A}*\mathcal{B}=\mathcal{I}_{n}\quad \textrm{and}\quad \mathcal{B}*\mathcal{A}=\mathcal{I}_{n},$$
	then $\mathcal{A}$ is said to be invertible, and $\mathcal{B}$ is the inverse of $\mathcal{A}$, which is denoted by $\mathcal{A}^{-1}$.
\end{definition}

\begin{definition}[Moore-Penrose inverse \cite{jin2017generalized}]\rm
	For a tubal matrix $\mathcal{A}\in \mathbb{K}^{m\times n}_{l}$, if there exists $\mathcal{B} \in \mathbb{K}^{n\times m}_{l}$ such that
	$$\mathcal{A}*\mathcal{B}*\mathcal{A}=\mathcal{A},\quad\mathcal{B}*\mathcal{A}*\mathcal{B}=\mathcal{B},\quad (\mathcal{A}*\mathcal{B})^{T}=\mathcal{A}*\mathcal{B},\quad (\mathcal{B}*\mathcal{A})^{T}=\mathcal{B}*\mathcal{A},$$
	then $\mathcal{B}$ is called the Moore-Penrose inverse of $\mathcal{A}$ and is denoted by $\mathcal{A}^{\dag}$.
\end{definition}

\begin{lemma}[\cite{jin2017generalized}]
	The Moore-Penrose inverse of any tubal matrix $\mathcal{A}\in \mathbb{K}^{m\times n}_{l}$ exists and is unique, and if $\mathcal{A}$ is invertible, then $\mathcal{A}^{\dag}=\mathcal{A}^{-1}$.
\end{lemma}

\begin{definition}[transpose \cite{kilmer2011factorization}]\rm
	For a tubal matrix $\mathcal{A}\in \mathbb{K}^{m\times n}_{l}$, the transpose $\mathcal{A}^{T}$ is defined by transposing each of the frontal slices of $\mathcal{A}$ and then reversing the order of transposed frontal slices $2$ through $l$, that is
	\begin{equation*}
		\left\{
		\begin{array}{lcl}
			(\mathcal{ A}^{T})_{(1)}=\mathcal{ A}_{(1)}^{T}, && \\
			(\mathcal{ A}^{T})_{(k)}=\mathcal{ A}_{(l-k+2)}^{T}, & &\text{for}~ k=2,\cdots,l.
		\end{array} \right.
	\end{equation*}
\end{definition}
In addition, we define the slice transpose $\mathcal{A}^{ST}$ by transposing each of the frontal slices of $\mathcal{A}$, that is
 \begin{equation*}
 		(\mathcal{ A}^{ST})_{(k)}=\mathcal{ A}_{(k)}^{T}, \quad\text{for}~ k=1,\cdots,l,
 \end{equation*}
 and the reverse $\mathcal{A}^{R}$ by reversing the order of $\mathcal{A}$'s frontal slices $2$ through $l$, that is
  \begin{equation*}
  	\left\{
  	\begin{array}{lcl}
  		(\mathcal{ A}^{R})_{(1)}=\mathcal{ A}_{(1)}, && \\
  		(\mathcal{ A}^{R})_{(k)}=\mathcal{ A}_{(l-k+2)}, & &\text{for}~ k=2,\cdots,l.
  	\end{array} \right.
  \end{equation*}
\begin{lemma}
		Let $\mathcal{ A}$ and $\mathcal{ B}$ be tubal matrices of any multiplicable dimension. Then
	\begin{enumerate}
		\item $(\mathcal{ A}^{ST})^{ST}=\mathcal{ A}$, $(\mathcal{ A}*\mathcal{ B})^{ST}=\mathcal{ B}^{ST}*\mathcal{ A}^{ST}$, $(\mathcal{ A}^{-1})^{ST}=(\mathcal{ A}^{ST})^{-1}$, $(\mathcal{ A}^{\dag})^{ST}=(\mathcal{ A}^{ST})^{\dag}$; 
		\item $(\mathcal{ A}^{R})^{R}=\mathcal{ A}$, $(\mathcal{ A}*\mathcal{ B})^{R}=\mathcal{ B}^{R}*\mathcal{ A}^{R}$, $(\mathcal{ A}^{-1})^{R}=(\mathcal{ A}^{R})^{-1}$, $(\mathcal{ A}^{\dag})^{R}=(\mathcal{ A}^{R})^{\dag}$;
		\item $\mathcal{A}^{T}=(\mathcal{A}^{ST})^{R}=(\mathcal{A}^{R})^{ST}$;
		\item $\widehat{\mathcal{ A}^{T}}_{(k)}=\widehat{\mathcal{ A}}_{(k)}^{H}$, $\widehat{\mathcal{ A}^{ST}}_{(k)}=\widehat{\mathcal{ A}}_{(k)}^{T}$ and $\widehat{\mathcal{ A}^{R}}_{(k)}=\text{\rm conj}(\widehat{\mathcal{ A}}_{(k)})$ for $k=1,\cdots,l$,  where $\widehat{\mathcal{ A}^{T}}=\texttt{\rm fft}(\mathcal{ A}^{T},[~],3)$, $\widehat{\mathcal{ A}^{ST}}=\texttt{\rm fft}(\mathcal{ A}^{ST},[~],3)$, $\widehat{\mathcal{ A}^{R}}=\texttt{\rm fft}(\mathcal{ A}^{R},[~],3)$ and $\widehat{\mathcal{ A}}_{(k)}^{H}$, $\widehat{\mathcal{ A}}_{(k)}^{T}$, and $\text{\rm conj}(\widehat{\mathcal{ A}}_{(k)})$ represent the conjugate transpose, transpose, and conjugate of $\widehat{\mathcal{ A}}_{(k)}$, respectively.
	\end{enumerate}
\end{lemma}
\begin{proof}
    The proof is straightforward, but tedious, so we omit it here.
\end{proof}
\begin{definition}[T-symmetric \cite{kilmer2011factorization}]\rm
	For a tubal matrix $\mathcal{A}\in \mathbb{K}^{n\times n}_{l}$, it is T-symmetric if $\mathcal{A}=\mathcal{A}^{T}$. 
\end{definition}

It is easy to see that if  $\mathcal{ A}$ is T-symmetric, then $\mathcal{A}^{ST}=\mathcal{A}^{R}$.

\begin{definition} [orthogonal tubal matrix \cite{kilmer2011factorization}]\rm
	For a tubal matrix $\mathcal{A}\in \mathbb{K}^{n\times n}_{l}$, it is orthogonal if $\mathcal{A}^{T}*\mathcal{A}=\mathcal{A}*\mathcal{A}^{T}=\mathcal{I}$.
\end{definition}

\begin{definition}[\cite{kilmer2011factorization}]\rm
	For a tubal matrix
	$\mathcal{A} \in \mathbb{K}^{m\times n}_{l}$,  define
 \begin{align*}
     \mathbf{Range}(\mathcal{A})=\left\{\overrightarrow{\mathcal{Y}}\in\mathbb{K}^{m}_{l}\mid\overrightarrow{\mathcal{Y}}=\mathcal{A}*\overrightarrow{\mathcal{X}},~\text{for}~\text{any}~\overrightarrow{\mathcal{X}}\in \mathbb{K}^{n}_{l}\right\}.
 \end{align*}
\end{definition}

\begin{definition}[\cite{kilmer2013third}]\rm
	For  a tubal matrix
	$\mathcal{P}\in \mathbb{K}^{n\times n}_{l}$, it is a projector if $\mathcal{P}^{2}=\mathcal{P}*\mathcal{P}=\mathcal{P}$, and is orthogonal projector if $\mathcal{P}^{T}=\mathcal{P}$ also holds. 
\end{definition}

Note that $\mathcal{A}*(\mathcal{A}^{T}*\mathcal{A})^{\dag}*\mathcal{A}^{T}$ is an orthogonal projector onto $\mathbf{Range}(\mathcal{A})$.
\begin{lemma} \label{sub1lem4}
Let $\mathcal{ A}\in \mathbb{K}^{m\times n}_{l}$ be any tubal matrix. Then
$$\|\mathcal{A}-\mathcal{P}_{1}*\mathcal{A}*\mathcal{P}_{2}\|_{F}^{2}=\|\mathcal{A}\|_{F}^{2}-\|\mathcal{P}_{1}*\mathcal{A}*\mathcal{P}_{2}\|_{F}^{2},$$ where $\mathcal{P}_{1}$ and $\mathcal{P}_{2}$ are orthogonal projectors.
\end{lemma}
\begin{proof}
Since  $\mathcal{P}_{1}$ and $\mathcal{P}_{2}$ are orthogonal projectors,  $\text{bcirc}(\mathcal{P}_{1})$ and $\text{bcirc}(\mathcal{P}_{2})$ are also orthogonal projectors, which will be given in Proposition \ref{sec2-pro1}. Thus
\begin{align*}
\|\mathcal{A}-\mathcal{P}_{1}*\mathcal{A}*\mathcal{P}_{2}\|_{F}^{2}&=\frac{1}{l}\|\text{bcirc}(\mathcal{A}-\mathcal{P}_{1}*\mathcal{A}*\mathcal{P}_{2})\|_{F}^{2}\\
&=\frac{1}{l}\|\text{bcirc}(\mathcal{A})-\text{bcirc}(\mathcal{P}_{1})\text{bcirc}(\mathcal{A})\text{bcirc}(\mathcal{P}_{2})\|_{F}^{2}\\
&=\frac{1}{l}\|\text{bcirc}(\mathcal{A})\|_{F}^{2}-\frac{1}{l}\|\text{bcirc}(\mathcal{P}_{1})\text{bcirc}(\mathcal{A})\text{bcirc}(\mathcal{P}_{2})\|_{F}^{2}\\
&=\frac{1}{l}\|\text{bcirc}(\mathcal{A})\|_{F}^{2}-\frac{1}{l}\|\text{bcirc}(\mathcal{P}_{1}*\mathcal{A}*\mathcal{P}_{2})\|_{F}^{2}\\
&=\|\mathcal{A}\|_{F}^{2}-\|\mathcal{P}_{1}*\mathcal{A}*\mathcal{P}_{2}\|_{F}^{2}.
\end{align*}
\end{proof}

\begin{definition}[T-symmetric T-positive (semi)definite \cite{zheng2021t}]\rm
	For  a tubal matrix $\mathcal{A}\in \mathbb{K}^{n\times n }_{l}$, 
	it is T-symmetric T-positive (semi)definite if $\mathcal{A}$ is  T-symmetric and $\langle\overrightarrow{\mathcal{X}},\mathcal{A}*\overrightarrow{\mathcal{X}}\rangle>(\geq) 0$ holds for any nonzero $\overrightarrow{\mathcal{X}}\in \mathbb{K}^{n}_{l}$ (for any $\overrightarrow{\mathcal{X}}\in \mathbb{K}^{n}_{l}$).
\end{definition}

\begin{proposition}[\cite{qi2021t, zheng2021t}]\label{sec2-pro1}
	For a tubal matrix
	$\mathcal{A}\in \mathbb{K}^{n\times n}_{l}$, it is T-symmetric if and only if $\text{bcirc}(\mathcal{A})$ is symmetric, 
	is invertible if and only if $\text{bcirc}(\mathcal{A})$ is invertible, 
	is orthogonal if and only if $\text{bcirc}(\mathcal{A})$ is orthogonal, and is T-symmetric T-positive (semi)definite if and only if $\text{bcirc}(\mathcal{A})$ is symmetric positive (semi)definite 
	if and only if $\widehat{\mathcal{A}}_{(k)}$ for $k=1,\cdots,l$ are all Hermitian positive (semi)definite.
\end{proposition}

\begin{definition}[\cite{tang2022sketch}]\rm
For a T-symmetric T-positive (semi)definite tubal matrix $\mathcal{A}\in \mathbb{K}^{n\times n }_{l}$, its square root is defined as $\mathcal{A}^{\frac{1}{2}}=\text{bcirc}^{-1}(\text{bcirc}(\mathcal{A})^{\frac{1}{2}})$, where $\text{bcirc}^{-1}(\cdot)$ denotes the inverse operation of $\text{bcirc}(\cdot)$. Moreover, $\mathcal{A}=\mathcal{A}^{\frac{1}{2}}*\mathcal{A}^{\frac{1}{2}}$ and $\text{bcirc}(\mathcal{A}^{\frac{1}{2}})=\text{bcirc}(\mathcal{A})^{\frac{1}{2}}$.
\end{definition}

Furthermore, we can prove the following results.
\begin{lemma}
	Let $\mathcal{ A}$ be  any T-symmetric T-positive (semi)definite tubal matrix. Then
	\begin{enumerate}
		\item $\mathcal{ A}^{ST}$ is also a T-symmetric T-positive (semi)deﬁnite tubal matrix and $(\mathcal{A}^{ST})^{\frac{1}{2}}=(\mathcal{A}^{\frac{1}{2}})^{ST}$;
		\item $\mathcal{ A}^{R}$ is also a T-symmetric T-positive (semi)deﬁnite tubal matrix and $(\mathcal{ A}^{R})^{\frac{1}{2}}=(\mathcal{ A}^{\frac{1}{2}})^{R}$.
	\end{enumerate}
\end{lemma}

In the following, we give some new definitions or results, which are essential for the subsequent proposed methods and their corresponding convergence analysis.

\begin{definition}\rm
	For a tubal matrix $\mathcal{A}\in\mathbb{K}_{l}^{m\times n}$, its t-vectorization is denoted by $\text{vec}_{t}(\mathcal{ A})$ and defined as
   $$\text{vec}_{t}(\mathcal{ A})\overset{\text{def}}{=}\begin{bmatrix}
		\mathcal{ A}_{(:,1,:)}  \\
		\mathcal{ A}_{(:,2,:)} \\
		\vdots   \\
		\mathcal{ A}_{(:,n,:)} \\
	\end{bmatrix}\in\mathbb{K}_{l}^{mn}.$$ 
\end{definition}

It is true that $\widehat{\text{vec}_{t}(\mathcal{ A})}_{(k)}=\text{vec}(\widehat{\mathcal{ A}}_{(k)})$ for $k=1,\cdots,l$, where $\widehat{\text{vec}_{t}(\mathcal{ A})}=\texttt{fft}(\text{vec}_{t}(\mathcal{ A}),[~],3)$ and $\text{vec}(\cdot)$ means the vectorization of a matrix.

\begin{definition}\rm
	For the tubal matrices $\mathcal{ A}\in\mathbb{K}_{l}^{m\times n}$ and $\mathcal{ B}\in\mathbb{K}_{l}^{r\times s}$, their t-Kronecker product is denoted by $\mathcal{A}\otimes_{t}\mathcal{B}$ and defined as
	$$\mathcal{A}\otimes_{t}\mathcal{B}\overset{\text{def}}{=}\begin{bmatrix}
		\mathcal{ A}_{(1,1,:)}*\mathcal{ B} & \mathcal{ A}_{(1,2,:)}*\mathcal{ B} &  \cdots & \mathcal{ A}_{(1,n,:)} *\mathcal{ B} \\
		\mathcal{ A}_{(2,1,:)}*\mathcal{ B}&  \mathcal{ A}_{(2,2,:)}*\mathcal{ B}&  \cdots  & \mathcal{ A}_{(2,n,:)} *\mathcal{ B} \\
		\vdots 	& \vdots 	&   \ddots &  	\vdots  \\
		\mathcal{ A}_{(m,1,:)}*\mathcal{ B} &\mathcal{ A}_{(m,2,:)}*\mathcal{ B} &  \cdots  &  \mathcal{ A}_{(m,n,:)} *\mathcal{ B}\\
	\end{bmatrix}\in\mathbb{K}_{l}^{mr\times ns}.$$
\end{definition}

Let $\widehat{\mathcal{ A}}=\texttt{fft}(\mathcal{ A},[~],3)$, $\widehat{\mathcal{ B}}=\texttt{fft}(\mathcal{ B},[~],3)$ and  $\widehat{\mathcal{ C}}=\texttt{fft}(\mathcal{A}\otimes_{t}\mathcal{B},[~],3)$. Then it is easy to check that $\widehat{\mathcal{ C}}_{(k)}=\widehat{\mathcal{ A}}_{(k)}\otimes\widehat{\mathcal{ B}}_{(k)}$ for $k=1,\cdots,l$, where $A\otimes B$ stands for the Kronecker product of the two matrices $A$ and $B$.
\begin{lemma}
	Let $\mathcal{ A}$, $\mathcal{ B}$, $\mathcal{ C}$ and $\mathcal{ D}$ be tubal matrices of any multiplicable dimension. The following results hold.
	\begin{enumerate}
		\item $\text{\rm vec}_{t}(\mathcal{ A}*\mathcal{ B}*\mathcal{ C})=(\mathcal{ C}^{ST}\otimes_{t}\mathcal{ A})*\text{\rm vec}_{t}(\mathcal{ B})$;
		\item $(\mathcal{ A}\otimes_{t}\mathcal{ B})^{T}=\mathcal{ A}^{T}\otimes_{t}\mathcal{ B}^{T}$,  $(\mathcal{ A}\otimes_{t}\mathcal{ B})^{ST}=\mathcal{ A}^{ST}\otimes_{t}\mathcal{ B}^{ST}$, $(\mathcal{ A}\otimes_{t}\mathcal{ B})^{R}=\mathcal{ A}^{R}\otimes_{t}\mathcal{ B}^{R}$;
		
		\item $(\mathcal{ A}*\mathcal{ B})\otimes_{t}(\mathcal{ C}*\mathcal{ D})=(\mathcal{ A}\otimes_{t}\mathcal{ C})*(\mathcal{ B}\otimes_{t}\mathcal{ D})$;
		
		\item $\|\mathcal{ A}\otimes_{t}\mathcal{ B}\|_{F}^{2}=\frac{1}{l}\sum_{k=1}^{l}\|\widehat{\mathcal{ A}}_{(k)}\|_{F}^{2}\|\widehat{\mathcal{ B}}_{(k)}\|_{F}^{2}$ ;
		
		\item $(\mathcal{ A}\otimes_{t}\mathcal{ B})^{\dag}=\mathcal{ A}^{\dag}\otimes_{t}\mathcal{ B}^{\dag}$;
		
		\item $\lambda_{\min}({\text{\rm bcirc}}(\mathcal{ A}\otimes_{t}\mathcal{ B}))\geq\lambda_{\min}({\text{\rm bcirc}}(\mathcal{ A}))\lambda_{\min}({\text{\rm bcirc}}(\mathcal{ B}))$;
		
		\item If $\mathcal{ A}$ and $\mathcal{ B}$ are invertible, then $\mathcal{ A}\otimes_{t}\mathcal{ B}$ is also invertible, and $\mathcal{ A}^{-1}\otimes_{t}\mathcal{ B}^{-1}$ is the inverse of $\mathcal{ A}\otimes_{t}\mathcal{ B}$, that is $(\mathcal{ A}\otimes_{t}\mathcal{ B})^{-1}=\mathcal{ A}^{-1}\otimes_{t}\mathcal{ B}^{-1}$;
		
		\item If $\mathcal{ A}$ and $\mathcal{ B}$ are T-symmetric T-positive (semi)deﬁnite, then $\mathcal{ A}\otimes_{t}\mathcal{ B}$ is also T-symmetric T-positive (semi)deﬁnite;
		
		\item If $\mathcal{ A}$ and $\mathcal{ B}$ are orthogonal projectors, then $\mathcal{ A}\otimes_{t}\mathcal{ B}$ is also an orthogonal projector.
		
	\end{enumerate}
\end{lemma}
\begin{proof}
The proof  is straightforward, but tedious, so we omit it here.
\end{proof}
\begin{definition}[\cite{tang2022sketch}]\rm 
	Let $\mathcal{M}\in \mathbb{K}^{n\times n}_{l}$ be a T-symmetric 
	T-positive definite tubal matrix. For any tubal vectors $\overrightarrow{\mathcal{X}}$, $\overrightarrow{\mathcal{Y}}\in \mathbb{K}^{n}_{l}$, their weighted inner product and the weighted induced norm are defined as
	$$\langle\overrightarrow{\mathcal{X}},\overrightarrow{\mathcal{Y}}\rangle _{\mathcal{M}}=\langle\mathcal{M}*\overrightarrow{\mathcal{X}},\overrightarrow{\mathcal{Y}}\rangle \quad\textrm{ and }\quad \|\overrightarrow{\mathcal{X}}\|_{\mathcal{M}}=\sqrt{\langle\overrightarrow{\mathcal{X}},\overrightarrow{\mathcal{X}}\rangle _{\mathcal{M}}},
	$$
	respectively. 
\end{definition}
\begin{definition}\rm
	Let $\mathcal{M}\in\mathbb{K}^{m\times m}_{l}$ and $\mathcal{N}\in\mathbb{K}^{n\times n}_{l}$ be T-symmetric T-positive definite tubal matrices. For any tubal matrix $\mathcal{ A}\in\mathbb{R}^{m\times n}_{l}$, define 
	\begin{align*}
		\|\mathcal{ A}\|_{F(\mathcal{ M},\mathcal{ N})}\overset{\text{def}}=\|\text{vec}_{t}(\mathcal{ A})\|_{(\mathcal{ N}^{ST}\otimes_{t} \mathcal{ M})}=\|\text{vec}_{t}(\mathcal{ A})\|_{(\mathcal{ N}^{R}\otimes_{t} \mathcal{ M})}.
	\end{align*}
 
In addition, for the T-symmetric T-positive semidefinite tubal matrices $\mathcal{M}\in\mathbb{K}^{m\times m}_{l}$ and $\mathcal{N}\in\mathbb{K}^{n\times n}_{l}$, we define 
$$\|\mathcal{ A}\|_{F(\mathcal{ M},\mathcal{ N})}\overset{\text{def}}=\|\text{vec}_{t}(\mathcal{ A})\|_{(\mathcal{ N}^{ST}\otimes_{t} \mathcal{ M})}=\|\text{vec}_{t}(\mathcal{ A})\|_{(\mathcal{ N}^{R}\otimes_{t} \mathcal{ M})},$$
where $\|\cdot\|_{(\mathcal{ D})}$ is the seminorm induced by a T-symmetric T-positive semidefinite tubal matrx $\mathcal{ D}$ \cite{tang2022sketch}. 
	\end{definition}
	
We can check that
	\begin{align*}
		\|\mathcal{ A}\|_{F(\mathcal{ M},\mathcal{ N})}^{2}&=\|\text{vec}_{t}(\mathcal{ A})\|_{(\mathcal{ N}^{ST}\otimes_{t} \mathcal{ M})}^{2}
		=\langle(\mathcal{ N}^{ST}\otimes_{t}  \mathcal{ M})*\text{vec}_{t}(\mathcal{ A}),\text{vec}_{t}(\mathcal{ A})\rangle\\
		&=\langle((\mathcal{ N}^{\frac{1}{2}})^{ST}\otimes_{t}  \mathcal{ M}^{\frac{1}{2}})*\text{vec}_{t}(\mathcal{ A}),((\mathcal{ N}^{\frac{1}{2}})^{ST}\otimes_{t}  \mathcal{ M}^{\frac{1}{2}})*\text{vec}_{t}(\mathcal{ A})\rangle\\
		&=\|\text{vec}_{t}(\mathcal{ M}^{\frac{1}{2}}*\mathcal{ A}*\mathcal{ N}^{\frac{1}{2}})\|_{F}^{2}
  =\|\mathcal{ M}^{\frac{1}{2}}*\mathcal{ A}*\mathcal{ N}^{\frac{1}{2}}\|_{F}^{2}.
	\end{align*}

Finally, we  give the definitions of two common used sketching tubal matrices.

\begin{definition}[Gaussian random tubal matrix \cite{zhang2018randomizedddd}]\label{guassrandom}\rm
	A tubal matrix $\mathcal{S} \in\mathbb{K}^{m\times \tau}_{l}$ is called a Gaussian random tubal matrix, if the elements of $\mathcal{S}_{(1)}$ satisfy the standard normal distribution, and other frontal slices are all zeros.
\end{definition}

\begin{definition}[random sampling tubal matrix \cite{tarzanagh2018fast}]\label{samplingrandom}\rm
	Assume that a random sampling is implemented for choosing $\tau$ lateral slices, one in each of independent and identical distributed (i.i.d.) trials.  A tubal matrix $\mathcal{S} \in\mathbb{K}^{m\times \tau}_{l}$ is called a random sampling tubal matrix, when 
	$\mathcal{S}_{(i,j,1)}=1
	$ if the $i$-th lateral slice is picked in the $j$-th independent trial and $\mathcal{S}_{(i,j,1)}=0$ otherwise, and other frontal slices are all zeros.
\end{definition}

\section{The proposed methods}\label{subTESP}
In this section, we first detail the derivation of the TESP method and its convergence analysis. Then the adaptive variants are presented based on three adaptive sampling strategies, followed by their theoretical guarantees. 

\subsection{TESP method}\label{subsubTESP}
 Similar to the previous works \cite{gower2015randomized,gower2019adaptive, tang2022sketch}, we take the point which is closest to the current iteration $\mathcal{X}^{t}$ and solve a sketched version of the oringinal tensor equation (\ref{tensorequation}) as the next iteration $\mathcal{X}^{t+1}$, that is
$$\mathcal{X}^{t+1}=\mathop{\arg\min}\limits_{\mathcal{X}\in\mathbb{K}_{l}^{r\times s}}\|\mathcal{X}-\mathcal{X}^{t}\|^{2}_{F(\mathcal{M},\mathcal{N})}~~ \text{ s.t. }~~ \mathcal{S}^{T}*\mathcal{A}*\mathcal{X}*\mathcal{B}*\mathcal{V}=\mathcal{S}^{T}*\mathcal{C}*\mathcal{V},$$
where  $\mathcal{S}\in\mathbb{K}^{m\times\tau}_{l}$ and $\mathcal{V}\in\mathbb{K}^{n\times\zeta}_{l}$ with $\tau$ and $\zeta$ being sketch sizes are sketching tubal matrices which are drawn in an i.i.d. fashion from the fixed distributions $\mathfrak{D}_{\mathcal{S}}$  and $\mathfrak{D}_{\mathcal{V}}$, respectively, and $\mathcal{M}\in\mathbb{K}^{r\times r}_{l}$ and $\mathcal{N}\in\mathbb{K}^{s\times s}_{l}$ are T-symmetric T-positive definite tubal matrices. Based on the algebraic properties of t-product, we can get the following update formula of the TESP method:
\begin{align}
	\mathcal{X}^{t+1} =&\mathcal{X}^{t}-\mathcal{M}^{-1}*\mathcal{A}^{T}*\mathcal{S}*(\mathcal{S}^{T}*\mathcal{A}*\mathcal{M}^{-1}*\mathcal{A}^{T}*\mathcal{S})^{\dag}*\mathcal{S}^{T}*(\mathcal{A}*\mathcal{X}^{t}*\mathcal{B}\nonumber\\
	&
 -\mathcal{C})*\mathcal{V}*(\mathcal{V}^{T}*\mathcal{B}^{T}*\mathcal{ N}^{-1}*\mathcal{B}*\mathcal{V})^{\dag}*\mathcal{V}^{T}*\mathcal{B}^{T}*\mathcal{ N}^{-1}. \label{update}
\end{align}
The details of the method are summarized in Algorithm \ref{TESP}.
\begin{algorithm}[htbp]
\caption{TESP method}\label{TESP}
\begin{algorithmic}[1]
\State  \textbf{Input:}
$\mathcal{X}^{0}\in \mathbb{K}^{r\times s}_{l}$, $\mathcal{A}\in \mathbb{K}^{m\times r }_{l}$, $\mathcal{B}\in \mathbb{K}^{s\times n }_{l}$, $\mathcal{C}\in\mathbb{K}_{l}^{m\times n}$  
\State \textbf{Parameters:} 
    fixed distributions $\mathfrak{D}_{\mathcal{S}}$ and $\mathfrak{D}_{\mathcal{V}}$ over random tubal matrices, T-symmetric T-positive definite tubal matrices $\mathcal{M}\in \mathbb{K}^{r\times r}_{l}$ and $\mathcal{N}\in \mathbb{K}^{s\times s}_{l}$
\For{$t=0,1,\cdots$}
\State Sample independent copies $\mathcal{S}\sim \mathfrak{D}_{\mathcal{S}}$ and $\mathcal{V}\sim \mathfrak{D}_{\mathcal{V}}$
\State Compute $\mathcal{E}=\mathcal{S}*(\mathcal{S}^{T}*\mathcal{A}*\mathcal{M}^{-1}*\mathcal{A}^{T}*\mathcal{S})^{\dag}*\mathcal{S}^{T}$ and $\mathcal{G}=\mathcal{V}*(\mathcal{V}^{T}*\mathcal{B}^{T}*\mathcal{N}^{-1}*\mathcal{B}*\mathcal{V})^{\dag}*\mathcal{V}^{T}$
\State $\mathcal{X}^{t+1}=\mathcal{X}^{t}-\mathcal{M}^{-1}*\mathcal{A}^{T}*\mathcal{E}*(\mathcal{A}*\mathcal{X}^{t}*\mathcal{B}-\mathcal{C})*\mathcal{G}*\mathcal{B}^{T}*\mathcal{N}^{-1}$
\EndFor
\State  \textbf{Output:} last iterate $\mathcal{X}^{t+1}$
\end{algorithmic}
\end{algorithm}

\begin{remark}\rm
	The distributions $\mathfrak{D}_{\mathcal{S}}$ and $\mathfrak{D}_{\mathcal{V}}$, and the T-symmetric T-positive definite tubal matrices $\mathcal{M}$ and $\mathcal{N}$ are parameters of the  TESP method. Generally, $\mathfrak{D}_{\mathcal{S}}$ and $\mathfrak{D}_{\mathcal{V}}$ can be any continuous or discrete distributions, and $\mathcal{M}$ and $\mathcal{N}$ can be any T-symmetric T-positive definite tubal  matrices.
	By choosing different parameters, different results will be obtained. The details will be further discussed in Section {\ref{subspecialcase}}.
\end{remark}

Now, we present the convergence of the TESP method. 
\begin{theorem}\label{thmTESP}
	With the notation in Algorithm \ref{TESP}, assume that 
	$\mathbb{E}[\mathcal{Z}]$ and $\mathbb{E}[\mathcal{W}]$ are T-symmetric T-positive definite with probability $1$, where $\mathcal{Z}=\mathcal{M}^{-\frac{1}{2}}*\mathcal{A}^{T}*\mathcal{E}*\mathcal{A}*\mathcal{ M}^{-\frac{1}{2}}$ and $\mathcal{W}=\mathcal{N}^{-\frac{1}{2}}*\mathcal{B}*\mathcal{G}*\mathcal{B}^{T}*\mathcal{N}^{-\frac{1}{2}}$. Let $\mathcal{X}^{\star}$ satisfy $\mathcal{A}*\mathcal{X}^{\star}*\mathcal{B}=\mathcal{C}$. Then the iteration sequence $\{\mathcal{X}^{t}\}_{t=0}^{\infty}$ calculated by the TESP method, i.e., Algorithm \ref{TESP} , with initial iteration $\mathcal{X}^{0}$ satisfies 
	\begin{align*}
		\mathbb{E}\left[\|\mathcal{X}^{t}-\mathcal{X}^{\star}\|_{F(\mathcal{M},\mathcal{N})}^{2}\mid\mathcal{X}^{0}\right]&\leq{\rho_{\text{TESP}}}^{t}\|\mathcal{X}^{0}-\mathcal{X}^{\star}\|_{F(\mathcal{M},\mathcal{ N})}^{2},
	\end{align*}
 where $\rho_{\text{TESP}}=1-\lambda_{\min}(\mathbb{E}[\text{\rm bcirc}(\mathcal{ W}\otimes_{t}\mathcal{ Z})])$.
\end{theorem}
\begin{remark}\rm
It is easy to verify that $\text{bcirc}(\mathcal{ W}\otimes_{t}\mathcal{ Z})$ is an orthogonal projector and hence has eigenvalues $0$ or $1$. Combining Jensen's inequality, as well as the fact that both $A \longmapsto \lambda_{\max}(A)$ and $A \longmapsto -\lambda_{\min}(A)$ are convex on the symmetric matrices, we can get the spectrum of $\mathbb{E}[\text{bcirc}(\mathcal{ W}\otimes_{t}\mathcal{ Z})]$ is contained in $[0,1]$. In addition, since $\mathbb{E}[\mathcal{Z}]$ and $\mathbb{E}[\mathcal{W}]$ are T-symmetric T-positive definite, $\mathbb{E}[\text{bcirc}(\mathcal{ W}\otimes_{t}\mathcal{ Z})]$ is symmetric positive definite, thus $\lambda_{\min}(\mathbb{E}[\text{bcirc}(\mathcal{ W}\otimes_{t}\mathcal{ Z})])> 0$. All together, we have
$$0\leq\rho_{\text{TESP}}=1-\lambda_{\min}(\mathbb{E}[\text{bcirc}(\mathcal{ W}\otimes_{t}\mathcal{ Z})])<1,$$
which implies that the TESP method is convergent in expectation.
\end{remark}

\subsection{The adaptive TESP methods}\label{subsubATESP}
As pointed out in Section \ref{subsubTESP}, for each iteration of the TESP method, two sketching tubal matrices $\mathcal{S}$ and $\mathcal{V}$ need to be selected in an i.i.d. fashion from two pre-given ﬁxed distributions $\mathfrak{D}_{\mathcal{S}}$  and $\mathfrak{D}_{\mathcal{V}}$, respectively. Since the same distributions $\mathfrak{D}_{\mathcal{S}}$  and $\mathfrak{D}_{\mathcal{V}}$ are used in each iteration, this may lead to poor selection of $\mathcal{S}$ and $\mathcal{V}$ in some iterations, resulting in slow convergence. With this in mind, similar to \cite{gower2019adaptive,tang2022sketch}, we will give three adaptive sampling strategies, which utilize the information of the current iteration. It is worth mentioning that we will derive these adaptive sampling strategies on two finite sets of sketching tubal matrices preselected from two distributions. Since how to preselect  these two finite sets is not the focus of this paper, we assume that they have already been selected. 
Specifically, we suppose that $\boldsymbol{\mathcal{S}} = \{ \mathcal{S}_{i} \in \mathbb{K}^{m \times \tau}_{l},~\text{for}~ i=1,\cdots,q_{\mathcal{ S}},~q_{\mathcal{ S}} \in \mathbb{N}\}$ and $\boldsymbol{\mathcal{V}} = \{ \mathcal{V}_{j} \in \mathbb{K}^{n \times \zeta}_{l},~\text{for}~ j=1,\cdots,q_{\mathcal{ V}},~q_{\mathcal{ V}} \in \mathbb{N}\}$ are two finite sets of sketching tubal matrices chosen in advance, then our purpose is to give three adaptive samppling strategies to select $\mathcal{S}=\mathcal{S}_{i}$ and $\mathcal{V}=\mathcal{V}_{j}$ from $\boldsymbol{\mathcal{S}} $ and $\boldsymbol{\mathcal{V}}$, respectively. 

Before giving the adaptive sampling strategies, we first list the nonadaptive TESP (NTESP) method in Algorithm \ref{NTESP}, 
where 
$\Delta_{q}$ with $q\in\mathbb{N}$ is defines as $\Delta_{q}\overset{\text{def}}{=}\{\mathbf{p}=(p_{1},\cdots,p_{q})^{T}\in\mathbb{R}^{q}\mid\sum_{i=1}^{q}p_{i}=1,p_{i}\geq0 \text{ for } i=1,\cdots,q\}$, $i\sim \mathbf{p}$ means that the index $i$ is sampled with the probability $p_{i}$, and $\mathbf{p}_{\mathcal{ S}}$ and $\mathbf{p}_{\mathcal{ V}}$ are two given probability distributions.

\begin{algorithm}[htbp]
\caption{NTESP method}\label{NTESP}
\begin{algorithmic}[1]
\State  \textbf{Input:} $\mathcal{X}^{0}\in \mathbb{K}^{r\times s}_{l}$, $\mathcal{A}\in \mathbb{K}^{m\times r }_{l}$, $\mathcal{B}\in \mathbb{K}^{s\times n }_{l}$, $\mathcal{C}\in\mathbb{K}_{l}^{m\times n}$, $\mathbf{p}_{\mathcal{ S}}\in \Delta_{q_{\mathcal{ S}}}$ and $\mathbf{p}_{\mathcal{ V}}\in \Delta_{q_{\mathcal{ V}}}$
\State  \textbf{Parameters:} two finite sets of sketching tubal matrices $\boldsymbol{\mathcal{S}}=[\mathcal{S}_{1},\cdots,\mathcal{S}_{q_{\mathcal{ S}}}] \in\mathbb{K}_{l}^{m\times q_{\mathcal{S}}\tau}$ and $\boldsymbol{\mathcal{V}}=[\mathcal{V}_{1},\cdots,\mathcal{V}_{q_{\mathcal{V}}}] \in\mathbb{K}_{l}^{n\times q_{\mathcal{ V}}\zeta}$, T-symmetric T-positive definite tubal matrices $\mathcal{M}\in \mathbb{K}^{r\times r}_{l}$ and $\mathcal{N}\in \mathbb{K}^{s\times s}_{l}$
\For{$t=0,1,\cdots$}
\State $i^{t}\sim \mathbf{p}_{\mathcal{ S}}$ and $j^{t}\sim \mathbf{p}_{\mathcal{ V}}$
\State Compute $\mathcal{E}_{i^{t}}=\mathcal{S}_{i^{t}}*(\mathcal{S}_{i^{t}}^{T}*\mathcal{A}*\mathcal{M}^{-1}*\mathcal{A}^{T}*\mathcal{S}_{i^{t}})^{\dag}*\mathcal{S}_{i^{t}}^{T}$ and $\mathcal{G}_{j^{t}}=\mathcal{V}_{j^{t}}*(\mathcal{V}_{j^{t}}^{T}*\mathcal{B}^{T}*\mathcal{N}^{-1}*\mathcal{B}*\mathcal{V}_{j^{t}})^{\dag}*\mathcal{V}_{j^{t}}^{T}$
\State $\mathcal{X}^{t+1}=\mathcal{X}^{t}-\mathcal{M}^{-1}*\mathcal{A}^{T}*\mathcal{E}_{i^{t}}*(\mathcal{A}*\mathcal{X}^{t}*\mathcal{B}-\mathcal{C})*\mathcal{G}_{j^{t}}*\mathcal{B}^{T}*\mathcal{N}^{-1}$
\EndFor
\State  \textbf{Output:} last iterate $\mathcal{X}^{t+1}$
\end{algorithmic}
\end{algorithm}

\subsubsection{Three adaptive sampling strategies and corresponding adaptive methods}\label{secadaptiveTESP}

We first let
 \begin{align}\label{ZandW}
 	\mathcal{Z}_{i^{t}}=\mathcal{M}^{-\frac{1}{2}}*\mathcal{A}^{T}*\mathcal{E}_{i^{t}}*\mathcal{A}*\mathcal{M}^{-\frac{1}{2}} \quad \text{and} \quad \mathcal{W}_{j^{t}}=\mathcal{N}^{-\frac{1}{2}}*\mathcal{B}*\mathcal{G}_{j^{t}}*\mathcal{B}^{T}*\mathcal{N}^{-\frac{1}{2}},
 \end{align}
where $\mathcal{E}_{i^{t}}$ and $\mathcal{G}_{j^{t}}$ are defined in Algorithm \ref{NTESP}. Then, we can verify that $\mathcal{Z}_{i^{t}}$ and $\mathcal{W}_{j^{t}}$ are orthogonal projectors. Thus,
applying the update fomula of the TESP method and the fact that $\mathcal{ A}*\mathcal{X}^{\star}*\mathcal{ B}=\mathcal{C}$, we have
\begin{align}
&\|\mathcal{X}^{t+1}-\mathcal{X}^{\star}\|_{F(\mathcal{M},\mathcal{ N})}^{2}=\|\mathcal{M}^{\frac{1}{2}}*(\mathcal{X}^{t+1}-\mathcal{X}^{\star})*\mathcal{N}^{\frac{1}{2}}\|_{F}^{2}\nonumber \\
=&\|\mathcal{M}^{\frac{1}{2}}*(\mathcal{X}^{t}-\mathcal{X}^{\star})*\mathcal{N}^{\frac{1}{2}}-\mathcal{Z}_{i^{t}}*\mathcal{M}^{\frac{1}{2}}*(\mathcal{X}^{t}-\mathcal{X}^{\star})*\mathcal{N}^{\frac{1}{2}}*\mathcal{W}_{j^{t}}\|_{F}^{2}\nonumber \\
=&\|\mathcal{M}^{\frac{1}{2}}*(\mathcal{X}^{t}-\mathcal{X}^{\star})*\mathcal{N}^{\frac{1}{2}}\|_{F}^{2}-\|\mathcal{Z}_{i^{t}}*\mathcal{M}^{\frac{1}{2}}*(\mathcal{X}^{t}-\mathcal{X}^{\star})*\mathcal{N}^{\frac{1}{2}}*\mathcal{W}_{j^{t}}\|_{F}^{2}\nonumber \\
=&\|\mathcal{X}^{t}-\mathcal{X}^{\star}\|_{F(\mathcal{M},\mathcal{ N})}^{2}-\|\mathcal{Z}_{i^{t}}*\mathcal{M}^{\frac{1}{2}}*(\mathcal{X}^{t}-\mathcal{X}^{\star})*\mathcal{N}^{\frac{1}{2}}*\mathcal{W}_{j^{t}}\|_{F}^{2}, \label{secloss}
\end{align}
where the third equality follows from Lemma \ref{sub1lem4}. Hence, we can conclude that the magnitude of $	\|\mathcal{X}^{t+1}-\mathcal{X}^{\star}\|_{F(\mathcal{M},\mathcal{ N})}^{2}$ is determined by $f_{i^{t},j^{t}}(\mathcal{X}^{t})\overset{\text{def}}=\|\mathcal{Z}_{i^{t}}*\mathcal{M}^{\frac{1}{2}}*(\mathcal{X}^{t}-\mathcal{X}^{\star})*\mathcal{N}^{\frac{1}{2}}*\mathcal{W}_{j^{t}}\|_{F}^{2}$. Consequently, in order to make the most progress in one iteration, we should pick the index pair $(i^{t},j^{t})$ corresponding to the largest sketched loss $f_{i^{t},j^{t}}(\mathcal{X}^{t})$.
Since $\mathcal{X}^{\star}$ is unknown in practice, we rewrite $f_{i^{t},j^{t}}(\mathcal{X}^{t})$ as
\begin{align*}
f_{i^{t},j^{t}}(\mathcal{X}^{t})
=&\|\text{vec}_{t}(\mathcal{Z}_{i^{t}}*\mathcal{M}^{\frac{1}{2}}*(\mathcal{X}^{t}-\mathcal{X}^{\star})*\mathcal{N}^{\frac{1}{2}}*\mathcal{W}_{j^{t}})\|_{F}^{2}\\
=&\Big\langle((\mathcal{ W}_{j^{t}}^{ST}*(\mathcal{N}^{\frac{1}{2}})^{ST})\otimes_{t}(\mathcal{ Z}_{i^{t}}*\mathcal{M}^{\frac{1}{2}}))^{T}*((\mathcal{ W}_{j^{t}}^{ST}*(\mathcal{N}^{\frac{1}{2}})^{ST})\\
&\otimes_{t}(\mathcal{ Z}_{i^{t}}*\mathcal{M}^{\frac{1}{2}}))*\text{vec}_{t}(\mathcal{X}^{t}-\mathcal{X}^{\star}),\text{vec}_{t}(\mathcal{X}^{t}-\mathcal{X}^{\star})\Big\rangle\\
=&\Big\langle(((\mathcal{N}^{\frac{1}{2}})^{ST}*\mathcal{ W}_{j^{t}}^{ST}*(\mathcal{N}^{\frac{1}{2}})^{ST})\otimes_{t}(\mathcal{M}^{\frac{1}{2}}*\mathcal{ Z}_{i^{t}}*\mathcal{M}^{\frac{1}{2}}))
\\
&*\text{vec}_{t}(\mathcal{X}^{t}-\mathcal{X}^{\star}),\text{vec}_{t}(\mathcal{X}^{t}-\mathcal{X}^{\star})\Big\rangle\\
=&\big\langle((\mathcal{B}^{C}*\mathcal{G}_{j^{t}}^{ST}*\mathcal{B}^{ST})\otimes_{t}(\mathcal{A}^{T}*\mathcal{E}_{i^{t}}*\mathcal{A}))*\text{vec}_{t}(\mathcal{X}^{t}-\mathcal{X}^{\star}),
\\
&\text{vec}_{t}(\mathcal{X}^{t}-\mathcal{X}^{\star})\big\rangle\\
=&\left\langle((\mathcal{G}_{j^{t}}^{ST}\otimes_{t}\mathcal{E}_{i^{t}})*\text{vec}_{t}(\mathcal{A}*(\mathcal{X}^{t}-\mathcal{X}^{\star})*\mathcal{B}),\text{vec}_{t}(\mathcal{A}*(\mathcal{X}^{t}-\mathcal{X}^{\star})*\mathcal{B})\right\rangle\\
=&\|\text{vec}_{t}(\mathcal{A}*\mathcal{X}^{t}*\mathcal{B}-\mathcal{C})\|_{\mathcal{G}_{j^{t}}^{ST}\otimes_{t}\mathcal{E}_{i^{t}}}^{2}=\|\mathcal{A}*\mathcal{X}^{t}*\mathcal{B}-\mathcal{C}\|_{F(\mathcal{E}_{i^{t}},\mathcal{G}_{j^{t}})}^{2},
\end{align*}
where the fourth equality is from (\ref{ZandW}). Thus, based on the above analysis, we can propose the ﬁrst adaptive sampling strategy as follows:
\begin{align}
	(i^{t},j^{t})=\mathop{\arg\max}_{i\in[q_{\mathcal{S}}], j\in[q_{\mathcal{V}}]}f_{i,j}(\mathcal{X}^{t})=\mathop{\arg\max}_{i\in[q_{\mathcal{S}}], j\in[q_{\mathcal{V}}]}\|\mathcal{A}*\mathcal{X}^{t}*\mathcal{B}-\mathcal{ C}\|_{F(\mathcal{E}_{i},\mathcal{G}_{j})}^{2},
\end{align}
which is called the max-distance selection rule. The corresponding adaptive TESP method is referred as ATESP-MD method for short, where A and MD stand for adaptive and max-distance, respectively, and the algorithm is summarized as the case 1 of Algorithm \ref{ATESP}.

\begin{algorithm}[htbp]
\caption{ATESP-(MD/PR/CS) method}\label{ATESP}
\begin{algorithmic}[1]
 \State  \textbf{Input:} $\mathcal{X}^{0}\in \mathbb{K}^{r\times s}_{l}$, $\mathcal{A}\in \mathbb{K}^{m\times r }_{l}$, $\mathcal{B}\in \mathbb{K}^{s\times n }_{l}$, $\mathcal{C}\in\mathbb{K}_{l}^{m\times n}$, $\mathbf{p}_{\mathcal{ S}}\in \Delta_{q_{\mathcal{ S}}}$, $\mathbf{p}_{\mathcal{ V}}\in \Delta_{q_{\mathcal{ V}}}$ and $\theta\in [0,1]$
\State  \textbf{Parameters:} two finite sets of sketching tubal matrices $\boldsymbol{\mathcal{S}}=[\mathcal{S}_{1},\cdots,\mathcal{S}_{q_{\mathcal{ S}}}] \in\mathbb{K}_{l}^{m\times q_{\mathcal{ S}}\tau}$ and $\boldsymbol{\mathcal{V}}=[\mathcal{V}_{1},\cdots,\mathcal{V}_{q_{\mathcal{ V}}}] \in\mathbb{K}_{l}^{n\times q_{\mathcal{ V}}\zeta}$, T-symmetric T-positive definite tubal matrices $\mathcal{M}\in \mathbb{K}^{r\times r}_{l}$ and $\mathcal{N}\in \mathbb{K}^{s\times s}_{l}$
		\For{$t=0,1,\cdots$}
		\State $f_{i,j}(\mathcal{X}^{t})=\|\mathcal{A}*\mathcal{X}^{t}*\mathcal{B}-\mathcal{ C}\|_{F(\mathcal{E}_{i},\mathcal{G}_{j})}^{2}$ for $i=1,\cdots,q_{\mathcal{ S}}$ and $j=1,\cdots,q_{\mathcal{ V}}$
		\State \textbf{switch} \emph{adaptive sampling strategies}
	   \State $\triangleright$	\texttt{switch}  means that each method corresponds to one case rather than choosing a different case for each iteration
	   \State \quad \quad \textbf{case 1} ~Max-distance selection rule (MD)
	   
	   \State \quad \quad$(i^{t},j^{t})=\mathop{\arg\max}\limits_{i\in[q_{\mathcal{ A}}], j\in[q_{\mathcal{B}}]}f_{i,j}(\mathcal{X}^{t})$
	   
	   \State\quad\quad  \textbf{case 2} ~Adaptive probabilities rule (PR)
	   
	  \State\quad\quad Calculate $\mathbf{p}_{\mathcal{ S},\mathcal{ V}}^{t}\in \Delta_{q_{\mathcal{ S}}q_{\mathcal{ V}}}$ such that, for $i=1,\cdots,q_{\mathcal{ S}}$ and $j=1,\cdots,q_{\mathcal{ V}}$, $p_{i,j}^{t}=f_{i,j}(\mathcal{X}^{t})/(\sum_{i=1}^{q_{\mathcal{ S}}}\sum_{j=1}^{q_{\mathcal{ V}}}f_{i,j}(\mathcal{X}^{t}))$ 
	   
	   \State \quad\quad$(i^{t},j^{t})\sim \mathbf{p}_{\mathcal{ S},\mathcal{ V}}^{t}$
	   
	   \State\quad\quad  \textbf{case 3} ~Capped sampling rule (CS)
	   
	  \State \quad\quad Determine the index set $\mathfrak{W}_{t}$, which is defined in (\ref{indexset})
	   
	   \State\quad\quad Calculate $\mathbf{p}_{\mathcal{ S},\mathcal{ V}}^{t}\in \Delta_{q_{\mathcal{ S}}q_{\mathcal{ V}}}$, which is defined in (\ref{cappedprob})
	   
	   \State\quad\quad $(i^{t},j^{t})\sim \mathbf{p}_{\mathcal{ S},\mathcal{ V}}^{t}$
	   \State  \textbf{end switch}
		\State Compute $\mathcal{E}_{i^{t}}=\mathcal{S}_{i^{t}}*(\mathcal{S}_{i^{t}}^{T}*\mathcal{A}*\mathcal{M}^{-1}*\mathcal{A}^{T}*\mathcal{S}_{i^{t}})^{\dag}*\mathcal{S}_{i^{t}}^{T}$ and $\mathcal{G}_{j^{t}}=\mathcal{V}_{j^{t}}*(\mathcal{V}_{j^{t}}^{T}*\mathcal{B}^{T}*\mathcal{N}^{-1}*\mathcal{B}*\mathcal{V}_{j^{t}})^{\dag}*\mathcal{V}_{j^{t}}^{T}$
		
		\State $\mathcal{X}^{t+1}=\mathcal{X}^{t}-\mathcal{M}^{-1}*\mathcal{A}^{T}*\mathcal{E}_{i^{t}}*(\mathcal{A}*\mathcal{X}^{t}*\mathcal{B}-\mathcal{C})*\mathcal{G}_{j^{t}}*\mathcal{B}^{T}*\mathcal{N}^{-1}$
		\EndFor
  \State  \textbf{Output:} last iterate $\mathcal{X}^{t+1}$
	\end{algorithmic}
\end{algorithm}

Next, we consider the expected decrease of  $\|\mathcal{X}^{t+1}-\mathcal{X}^{\star}\|_{F(\mathcal{M},\mathcal{ N})}^{2}$. Let $\mathbf{p}_{\mathcal{S},\mathcal{ V}}^{t}\in \Delta_{q_{\mathcal{ S}}q_{\mathcal{V}}}$ and $(i^{t},j^{t})\sim \mathbf{p}_{\mathcal{ S},\mathcal{ V}}^{t}$, where $\mathbf{p}_{\mathcal{ S},\mathcal{ V}}^{t}\overset{\text{def}}{=}(p^{t}_{1,1},\cdots,p^{t}_{q_{\mathcal{ S}},q_{\mathcal{ V}}})^{T}$ with $p_{i,j}^{t}=\mathbb{P}[\mathcal{S}_{i^{t}}=\mathcal{S}_{i},\mathcal{V}_{j^{t}}=\mathcal{V}_{j}\mid\mathcal{X}^{t}]$ 
for $i=1,\cdots,q_{\mathcal{ S}}$ and $j=1,\cdots,q_{\mathcal{ V}}$, i.e., $p_{i,j}^{t}$ is the probability of $\mathcal{S}_{i}$ and $\mathcal{V}_{j}$ being sampled at the $t$-th iteration. Thus, taking expectation conditioned on $\mathcal{X}^{t}$ in (\ref{secloss}), we have
\begin{align}
	\mathbb{E}[\|\mathcal{X}^{t+1}-\mathcal{X}^{\star}\|_{F(\mathcal{M},\mathcal{ N})}^{2}\mid\mathcal{X}^{t}]&=\|\mathcal{X}^{t}-\mathcal{X}^{\star}\|_{F(\mathcal{M},\mathcal{ N})}^{2}-\mathbb{E}_{(i,j)\sim \mathbf{p}_{\mathcal{S},\mathcal{ V}}^{t}}[f_{i,j}(\mathcal{X}^{t})]\label{expectation}\\
	&=\|\mathcal{X}^{t}-\mathcal{X}^{\star}\|_{F(\mathcal{M},\mathcal{ N})}^{2}-\sum_{i=1}^{q_{\mathcal{S}}}\sum_{j=1}^{q_{\mathcal{V}}}p_{i,j}^{t}f_{i,j}(\mathcal{X}^{t}),\nonumber
\end{align}
which implies that if we want $\mathbb{E}[\|\mathcal{X}^{t+1}-\mathcal{X}^{\star}\|_{F(\mathcal{M},\mathcal{ N})}^{2}\mid\mathcal{X}^{t}]$ to be as small as possible, we should sample the index pairs corresponding to larger sketched losses with higher probabilities. An intuitive way is to set the sampling probabilities proportional to the sketched losses and we call such strategy adaptive probabilities rule. The  corresponding adaptive TESP method is referred as ATESP-PR method for short, where PR stands for probabilities, and the algorithm is summarized as the case 2 of Algorithm \ref{ATESP}.

In addition, there is another commonly used way to define the sampling probability. The idea is to avoid sampling the index pairs corresponding to the smaller sketched losses by removing them, so that the probabilities of the index pairs corresponding to larger sketched losses being selected will increase. 
Specifically, we first define an index pair set 
\begin{align}\label{indexset}
	&\mathfrak{W}_{t}\overset{\text{\rm def}}{=}\nonumber\\
 &\left\{(i,j)\mid f_{i,j}(\mathcal{X}^{t})\geq\theta\mathop{\max}_{v\in[q_{\mathcal{ S}}], w\in[q_{\mathcal{V}}]}f_{v,w}(\mathcal{X}^{t})+(1-\theta)\mathbb{E}_{v\sim\mathbf{p}_{\mathcal{ S}},w\sim \mathbf{p}_{\mathcal{ V}}}[f_{v,w}(\mathcal{X}^{t})]\right\},
\end{align}
where $\mathbf{p}_{\mathcal{ S}}\in \Delta_{q_{\mathcal{ S}}}$, $\mathbf{p}_{\mathcal{ V}}\in \Delta_{q_{\mathcal{ V}}}$ and $\theta\in [0,1]$. Then, we set the probability $\mathbf{p}_{\mathcal{ S},\mathcal{ V}}^{t}\in \Delta_{q_{\mathcal{ S}}q_{\mathcal{V}}}$ such that 
\begin{equation}\label{cappedprob}
	p_{i,j}^{t}=\left\{
	\begin{array}{lcl}
		\frac{f_{i,j}(\mathcal{X}^{t})}{\sum\limits_{(i,j)\in \mathfrak{W}_{t}}f_{i,j}(\mathcal{X}^{t})},& & (i,j)\in \mathfrak{W}_{t},\\
		0, & & (i,j)\notin \mathfrak{W}_{t}.
	\end{array} \right.
\end{equation}
We call this strategy the capped sampling rule, and the corresponding adaptive TESP method is referred as ATESP-CS method for short, where CS stands for capped sampling, and the algorithm is summarized as the case 3 of Algorithm \ref{ATESP}.

\subsubsection{Convergence analysis}
We now present the convergence results for the above proposed nonadaptive and adaptive TESP methods, i.e., the NTESP, ATESP-MD, ATESP-PR and  ATESP-CS methods. Before that, we first give two lemmas which are crucial to the convergence analysis in the following theorems.

\begin{lemma}\label{defdelta}
	With the notation in the NTESP, ATESP-MD, ATESP-PR, and ATESP-CS methods, let $\mathbf{p}_{\mathcal{ S}}\in \Delta_{q_{\mathcal{ S}}}$, $\mathbf{p}_{\mathcal{ V}}\in \Delta_{q_{\mathcal{ V}}}$  and define
	\begin{align}
		\delta_{\infty}^{2}(\mathcal{M},\mathcal{ N},\boldsymbol{\mathcal{S}},\boldsymbol{\mathcal{V}})\overset{\text{\rm def}}{=}&\mathop{\min}_{\overrightarrow{\mathcal{Y}}\in \mathbf{Range}(((\mathcal{N}^{-1})^{ST}*\mathcal{ B}^{C})\otimes_{t}(\mathcal{M}^{-1}*\mathcal{A}^{T}))}\mathop{\max}_{i\in[q_{\mathcal{S}}], j\in[q_{\mathcal{ V}}]}\nonumber\\
		&\frac{\|((\mathcal{N}^{\frac{1}{2}})^{ST}\otimes_{t} \mathcal{M}^{\frac{1}{2}})*\overrightarrow{\mathcal{Y}}\|_{\mathcal{ W}_{j}^{ST}\otimes_{t}\mathcal{ Z}_{i}}^{2}}{\|\overrightarrow{\mathcal{Y}}\|_{\mathcal{N}^{ST}\otimes_{t} \mathcal{M}}^{2}},\label{deinf}
	\end{align}
   \begin{align}
		\delta_{\mathbf{p}_{\mathcal{ S}}, \mathbf{p}_{\mathcal{ V}}}^{2}(\mathcal{M},\mathcal{ N},\boldsymbol{\mathcal{S}},\boldsymbol{\mathcal{V}})\overset{\text{\rm def}}{=}&\mathop{\min}_{\overrightarrow{\mathcal{Y}}\in \mathbf{Range}(((\mathcal{N}^{-1})^{ST}*\mathcal{ B}^{C})\otimes_{t}(\mathcal{M}^{-1}*\mathcal{A}^{T}))}\nonumber\\
		&\frac{\|((\mathcal{N}^{\frac{1}{2}})^{ST}\otimes_{t} \mathcal{M}^{\frac{1}{2}})*\overrightarrow{\mathcal{Y}}\|_{\mathbb{E}_{j\sim \mathbf{p}_{\mathcal{ V}}}[\mathcal{ W}_{j}^{ST}] \otimes_{t} \mathbb{E}_{i\sim \mathbf{p}_{\mathcal{ S}}}[\mathcal{ Z}_{i}]}^{2}}{\|\overrightarrow{\mathcal{Y}}\|_{\mathcal{N}^{ST}\otimes_{t} \mathcal{M}}^{2}},\label{deexpe}
	\end{align}
   where $\mathcal{ Z}_{i}$ and $\mathcal{W}_{j}$ are the same as $\mathcal{ Z}_{i^{t}}$ and $\mathcal{W}_{j^{t}}$ deﬁned in (\ref{ZandW}) except that $i^{t}$ and $j^{t}$ are replaced by $i$ and $j$, respectively. Let
	 $\mathcal{X}^{\star}$ satisfy $\mathcal{ A}*\mathcal{X}^{\star}*\mathcal{ B}=\mathcal{ C}$. Then the iteration sequence $\{\mathcal{X}^{t}\}_{t=0}^{\infty}$ calculated by any nonadaptive and adaptive TESP methods with initial iteration $\text{vec}_{t}(\mathcal{X}^{0})\in \mathbf{Range}((\mathcal{N}^{-1})^{ST}*\mathcal{ B}^{C}\otimes_{t}\mathcal{M}^{-1}*\mathcal{A}^{T})$ satisfies
	\begin{align}
		\mathop{\max}\limits_{i\in[q_{\mathcal{ S}}], j\in[q_{\mathcal{ V}}]}f_{i,j}(\mathcal{X}^{t})\geq 		\delta_{\infty}^{2}(\mathcal{M},\mathcal{ N},\boldsymbol{\mathcal{S}},\boldsymbol{\mathcal{V}})\|\mathcal{X}^{t}-\mathcal{X}^{\star}\|_{F(\mathcal{M},\mathcal{N})}^{2},\label{dein}\\
		\mathbb{E}_{i\sim \mathbf{p}_{\mathcal{ S}},j\sim \mathbf{p}_{\mathcal{ V}}}[f_{i,j}(\mathcal{X}^{t})]\geq 	\delta_{\mathbf{p}_{\mathcal{ S}}, \mathbf{p}_{\mathcal{ V}}}^{2}(\mathcal{M},\mathcal{ N},\boldsymbol{\mathcal{S}},\boldsymbol{\mathcal{V}})\|\mathcal{X}^{t}-\mathcal{X}^{\star}\|_{F(\mathcal{M},\mathcal{N})}^{2}.\label{deex}
	\end{align}
\end{lemma}

\begin{lemma}\label{sizedelta}
	Let $\mathbf{p}_{\mathcal{ S}}\in \Delta_{q_{\mathcal{ S}}}$ and $\mathbf{p}_{\mathcal{ V}}\in \Delta_{q_{\mathcal{ V}}}$. Assume that the finite sets of sketching tubal matrices $\boldsymbol{\mathcal{S}}=[\mathcal{S}_{1},\cdots,\mathcal{S}_{q_{\mathcal{ S}}}]$ and $\boldsymbol{\mathcal{V}}=[\mathcal{V}_{1},\cdots,\mathcal{V}_{q_{\mathcal{ V}}}]$ respectively satisfy that $\mathbb{E}_{i\sim \mathbf{p}_{\mathcal{ S}}}[\mathcal{ Z}_{i}]$ and $\mathbb{E}_{j\sim \mathbf{p}_{\mathcal{ V}}}[\mathcal{ W}_{j}]$ are T-symmetric T-positive definite with probability $1$. Then 
	\begin{footnotesize}
	\begin{align*}
		0<\lambda_{\min}\left(\mathbb{E}_{i\sim \mathbf{p}_{\mathcal{ S}},j\sim \mathbf{p}_{\mathcal{V}}}[\text{\rm bcirc}(\mathcal{ W}_{j} \otimes_{t}\mathcal{ Z}_{i})]\right)
		=\delta_{\mathbf{p}_{\mathcal{ S}}, \mathbf{p}_{\mathcal{ V}}}^{2}(\mathcal{M},\mathcal{ N},\boldsymbol{\mathcal{S}},\boldsymbol{\mathcal{V}})\leq\delta_{\infty}^{2}(\mathcal{M},\mathcal{ N},\boldsymbol{\mathcal{S}},\boldsymbol{\mathcal{V}})\leq1.
	\end{align*}
	\end{footnotesize}
\end{lemma}

Next, we present the convergence guarantees of the NTESP, ATESP-MD, ATESP-PR, and ATESP-CS methods in turn.
	\begin{theorem}\label{thmNTESP}
	Let $\mathcal{X}^{\star}$ satisfy $\mathcal{ A}*\mathcal{X}^{\star}*\mathcal{B}=\mathcal{C}$. Then the iteration sequence $\{\mathcal{X}^{t}\}_{t=0}^{\infty}$ calculated by the NTESP method, i.e., Algorithm \ref{NTESP}, with initial iteration $\text{vec}_{t}(\mathcal{ X}^{0})\in \mathbf{Range}(((\mathcal{N}^{-1})^{ST}*\mathcal{ B}^{C})\otimes_{t}( \mathcal{ M}^{-1}*\mathcal{A}^{T}))$ satisfies 
	\begin{align*}
		&\mathbb{E}[\|\mathcal{X}^{t}-\mathcal{X}^{\star}\|_{F(\mathcal{M},\mathcal{ N})}^{2}\mid\mathcal{X}^{0}]\leq{\rho_{\text{NTESP}}}^{t}\|\mathcal{X}^{0}-\mathcal{X}^{\star}\|_{F(\mathcal{M},\mathcal{ N})}^{2},
	\end{align*}
	where $\rho_{\text{NTESP}}=1-\delta_{\mathbf{p}_{\mathcal{ S}}, \mathbf{p}_{\mathcal{ V}}}^{2}(\mathcal{M},\mathcal{ N},\boldsymbol{\mathcal{S}},\boldsymbol{\mathcal{V}})=1-\lambda_{\min}\left(\mathbb{E}_{i\sim \mathbf{p}_{\mathcal{ S}},j\sim \mathbf{p}_{\mathcal{V}}}[\text{\rm bcirc}(\mathcal{ W}_{j} \otimes_{t}\mathcal{ Z}_{i})]\right)$ and $\delta_{\mathbf{p}_{\mathcal{ S}}, \mathbf{p}_{\mathcal{ V}}}^{2}(\mathcal{M},\mathcal{ N},\boldsymbol{\mathcal{S}},\boldsymbol{\mathcal{V}})$ is as deﬁned in (\ref{deexpe}).
\end{theorem}
\begin{remark}\rm
	The conclusion in  Theorem \ref{thmNTESP} is in line with that in Theorem \ref{thmTESP}. This is because the  probability distributions used in the NTESP method can be regarded as the special cases of the ones in TESP method.  
\end{remark}
	\begin{theorem}\label{thmATESPMD}
	Let $\mathcal{ X}^{\star}$ satisfy $\mathcal{ A}*\mathcal{ X}^{\star}*\mathcal{ B}=\mathcal{ C}$. Then the iteration sequence $\{\mathcal{X}^{t}\}_{t=0}^{\infty}$ calculated by the ATESP-MD method, i.e., the first case of Algorithm \ref{ATESP}, with initial iteration $\text{vec}_{t}(\mathcal{ X}^{0})\in  \mathbf{Range}(((\mathcal{N}^{-1})^{ST}*\mathcal{ B}^{C})\otimes_{t}( \mathcal{ M}^{-1}*\mathcal{A}^{T}))$ satisfies
	$$\mathbb{E}[\|\mathcal{ X}^{t}-\mathcal{ X}^{\star}\|_{F(\mathcal{M},\mathcal{ N})}^{2}\mid\mathcal{ X}^{0}]\leq{\rho_{\text{ATESP-MD}}}^{t}\|\mathcal{ X}^{0}-\mathcal{ X}^{\star}\|_{F(\mathcal{M},\mathcal{ N})}^{2}.$$
	where $\rho_{\text{ATESP-MD}}=1- \delta_{\infty}^{2}(\mathcal{M},\mathcal{ N},\boldsymbol{\mathcal{S}},\boldsymbol{\mathcal{V}})$ and $\delta_{\infty}^{2}(\mathcal{M},\mathcal{ N},\boldsymbol{\mathcal{S}},\boldsymbol{\mathcal{V}})$ is as deﬁned in (\ref{deinf}).
\end{theorem}
\begin{remark}\rm
	According to Lemma \ref{sizedelta}, we have that the convergence factor of the ATESP-MD method is smaller than that of the NTESP method. That is,
	$$\rho_{\text{ATESP-MD}}=1- \delta_{\infty}^{2}(\mathcal{M},\mathcal{ N},\boldsymbol{\mathcal{S}},\boldsymbol{\mathcal{V}})\leq1-\delta_{\mathbf{p}_{\mathcal{ S}}, \mathbf{p}_{\mathcal{ V}}}^{2}(\mathcal{M},\mathcal{ N},\boldsymbol{\mathcal{S}},\boldsymbol{\mathcal{V}})=\rho_{\text{NTESP}}.$$
\end{remark}
\begin{theorem}\label{thmATESPRP}
	Let $\mathbf{u}_{\mathcal{ S}}=\left(\frac{1}{q_{\mathcal{ S}}},\cdots,\frac{1}{q_{\mathcal{ S}}}\right)^{T}\in \Delta_{q_{\mathcal{ S}}}$ and $\mathbf{u}_{\mathcal{ V}}=\left(\frac{1}{q_{\mathcal{ V}}},\cdots,\frac{1}{q_{\mathcal{ V}}}\right)^{T}\in \Delta_{q_{\mathcal{ V}}}$. 
	Let $\mathcal{ X}^{\star}$ satisfy $\mathcal{ A}*\mathcal{ X}^{\star}*\mathcal{ B}=\mathcal{ C}$. Then the iteration sequence $\{\mathcal{X}^{t}\}_{t=1}^{\infty}$ calculated by the ATESP-PR method, i.e., the second case of Algorithm \ref{ATESP}, with initial iteration $\text{vec}_{t}(\mathcal{ X}^{0})\in  \mathbf{Range}(((\mathcal{N}^{-1})^{ST}*\mathcal{ B}^{C})\otimes_{t}( \mathcal{ M}^{-1}*\mathcal{A}^{T}))$  satisfies
\begin{align*}	 
  \mathbb{E}\left[\|\mathcal{ X}^{t+1}-\mathcal{ X}^{\star}\|_{F(\mathcal{ M},\mathcal{ N})}^{2}\mid\mathcal{ X}^{t}\right]&\leq\rho_{\text{ATESP-PR}}\|\mathcal{ X}^{t}-\mathcal{ X}^{\star}\|_{F(\mathcal{ M},\mathcal{ N})}^{2},
 \end{align*} 
where $\rho_{\text{ATESP-PR}}=1-(1+q_{\mathcal{ S}}^{2}q_{\mathcal{ V}}^{2}\mathbf{Var}_{i\sim \mathbf{u}_{\mathcal{ S}}, j\sim \mathbf{u}_{\mathcal{ V}}}[p_{i,j}^{t}])\delta_{\mathbf{u}_{\mathcal{ S}}, \mathbf{u}_{\mathcal{ V}}}^{2}(\mathcal{M},\mathcal{ N},\boldsymbol{\mathcal{S}},\boldsymbol{\mathcal{V}})$ and  $\mathbf{Var}_{i\sim \mathbf{u}_{\mathcal{ S}}, j\sim \mathbf{u}_{\mathcal{ V}}}[\cdot]$ denotes the variance taken with respect to the uniform distributions $ \mathbf{u}_{\mathcal{ A}}$ and $\mathbf{u}_{\mathcal{ B}}$, i.e.,
$$\mathbf{Var}_{i \sim \mathbf{u}_{\mathcal{ S}}, j\sim \mathbf{u}_{\mathcal{ V}}}[v_{i,j}]=\frac{1}{q_{\mathcal{ S}}q_{\mathcal{ V}}}\sum_{i=1}^{q_{\mathcal{ S}}}\sum_{j=1}^{q_{\mathcal{ V}}}\left(v_{i,j}-\frac{1}{q_{\mathcal{ S}}q_{\mathcal{ V}}}\sum_{s=1}^{q_{\mathcal{ S}}}\sum_{r=1}^{q_{\mathcal{ V}}}v_{s,r}\right)^{2},~~~\forall ~\mathbf{v}\in \mathbb{R}^{q_{\mathcal{ S}}q_{\mathcal{ V}}}.$$
Furthermore, 
\begin{align*}
\mathbb{E}\left[\|\mathcal{ X}^{t+1}-\mathcal{ X}^{\star}\|_{F(\mathcal{ M},\mathcal{ N})}^{2}\mid\mathcal{ X}^{1}\right]&\leq\left(\prod_{d=1}^{t}\rho_{d}\right)\mathbb{E}\left[\|\mathcal{ X}^{1}-\mathcal{ X}^{\star}\|_{F(\mathcal{ M},\mathcal{ N})}^{2}\mid\mathcal{ X}^{0}\right],
\end{align*}
where $\rho_{d}=1-\left(1+ \frac{\vert\Omega_{d}\vert}{q_{\mathcal{ S}}q_{\mathcal{ V}}}\right)\delta_{\mathbf{u}_{\mathcal{ S}}, \mathbf{u}_{\mathcal{ V}}}^{2}(\mathcal{M},\mathcal{ N},\boldsymbol{\mathcal{S}},\boldsymbol{\mathcal{V}})$, $\delta_{\mathbf{u}_{\mathcal{ S}}, \mathbf{u}_{\mathcal{ V}}}^{2}(\mathcal{M},\mathcal{ N},\boldsymbol{\mathcal{S}},\boldsymbol{\mathcal{V}})$ is defined as in (\ref{deexpe}) and $\Omega_{d}=\{(i,j)\vert f_{i,j}(\mathcal{ X}^{d})=0, i\in [q_{\mathcal{ S}}], j\in [q_{\mathcal{ V}}]\}$ with  $\vert\Omega_{d}\vert$ denoting its cardinality for $d=1,\cdots,t$.
\end{theorem}
\begin{remark}\rm
For the case of $t\geq 1$, the set $\Omega_{t}$ is not empty, i.e., $\vert\Omega_{t}\vert\geq 1$. This is because 
\begin{align*}
&\mathcal{Z}_{i^{t-1}}*\mathcal{M}^{\frac{1}{2}}*(\mathcal{X}^{t}-\mathcal{X}^{\star})*\mathcal{N}^{\frac{1}{2}}*\mathcal{W}_{j^{t-1}}\\
=&\mathcal{Z}_{i^{t-1}}*\mathcal{M}^{\frac{1}{2}}*(\mathcal{X}^{t-1}-\mathcal{M}^{-\frac{1}{2}}*\mathcal{Z}_{i^{t-1}}*\mathcal{M}^{\frac{1}{2}}*(\mathcal{X}^{t-1}-\mathcal{X}^{\star})*\mathcal{N}^{\frac{1}{2}}*\mathcal{W}_{j^{t-1}}*\mathcal{N}^{-\frac{1}{2}}\\
&-\mathcal{X}^{\star})*\mathcal{N}^{\frac{1}{2}}*\mathcal{W}_{j^{t-1}}\\
	=&\mathcal{Z}_{i^{t-1}}*\mathcal{M}^{\frac{1}{2}}*(\mathcal{X}^{t-1}-\mathcal{X}^{\star})*\mathcal{N}^{\frac{1}{2}}*\mathcal{W}_{j^{t-1}}-\mathcal{Z}_{i^{t-1}}*\mathcal{M}^{\frac{1}{2}}*(\mathcal{X}^{t-1}-\mathcal{X}^{\star})*\mathcal{N}^{\frac{1}{2}}
 \\&*\mathcal{W}_{j^{t-1}}=0,
\end{align*}
leads to
\begin{align*}
	&f_{i^{t-1},j^{t-1}}(\mathcal{X}^{t})=\|((\mathcal{N}^{\frac{1}{2}})^{ST}\otimes_{t} \mathcal{M}^{\frac{1}{2}})*\text{vec}_{t}(\mathcal{X}^{t}-\mathcal{X}^{\star})\|_{\mathcal{ W}_{j^{t-1}}^{ST}\otimes_{t}\mathcal{ Z}_{i^{t-1}}}^{2}\\
	=&\Big\langle(\mathcal{ W}_{j^{t-1}}^{ST}\otimes_{t}\mathcal{ Z}_{i^{t-1}})*\text{vec}_{t}(\mathcal{M}^{\frac{1}{2}}*(\mathcal{X}^{t}-\mathcal{X}^{\star})*\mathcal{N}^{\frac{1}{2}}),\text{vec}_{t}(\mathcal{M}^{\frac{1}{2}}*(\mathcal{X}^{t}-\mathcal{X}^{\star})*\mathcal{N}^{\frac{1}{2}})\Big\rangle\\
	=&\Big\langle\text{vec}_{t}(\mathcal{ Z}_{i^{t-1}}*\mathcal{M}^{\frac{1}{2}}*(\mathcal{X}^{t}-\mathcal{X}^{\star})*\mathcal{N}^{\frac{1}{2}}*\mathcal{ W}_{j^{t-1}}),\text{vec}_{t}(\mathcal{M}^{\frac{1}{2}}*(\mathcal{X}^{t}-\mathcal{X}^{\star})*\mathcal{N}^{\frac{1}{2}})\Big\rangle\\
 =&0,
\end{align*}
which implies that $(i^{t-1},j^{t-1})\in\Omega_{t}$, that is, $\vert\Omega_{t}\vert\geq 1$.
\end{remark}
\begin{remark}\rm
Since 
\begin{align*}
\rho_{\text{ATESP-PR}}&=1-(1+q_{\mathcal{ S}}^{2}q_{\mathcal{ V}}^{2}\mathbf{Var}_{i\sim \mathbf{u}_{\mathcal{ S}}, j\sim \mathbf{u}_{\mathcal{ V}}}[p_{i,j}^{t}])\delta_{\mathbf{u}_{\mathcal{ S}}, \mathbf{u}_{\mathcal{ V}}}^{2}(\mathcal{M},\mathcal{ N},\boldsymbol{\mathcal{S}},\boldsymbol{\mathcal{V}})\\
&\leq1-\delta_{\mathbf{u}_{\mathcal{ S}}, \mathbf{u}_{\mathcal{ V}}}^{2}(\mathcal{M},\mathcal{ N},\boldsymbol{\mathcal{S}},\boldsymbol{\mathcal{V}}),
\end{align*}
	we can conclude that the convergence factor of the ATESP-PR method is smaller than that of the NTESP method with respect to uniform sampling, and how much smaller depends on the value of $1+q_{\mathcal{ S}}^{2}q_{\mathcal{ V}}^{2}\mathbf{Var}_{i\sim \mathbf{u}_{\mathcal{ S}}, j\sim \mathbf{u}_{\mathcal{ V}}}[p_{i,j}^{t}]$.
\end{remark}
	\begin{theorem}\label{thmATESPCS}
	Let $\mathcal{X}^{\star}$ satisfy $\mathcal{A}*\mathcal{X}^{\star}*\mathcal{B}=\mathcal{C}$.  Then the iteration sequence $\{\mathcal{X}^{t}\}_{t=1}^{\infty}$ calculated by the ATESP-CS method, i.e., the third case of Algorithm \ref{ATESP}, with initial iteration $\text{vec}_{t}(\mathcal{X}^{0})\in \mathbf{Range}(((\mathcal{N}^{-1})^{ST}*\mathcal{ B}^{C})\otimes_{t}( \mathcal{ M}^{-1}*\mathcal{A}^{T}))$ satisfies
\begin{align*}
\mathbb{E}[\|\mathcal{X}^{t}-\mathcal{X}^{\star}\|_{F(\mathcal{M},\mathcal{ N})}^{2}\mid\mathcal{X}^{0}]&\leq{\rho_{\text{ATESP-CS}}}^{t}\|\mathcal{X}^{0}-\mathcal{X}^{\star}\|_{F(\mathcal{M},\mathcal{ N})}^{2},
 \end{align*}
 where $\rho_{\text{ATESP-CS}}=1-\theta\delta_{\infty}^{2}(\mathcal{M},\mathcal{ N},\boldsymbol{\mathcal{S}},\boldsymbol{\mathcal{V}})-(1-\theta)\delta_{\mathbf{p}_{\mathcal{S}}, \mathbf{p}_{\mathcal{V}}}^{2}(\mathcal{M},\mathcal{ N},\boldsymbol{\mathcal{S}},\boldsymbol{\mathcal{V}})$.
\end{theorem}
\begin{remark}\rm
	The convergence factor of the ATESP-CS method is a convex combination of ones of the NTESP and ATESP-MD methods, and the closer $\theta$ approaches $1$, the smaller the convergence factor of the ATESP-CS method is. In our numerical experiments, we set $\theta=0.5$.
\end{remark}
\begin{remark}\rm
	From Lemma \ref{sizedelta}, we know that the convergence factors of the NTESP, ATESP-MD, ATESP-PR and ATESP-CS methods are smaller than $1$ under the assumption that $\mathbb{E}_{i\sim \mathbf{p}_{\mathcal{ S}}}[\mathcal{ Z}_{i}]$ and $\mathbb{E}_{j\sim \mathbf{p}_{\mathcal{ V}}}[\mathcal{ W}_{j}]$ are T-symmetric T-positive definite with probability $1$, which show that these methods are convergent.
\end{remark}
\section{The Fourier version of the TESP method}\label{subTESPF}

We first present an efficient implementation of the TESP method in the Fourier domain, i.e., Algorithm \ref{TESPF}, and then give its corresponding convergence guarantee.

\begin{algorithm}[htbp]
	\caption{TESP method in the Fourier domain}\label{TESPF} 
	\begin{algorithmic}[1]
\State \textbf{Input:} $\mathcal{X}^{0}\in \mathbb{K}^{r\times s}_{l}$, $\mathcal{A}\in \mathbb{K}^{m\times r }_{l}$, $\mathcal{B}\in \mathbb{K}^{s\times n}_{l}$, and $\mathcal{C}\in \mathbb{K}^{m\times n }_{l}$
\State \textbf{Parameters:} fixed distributions $\mathfrak{D}_{\mathcal{S}}$ and $\mathfrak{D}_{\mathcal{V}}$ over random tubal matrices, T-symmetric T-positive definite tubal matrices $\mathcal{M}\in \mathbb{K}^{r\times r}_{l}$ and $\mathcal{N}\in \mathbb{K}^{s\times s}_{l}$
		\State $\widehat{\mathcal{X}}^{0}= \texttt{fft}(\mathcal{X}^{0},[~],3)$, $\widehat{\mathcal{A}}=
		\texttt{fft}(\mathcal{A},[~],3)$, $\widehat{\mathcal{B}}= \texttt{fft}(\mathcal{B},[~],3)$, $\widehat{\mathcal{C}}= \texttt{fft}(\mathcal{C},[~],3)$,  $\widehat{\mathcal{M}}=\texttt{fft}(\mathcal{M},[~],3)$, $\widehat{\mathcal{N}}=\texttt{fft}(\mathcal{N},[~],3)$
		\For{$t=0,1,\cdots$}
		\State Sample independent copies $\mathcal{S}\sim \mathfrak{D}_{\mathcal{S}}$ and $\mathcal{V}\sim \mathfrak{D}_{\mathcal{V}}$ 
		\State $\widehat{\mathcal{S}}= \texttt{fft}(\mathcal{S},[~],3)$ and $\widehat{\mathcal{V}}= \texttt{fft}(\mathcal{V},[~],3)$
		\For{$k=1,\cdots,\lceil\frac{l+1}{2}\rceil$}
		\State  $\widehat{\mathcal{E}}_{(k)}=\widehat{\mathcal{S}}_{(k)}\left(\widehat{\mathcal{S}}_{(k)}^{H}\widehat{\mathcal{A}}_{(k)}\widehat{\mathcal{M}}_{(k)}^{-1}\widehat{\mathcal{A}}_{(k)}^{H}\widehat{\mathcal{S}}_{(k)}\right)^{\dag}\widehat{\mathcal{S}}_{(k)}^{H}$ and $\widehat{\mathcal{G}}_{(k)}=\widehat{\mathcal{V}}_{(k)}(\widehat{\mathcal{V}}_{(k)}^{H}\widehat{\mathcal{B}}_{(k)}^{H}\widehat{\mathcal{N}}_{(k)}^{-1}\widehat{\mathcal{B}}_{(k)}\widehat{\mathcal{V}}_{(k)})^{\dag}\widehat{\mathcal{V}}_{(k)}^{H}$
		\State $\widehat{\mathcal{X}}^{t+1}_{(k)}=\widehat{\mathcal{X}}^{t}_{(k)}-{\widehat{\mathcal{M}}_{(k)}}^{-1}{\widehat{\mathcal{A}}_{(k)}}^{H}\widehat{\mathcal{E}}_{(k)}\left(\widehat{\mathcal{A}}_{(k)}\widehat{\mathcal{X}}_{(k)}^{t}\widehat{\mathcal{B}}_{(k)}-\widehat{\mathcal{C}}_{(k)}\right)\widehat{\mathcal{G}}_{(k)}\widehat{\mathcal{B}}_{(k)}^{H}\widehat{\mathcal{N}}_{(k)}^{-1}$
		\EndFor
		\For{$k=\lceil\frac{l+1}{2}\rceil+1,\cdots,l$}
		\State $\widehat{\mathcal{X}}^{t+1}_{(k)}=\text{conj}(\widehat{\mathcal{X}}^{t+1}_{(l-k+2)})$		
		\EndFor
		\EndFor
		\State $\mathcal{X}^{t+1}=\texttt{ifft}\left(\widehat{\mathcal{X}}^{t+1},[~],3\right)$
  \State \textbf{Output:} last iterate $\mathcal{X}^{t+1}$
	\end{algorithmic}
\end{algorithm}

\begin{theorem}\label{thmTESPF}
	With the notation in Algorithm \ref{TESPF}, 
	assume that $\mathbb{E}[\text{bdiag}(\widehat{\mathcal{Z}})]$ and $\mathbb{E}[\text{bdiag}(\widehat{\mathcal{W}})]$ are Hermitian positive definite with probability $1$, where $\text{bdiag}(\widehat{\mathcal{Z}})$ and $\text{bdiag}(\widehat{\mathcal{W}})$ are the block diagonal matrices with each block respectively corresponding to the frontal slices of the tubal matrix $\widehat{\mathcal{Z}}=\texttt{\rm fft}(\mathcal{Z},[~],3)$ and $\widehat{\mathcal{W}}=\texttt{\rm fft}(\mathcal{W},[~],3)$ with $\mathcal{Z}$ and $\mathcal{W}$  are as defined in  Theorem \ref{thmTESP}. Let $\mathcal{X}^{\star}$ satisfy $\mathcal{A}*\mathcal{X}^{\star}*\mathcal{B}=\mathcal{C}$. Then the iteration sequence $\{\mathcal{X}^{t}\}_{t=1}^{\infty}$ calculated by Algorithm \ref{TESPF} with initial iteration $\mathcal{X}^{0}$ satisfies
\begin{align}\label{sec3.2e2}
\mathbb{E}\left[\|\mathcal{X}^{t}-\mathcal{X}^{\star}\|_{F(\mathcal{M},\mathcal{ N})}^{2}\mid\mathcal{X}^{0}\right]\leq\rho^{t}\|\mathcal{X}^{0}-\mathcal{X}^{\star}\|_{F(\mathcal{M},\mathcal{ N})}^{2},
\end{align}
where $\rho=1-\mathop{\min}_{k\in[l]}\lambda_{\min}(\mathbb{E}[\widehat{\mathcal{Z}}_{(k)}])\lambda_{\min}(\mathbb{E}[\widehat{\mathcal{W}}_{(k)}]$.
\end{theorem}	
\begin{remark}\rm
 From Algorithm \ref{TESPF}, we can find that the TESP method is equivalent to using the MESP method to solve $\lceil\frac{l+1}{2}\rceil$ independent matrix equations. And if line $5$ of Algorithm \ref{TESPF} uses the sketching tubal matrices defined in Definitions \ref{guassrandom} or \ref{samplingrandom}, then $\widehat{\mathcal{S}}_{(k)}$ and $\widehat{\mathcal{V}}_{(k)}$ for $k=1,\cdots,\lceil\frac{l+1}{2}\rceil$ in line $8$ of Algorithm \ref{TESPF}  will be the same. As pointed out in \cite{ma2021randomized,tang2022sketch}, it would be better to use different sketching matrices for these $\lceil\frac{l+1}{2}\rceil$  independent matrix equations. 
\end{remark}

\begin{remark}\rm
 For Theorem \ref{thmTESPF}, it is essentially equivalent to Theorem \ref{thmTESP}. The reason for giving it is to facilitate the convergence analysis for the TESP method when the random tubal matrices $\mathcal{ S}$ and $\mathcal{ V}$  have special discrete probability distributions specified as follows.
\end{remark}

We first recall the definition of the complete discrete sampling matrix presented in \cite{gower2015randomized}: A sampling matrix $S$ is called a complete discrete sampling matrix if it satisfies three conditions, that is, the random matrix $S$ has a discrete distribution, $ S=S_{i}\in \mathbb{C}^{m\times \tau}$ with probability $p_{i}>0$ and $S_{i}^{H}A$ having full row rank for $i=1,\cdots,q_{S}$, and $\boldsymbol{S}=[S_{1},\cdots,S_{q_{S}}]\in \mathbb{C}^{m\times q_{S}\tau}$ is such that $A^{H}\boldsymbol{S}$ has full row rank.
\begin{corollary}\label{corspede}
	With the notation in Algorithm \ref{TESPF} and Theorem \ref{thmTESPF}, let $\mathcal{S}$ and $\mathcal{V}$ be discrete sampling tubal matrices satisfying that $\widehat{\mathcal{S}}_{(k)}$  and $\widehat{\mathcal{V}}_{(k)}$ for $k=1,\cdots,l$ are all complete discrete sampling matrices, where $\widehat{\mathcal{S}}=\texttt{\rm fft
	}(\mathcal{S},[~],3)$ and $\widehat{\mathcal{V}}=\texttt{\rm fft
	}(\mathcal{V},[~],3)$. Let $\mathcal{X}^{\star}$ satisfy $\mathcal{A}*\mathcal{X}^{\star}=\mathcal{B}$.  Then, when $\mathbb{P}[\mathcal{ S}=\mathcal{ S}_{i}]=\frac{\|\mathcal{M}^{-\frac{1}{2}}*\mathcal{A}^{T}*\mathcal{S}_{i}\|_{F}^{2}}{\|\mathcal{M}^{-\frac{1}{2}}*\mathcal{A}^{T}*\boldsymbol{\mathcal{S}}\|_{F}^{2}}$ and $\mathbb{P}[\mathcal{V}=\mathcal{ V}_{j}]=\frac{\|\mathcal{N}^{-\frac{1}{2}}*\mathcal{B}*\mathcal{V}_{j}\|_{F}^{2}}{\|\mathcal{N}^{-\frac{1}{2}}*\mathcal{B}*\boldsymbol{\mathcal{V}}\|_{F}^{2}}$ 
	with $\boldsymbol{\mathcal{S}}=[\mathcal{S}_{1},\cdots,\mathcal{S}_{q_{\mathcal{ S}}}]$ and $\boldsymbol{\mathcal{V}}=[\mathcal{V}_{1},\cdots,\mathcal{V}_{q_{\mathcal{ V}}}]$ for $i=1,\cdots,q_{\mathcal{ S}}$ and $j=1,\cdots,q_{\mathcal{ V}}$, the iteration sequence $\{\mathcal{X}^{t}\}_{t=1}^{\infty}$ calculated by Algorithm \ref{TESPF} with initial iteration $\mathcal{X}^{0}$ satisfies
	\begin{align}\label{sec3.2e4}
		\mathbb{E}\left[\left\|\mathcal{X}^{t}-\mathcal{X}^{\star}\right\|_{F(\mathcal{M},\mathcal{ N})}^{2}\mid\mathcal{X}^{0}\right]\leq\rho^{t}\left\|\mathcal{X}^{0}-\mathcal{X}^{\star}\right\|_{F(\mathcal{M},\mathcal{ N})}^{2},
	\end{align}
	where $\rho=1-\mathop{\min}\limits_{k\in[l]}\frac{\lambda_{\min}^{+}\left(\widehat{\boldsymbol{\mathcal{S}}}^{H}_{(k)}\widehat{\mathcal{A}}_{(k)}\widehat{\mathcal{M}}^{-1}_{(k)}\widehat{\mathcal{A}}^{H}_{(k)}\widehat{\boldsymbol{\mathcal{S}}}_{(k)}\right)}{\|\mathcal{M}^{-\frac{1}{2}}*\mathcal{A}^{T}*\boldsymbol{\mathcal{S}}\|_{F}^{2}}\cdot\frac{\lambda_{\min}^{+}\left(\widehat{\boldsymbol{\mathcal{V}}}^{H}_{(k)}\widehat{\mathcal{B}}^{H}_{(k)}\widehat{\mathcal{N}}^{-1}_{(k)}\widehat{\mathcal{B}}_{(k)}\widehat{\boldsymbol{\mathcal{V}}}_{(k)}\right)}{\|\mathcal{N}^{-\frac{1}{2}}*\mathcal{B}*\boldsymbol{\mathcal{V}}\|_{F}^{2}}$.
\end{corollary}
\emph{Proof}:
The proof is similar to that of Corollary 4.1 in \cite{tang2022sketch}, so we omit it here.

\section{Some special cases}\label{subspecialcase}
For the TESP method, it has four parameters, i.e., the distributions $\mathfrak{D}_{\mathcal{S}}$ and $\mathfrak{D}_{\mathcal{V}}$, and the T-symmetric T-positive definite tubal matrices $\mathcal{M}$ and $\mathcal{N}$. In the following, we will discuss some special cases of  the TESP method when choosing specific parameters.
\subsection{Tensor equation randomized Kaczmarz (TERK) methods}\label{secTERK}
\begin{enumerate}
	\item TERK-both method
	
	By choosing $\mathcal{ S}=\mathcal{I}_{m(:,i,:)}\in\mathbb{K}^{m}_{l}$, $\mathcal{V}=\mathcal{I}_{n(:,j,:)}\in\mathbb{K}^{n}_{l}$, $\mathcal{M}=\mathcal{I}_{r}\in\mathbb{K}^{r\times r}_{l}$ and $\mathcal{N}=\mathcal{I}_{s}\in\mathbb{K}^{s\times s}_{l}$, the update formula  (\ref{update}) can be simplified to
\begin{align*}
\mathcal{ X}^{t+1}=&\mathcal{ X}^{t}-\mathcal{ A}_{(i,:,:)}^{T}*(\mathcal{ A}_{(i,:,:)}*\mathcal{ A}_{(i,:,:)}^{T})^{\dag}*(\mathcal{ A}_{(i,:,:)}*\mathcal{ X}^{t}*\mathcal{ B}_{(:,j,:)}-\mathcal{ C}_{(i,j,:)})\\
 &*(\mathcal{ B}_{(:,j,:)}^{T}*\mathcal{ B}_{(:,j,:)})^{\dag}*\mathcal{ B}_{(:,j,:)}^{T}.
 \end{align*}
	When the index pair $(i,j)$ is randomly selected, we call the method the TERK-both method which is the tensor version of the matrix equation randomized Kaczmarz (MERK) method (for consistency, 
	we refer to it as the MERK-both method) proposed in \cite{niu2022global,wu2022kaczmarz}. According to Corollary \ref{corspede}, we find that selecting $i$ and $j$ respectively with probabilities $p_{i}=\frac{\|\mathcal{A}_{(i,:,:)}\|_{F}^{2}}{\|\mathcal{A}\|_{F}^{2}}$ (proportional to the magnitude of $i$-th horizontal slice of $\mathcal{ A}$) and $p_{j}=\frac{\|\mathcal{B}_{(:,j,:)}\|_{F}^{2}}{\|\mathcal{B}\|_{F}^{2}}$ (proportional to the magnitude of $j$-th  lateral slice of $\mathcal{B}$)  results in a convergence with
	\begin{align*}
		\mathbb{E}\left[\left\|\mathcal{X}^{t}-\mathcal{X}^{\star}\right\|_{F}^{2}\mid\mathcal{X}^{0}\right]\leq\rho^{t}\left\|\mathcal{X}^{0}-\mathcal{X}^{\star}\right\|_{F}^{2},
	\end{align*}
 where $\rho=1-\mathop{\min}_{k\in[l]}\frac{\lambda_{\min}^{+}\left(\widehat{\mathcal{A}}_{(k)}\widehat{\mathcal{A}}^{H}_{(k)}\right)}{\|\mathcal{A}\|_{F}^{2}}\cdot\frac{\lambda_{\min}^{+}\left(\widehat{\mathcal{B}}^{H}_{(k)}\widehat{\mathcal{B}}_{(k)}\right)}{\|\mathcal{B}\|_{F}^{2}}$, and this recovers the convergence result of the MERK-both method given in Remark 2.3 in \cite{niu2022global}.
	\item TERK-left method
	
	By setting  $\mathcal{ S}=\mathcal{I}_{m(:,i,:)}\in\mathbb{K}^{m}_{l}$, $\mathcal{V}=\mathcal{I}_{n}\in\mathbb{K}^{n\times n}_{l}$, $\mathcal{M}=\mathcal{I}_{r}\in\mathbb{K}^{r\times r}_{l}$ and $\mathcal{N}=\mathcal{I}_{s}\in\mathbb{K}^{s\times s}_{l}$, the update formula  (\ref{update}) can be written as
	$$\mathcal{ X}^{t+1}=\mathcal{ X}^{t}-\mathcal{ A}_{(i,:,:)}^{T}*(\mathcal{ A}_{(i,:,:)}*\mathcal{ A}_{(i,:,:)}^{T})^{\dag}*(\mathcal{ A}_{(i,:,:)}*\mathcal{ X}^{t}*\mathcal{ B}-\mathcal{ C}_{(i,:,:)})*\mathcal{ B}^{\dag},$$ 
	and we call it the TERK-left method when the index $i$ is selected at random. Applying Theorem \ref{thmTESPF} and Corollary \ref{corspede}, we find that selecting $i$ with probability $p_{i}=\frac{\|\mathcal{A}_{(i,:,:)}\|_{F}^{2}}{\|\mathcal{A}\|_{F}^{2}}$ results in a convergence with
	\begin{align*}
		\mathbb{E}\left[\left\|\mathcal{X}^{t}-\mathcal{X}^{\star}\right\|_{F}^{2}\mid\mathcal{X}^{0}\right]\leq\rho^{t}\left\|\mathcal{X}^{0}-\mathcal{X}^{\star}\right\|_{F}^{2},
	\end{align*}
where $\rho=1-\mathop{\min}_{k\in[l]}\frac{\lambda_{\min}^{+}\left(\widehat{\mathcal{A}}_{(k)}\widehat{\mathcal{A}}^{H}_{(k)}\right)}{\|\mathcal{A}\|_{F}^{2}}$.
	\item TERK-right method
	
	Let $\mathcal{ S}=\mathcal{I}_{m}\in\mathbb{K}^{m\times m}_{l}$, $\mathcal{V}=\mathcal{I}_{n(:,j,:)}\in\mathbb{K}^{n}_{l}$, $\mathcal{M}=\mathcal{I}_{r}\in\mathbb{K}^{r\times r}_{l}$ and $\mathcal{N}=\mathcal{I}_{s}\in\mathbb{K}^{s\times s}_{l}$, the update formula  (\ref{update}) can be expressed as
	$$\mathcal{ X}^{t+1}=\mathcal{ X}^{t}-\mathcal{ A}^{\dag}*(\mathcal{ A}*\mathcal{ X}^{t}*\mathcal{ B}_{(:,j,:)}-\mathcal{ C}_{(:,j,:)})*(\mathcal{ B}_{(:,j,:)}^{T}*\mathcal{ B}_{(:,j,:)})^{\dag}*\mathcal{ B}_{(:,j,:)}^{T}.$$
	Similarly, we call it the TERK-right method when the index $j$ is selected at random. From Theorem \ref{thmTESPF} and Corollary \ref{corspede}, we know that selecting $j$ with probability $p_{j}=\frac{\|\mathcal{B}_{(:,j,:)}\|_{F}^{2}}{\|\mathcal{B}\|_{F}^{2}}$ results in a convergence with
	\begin{align*}
		\mathbb{E}\left[\left\|\mathcal{X}^{t}-\mathcal{X}^{\star}\right\|_{F}^{2}\mid\mathcal{X}^{0}\right]\leq\rho^{t}\left\|\mathcal{X}^{0}-\mathcal{X}^{\star}\right\|_{F}^{2},
	\end{align*}
 where $\rho=1-\mathop{\min}_{k\in[l]}\frac{\lambda_{\min}^{+}\left(\widehat{\mathcal{B}}^{H}_{(k)}\widehat{\mathcal{B}}_{(k)}\right)}{\|\mathcal{B}\|_{F}^{2}}$.
\end{enumerate}

\subsection{Tensor equation randomized coordinate descent (TERCD) methods}
\begin{enumerate}
	\item TERCD-both method
	
By setting $\mathcal{ S}=\mathcal{ A}*\mathcal{I}_{r(:,i,:)}\in\mathbb{K}^{m}_{l}$, $\mathcal{V}=\mathcal{ B}^{T}*\mathcal{I}_{s(:,j,:)}\in\mathbb{K}^{n}_{l}$, $\mathcal{M}=\mathcal{A}^{T}*\mathcal{A}\in\mathbb{K}^{r\times r}_{l}$ and $\mathcal{N}=\mathcal{B}*\mathcal{ B}^{T}\in\mathbb{K}^{s\times s}_{l}$, the update formula  (\ref{update}) can be reduced to
\begin{align*}
\mathcal{ X}^{t+1}_{(i,j,:)}=&\mathcal{ X}^{t}_{(i,j,:)}-(\mathcal{A}_{(:,i,:)}^{T}*\mathcal{A}_{(:,i,:)})^{\dag}*\mathcal{A}_{(:,i,:)}^{T}*(\mathcal{ A}*\mathcal{ X}^{t}*\mathcal{ B}-\mathcal{ C})*\mathcal{B}_{(j,:,:)}^{T}\\
&*(\mathcal{B}_{(j,:,:)}*\mathcal{B}_{(j,:,:)}^{T})^{\dag}.
\end{align*}
When the index pair $(i,j)$ is randomly selected, we call it the TERCD-both method. Using Corollary \ref{corspede}, we find that selecting $i$ and $j$ respectively with probabilities $p_{i}=\frac{\|\mathcal{A}_{(:,i,:)}\|_{F}^{2}}{\|\mathcal{A}\|_{F}^{2}}$ (proportional to the magnitude of $i$-th lateral slice of $\mathcal{ A}$) and $p_{j}=\frac{\|\mathcal{B}_{(j,:,:)}\|_{F}^{2}}{\|\mathcal{B}\|_{F}^{2}}$ (proportional to the magnitude of $j$-th horizontal slice of $\mathcal{B}$)  results in a convergence with
\begin{align*}
	&\mathbf{E}\left[\left\|\mathcal{X}^{t}-\mathcal{X}^{\star}\right\|_{F(\mathcal{A}^{T}*\mathcal{A},\mathcal{B}*\mathcal{ B}^{T})}^{2}\mid\mathcal{X}^{0}\right]\leq\rho^{t}\left\|\mathcal{X}^{0}-\mathcal{X}^{\star}\right\|_{F(\mathcal{A}^{T}*\mathcal{A},\mathcal{B}*\mathcal{ B}^{T})}^{2},
\end{align*}
where $\rho=1-\mathop{\min}\limits_{k\in[l]}\frac{\lambda_{\min}^{+}\left(\widehat{\mathcal{A}}^{H}_{(k)}\widehat{\mathcal{A}}_{(k)}\right)}{\|\mathcal{A}\|_{F}^{2}}\cdot\frac{\lambda_{\min}^{+}\left(\widehat{\mathcal{B}}_{(k)}\widehat{\mathcal{B}}^{H}_{(k)}\right)}{\|\mathcal{B}\|_{F}^{2}}$.

\item TERCD-left method

By choosing $\mathcal{ S}=\mathcal{ A}*\mathcal{I}_{r(:,i,:)}\in\mathbb{K}^{m}_{l}$, $\mathcal{V}=\mathcal{I}_{n}\in\mathbb{K}^{n\times n}_{l}$, $\mathcal{M}=\mathcal{A}^{T}*\mathcal{A}\in\mathbb{K}^{r\times r}_{l}$ and $\mathcal{N}=\mathcal{I}_{s}\in\mathbb{K}^{s\times s}_{l}$, the update formula (\ref{update}) can be simplified to
$$\mathcal{ X}^{t+1}_{(i,:,:)}=\mathcal{ X}^{t}_{(i,:,:)}-(\mathcal{A}_{(:,i,:)}^{T}*\mathcal{A}_{(:,i,:)})^{\dag}*\mathcal{A}_{(:,i,:)}^{T}*(\mathcal{ A}*\mathcal{ X}^{t}*\mathcal{ B}-\mathcal{ C})*\mathcal{ B}^{\dag},$$ 
and we call it the TERCD-left method when the index $i$ is selected at random. According to Theorem \ref{thmTESPF} and Corollary \ref{corspede}, we see that by selecting $i$ with probability $p_{i}=\frac{\|\mathcal{A}_{(:,i,:)}\|_{F}^{2}}{\|\mathcal{A}\|_{F}^{2}}$ results in a convergence with
\begin{align*}
	\mathbf{E}\left[\left\|\mathcal{X}^{t}-\mathcal{X}^{\star}\right\|_{F(\mathcal{A}^{T}*\mathcal{A},\mathcal{I}_{s})}^{2}\mid\mathcal{X}^{0}\right]&\leq\rho^{t}\left\|\mathcal{X}^{0}-\mathcal{X}^{\star}\right\|_{F(\mathcal{A}^{T}*\mathcal{A},\mathcal{I}_{s})}^{2},
\end{align*}
where $\rho=1-\mathop{\min}_{k\in[l]}\frac{\lambda_{\min}^{+}\left(\widehat{\mathcal{A}}_{(k)}^{H}\widehat{\mathcal{A}}_{(k)}\right)}{\|\mathcal{A}\|_{F}^{2}}$.
\item TERCD-right method

By setting $\mathcal{ S}=\mathcal{I}_{m}\in\mathbb{K}^{m\times m}_{l}$, $\mathcal{V}=\mathcal{ B}^{T}*\mathcal{I}_{s(:,j,:)}\in\mathbb{K}^{n}_{l}$, $\mathcal{M}=\mathcal{I}_{r}\in\mathbb{K}^{r\times r}_{l}$ and $\mathcal{N}=\mathcal{B}*\mathcal{ B}^{T}\in\mathbb{K}^{s\times s}_{l}$, the update formula (\ref{update}) can be written as
$$\mathcal{ X}^{t+1}_{(:,j,:)}=\mathcal{ X}^{t}_{(:,j,:)}-\mathcal{ A}^{\dag}*(\mathcal{ A}*\mathcal{ X}^{t}*\mathcal{ B}-\mathcal{ C})*\mathcal{B}_{(j,:,:)}^{T}*(\mathcal{B}_{(j,:,:)}*\mathcal{B}_{(j,:,:)}^{T})^{\dag}.$$
Similarly, when the index $j$ is selected at random, we call it the TERCD-right method which is the tensor version of the projection-based randomized coordinate descent (PRCD) method (for consistency, 
we refer to it as the MERCD-right method) proposed in \cite{du2022convergence}. By Theorem \ref{thmTESPF} and Corollary \ref{corspede}, we find that selecting $j$ with probability $p_{j}=\frac{\|\mathcal{B}_{(j,:,:)}\|_{F}^{2}}{\|\mathcal{B}\|_{F}^{2}}$ results in a convergence with
\begin{align*}
	\mathbf{E}\left[\left\|\mathcal{X}^{t}-\mathcal{X}^{\star}\right\|_{F(\mathcal{I}_{r},\mathcal{B}*\mathcal{ B}^{T})}^{2}\mid\mathcal{X}^{0}\right]&\leq\rho^{t}\left\|\mathcal{X}^{0}-\mathcal{X}^{\star}\right\|_{F(\mathcal{I}_{r},\mathcal{B}*\mathcal{ B}^{T})}^{2},
\end{align*}
where $\rho=1-\mathop{\min}_{k\in[l]}\frac{\lambda_{\min}^{+}\left(\widehat{\mathcal{B}}_{(k)}\widehat{\mathcal{B}}^{H}_{(k)}\right)}{\|\mathcal{B}\|_{F}^{2}}$, and this recovers the convergence result of the MERCD-right method given in Remark 5 in \cite{du2022convergence}.
\end{enumerate}

\subsection{Some combinations of TERK and TERCD}
\begin{enumerate}
	\item TERK-RCD method
	
	By choosing $\mathcal{ S}=\mathcal{I}_{m(:,i,:)}\in\mathbb{K}^{m}_{l}$, $\mathcal{V}=\mathcal{ B}^{T}*\mathcal{I}_{s(:,j,:)}\in\mathbb{K}^{n}_{l}$, $\mathcal{M}=\mathcal{I}_{r}\in\mathbb{K}^{r\times r}_{l}$ and $\mathcal{N}=\mathcal{B}*\mathcal{ B}^{T}\in\mathbb{K}^{s\times s}_{l}$, the update formula (\ref{update}) can be reduced to
 \begin{align*}
	\mathcal{ X}^{t+1}_{(:,j,:)}=&\mathcal{ X}^{t}_{(:,j,:)}-\mathcal{A}_{(i,:,:)}^{T}*(\mathcal{A}_{(i,:,:)}*\mathcal{A}_{(i,:,:)}^{T})^{\dag}*(\mathcal{ A}_{(i,:,:)}*\mathcal{ X}^{t}*\mathcal{ B}-\mathcal{ C}_{(i,:,:)})\\
 &*\mathcal{B}_{(j,:,:)}^{T}*(\mathcal{B}_{(j,:,:)}*\mathcal{B}_{(j,:,:)}^{T})^{\dag}.
 \end{align*}
	When the index pair $(i,j)$ is randomly selected, we call the method the TERK-RCD method. Applying Corollary \ref{corspede}, we know that selecting $i$ and $j$ respectively with probabilities $p_{i}=\frac{\|\mathcal{A}_{(i,:,:)}\|_{F}^{2}}{\|\mathcal{A}\|_{F}^{2}}$ and $p_{j}=\frac{\|\mathcal{B}_{(j,:,:)}\|_{F}^{2}}{\|\mathcal{B}\|_{F}^{2}}$ results in a convergence with
	\begin{align*}
		\mathbf{E}\left[\left\|\mathcal{X}^{t}-\mathcal{X}^{\star}\right\|_{F(\mathcal{I}_{r},\mathcal{B}*\mathcal{ B}^{T})}^{2}\mid\mathcal{X}^{0}\right]\leq\rho^{t}\left\|\mathcal{X}^{0}-\mathcal{X}^{\star}\right\|_{F(\mathcal{I}_{r},\mathcal{B}*\mathcal{ B}^{T})}^{2},
	\end{align*}
where $\rho=1-\mathop{\min}\limits_{k\in[l]}\frac{\lambda_{\min}^{+}\left(\widehat{\mathcal{A}}_{(k)}\widehat{\mathcal{A}}^{H}_{(k)}\right)}{\|\mathcal{A}\|_{F}^{2}}\cdot\frac{\lambda_{\min}^{+}\left(\widehat{\mathcal{B}}_{(k)}\widehat{\mathcal{B}}^{H}_{(k)}\right)}{\|\mathcal{B}\|_{F}^{2}}$.
	\item TERCD-RK method
	
	By setting $\mathcal{ S}=\mathcal{ A}*\mathcal{I}_{r(:,i,:)}\in\mathbb{K}^{m}_{l}$, $\mathcal{V}=\mathcal{I}_{n(:,j,:)}\in\mathbb{K}^{n}_{l}$, $\mathcal{M}=\mathcal{A}^{T}*\mathcal{A}\in\mathbb{K}^{r\times r}_{l}$ and $\mathcal{N}=\mathcal{I}_{s}\in\mathbb{K}^{s\times s}_{l}$, the update formula (\ref{update}) can be expressed as
 
\begin{align*}
\mathcal{ X}^{t+1}_{(i,:,:)}=&\mathcal{ X}^{t}_{(i,:,:)}-(\mathcal{A}_{(:,i,:)}^{T}*\mathcal{A}_{(:,i,:)})^{\dag}*\mathcal{A}_{(:,i,:)}^{T}*(\mathcal{ A}*\mathcal{ X}^{t}*\mathcal{ B}_{(:,j,:)}-\mathcal{ C}_{(:,j,:)})\\
& *(\mathcal{B}_{(:,j,:)}^{T}*\mathcal{B}_{(:,j,:)})^{\dag}*\mathcal{B}_{(:,j,:)}^{T}.
\end{align*}
	Similarly, we call it the TERCD-RK method when the index pair $(i,j)$ is randomly selected. From Corollary \ref{corspede}, we see that selecting $i$ and $j$ respectively with probabilities $p_{i}=\frac{\|\mathcal{A}_{(:,i,:)}\|_{F}^{2}}{\|\mathcal{A}\|_{F}^{2}}$ and $p_{j}=\frac{\|\mathcal{B}_{(:,j,:)}\|_{F}^{2}}{\|\mathcal{B}\|_{F}^{2}}$ results in a convergence with
	\begin{align*}
		\mathbf{E}\left[\left\|\mathcal{X}^{t}-\mathcal{X}^{\star}\right\|_{F(\mathcal{A}^{T}*\mathcal{A},\mathcal{I}_{s})}^{2}\mid\mathcal{X}^{0}\right]\leq\rho^{t}\left\|\mathcal{X}^{0}-\mathcal{X}^{\star}\right\|_{F(\mathcal{A}^{T}*\mathcal{A},\mathcal{I}_{s})}^{2},
	\end{align*}
where $\rho=1-\mathop{\min}\limits_{k\in[l]}\frac{\lambda_{\min}^{+}\left(\widehat{\mathcal{A}}_{(k)}^{H}\widehat{\mathcal{A}}_{(k)}\right)}{\|\mathcal{A}\|_{F}^{2}}\cdot\frac{\lambda_{\min}^{+}\left(\widehat{\mathcal{B}}^{H}_{(k)}\widehat{\mathcal{B}}_{(k)}\right)}{\|\mathcal{B}\|_{F}^{2}}$.
\end{enumerate}

In a word, different parameters in the TESP method lead to different methods, which are summarized in Table \ref{TERKTERCD} for clarity. Further, these methods can also be combined with the adaptive sampling strategies proposed in Section \ref{secadaptiveTESP} to obtain corresponding adaptive methods.

\begin{table}[h]
\begin{center}
\caption{Summary of special cases of the TESP method.}\label{TERKTERCD}
\begin{tabular}{|c|c|c|c|}
\hline
TESP	 & \begin{tabular}[c]{@{}c@{}}not sample\\  $\mathcal{ A}$ \end{tabular}   & \begin{tabular}[c]{@{}c@{}}sample horizontal\\  slice of $\mathcal{ A}$\end{tabular} & \begin{tabular}[c]{@{}c@{}}sample lateral \\  slice of $\mathcal{ A}$\end{tabular} \\ \hline
\begin{tabular}[c]{@{}c@{}}not sample\\  $\mathcal{ B}$ \end{tabular}	& -- & TERK-left & TERCD-left \\ \hline
\begin{tabular}[c]{@{}c@{}}sample horizontal\\  slice of $\mathcal{B}$\end{tabular} 	&TERCD-right  & TERK-RCD & TERCD-both\\ \hline
\begin{tabular}[c]{@{}c@{}}sample lateral \\  slice of $\mathcal{B}$\end{tabular}	&TERK-right  & TERK-both &TERCD-RK  \\ \hline
\end{tabular}
\end{center}
\end{table}

\section{Numerical experiments}\label{subexperi} 
\subsection{Implementation tricks and computation complexity}
Similar to \cite{gower2019adaptive,tang2022sketch}, we implement the nonadaptive and adaptive TESP methods including the NTESP, ATESP-MD, ATESP-PR and ATESP-CS methods in their corresponding fast versions, 
for example, the fast version of the ATESP-PR method is given in Algorithm \ref{FATESPF} in the appendix. 

For computation complexity, we analyze the computational costs of each iteration of the above methods. The specific analysis 
is as follows:
\begin{enumerate}
	\item When the sketched residuals $\{\widehat{\mathcal{R}}^{t}_{i,j}: i=1,\cdots,q_{\mathcal{ S}}, j=1,\cdots,q_{\mathcal{ V}}\}$ are precomputed, computing the sketched losses $\{f_{i,j}(\mathcal{X}^{t}): i=1,\cdots,q_{\mathcal{ S}}, j=1,\cdots,q_{\mathcal{ V}}\}$ requires 
	$$2\tau\zeta lq_{\mathcal{S}}q_{\mathcal{V}}~(l>1)\quad \text{or} \quad(2\tau\zeta-1)q_{\mathcal{S}}q_{\mathcal{V}}~(l=1)$$
	 flops, i.e., $\mathcal{O}(\tau\zeta lq_{\mathcal{S}}q_{\mathcal{V}})$ flops.
	 
	  \item  For the NTESP, ATESP-PR and ATESP-CS methods, computing the sampling probabilities $\mathbf{p}_{\mathcal{ S},\mathcal{ V}}^{t}$ from the sketched losses $\{f_{i,j}(\mathcal{X}^{t}): i=1,\cdots,q_{\mathcal{ S}}, j=1,\cdots,q_{\mathcal{ V}}\}$ requires $\mathcal{O}(1)$, $2q_{\mathcal{ S}}q_{\mathcal{V}}$ and  $6q_{\mathcal{ S}}q_{\mathcal{V}}$ flops, respectively. For the ATESP-MD method, it requires $q_{\mathcal{ S}}q_{\mathcal{V}}(\tau\zeta>1)$ or  $\mathcal{O}(log(q_{\mathcal{ S}}q_{\mathcal{V}}))(\tau\zeta=1)$ flops.

	\item When $$\{\widehat{\mathcal{M}}^{-1}_{(k)}\widehat{\mathcal{A}}_{(k)}^{H}(\widehat{\mathcal{S}}_{i^{t}})_{(k)}(\widehat{\mathcal{C}}_{i^{t}})_{(k)}: i=1,\cdots,q_{\mathcal{ S}}, k=1,\cdots,l\}$$
	and
	$$\{(\widehat{\mathcal{D}}_{j})^{H}_{(k)}(\widehat{\mathcal{V}}_{j})_{(k)}^{H}\widehat{\mathcal{B}}^{H}_{(k)}\widehat{\mathcal{N}}^{-1}_{(k)}: j=1,\cdots,q_{\mathcal{ V}}, k=1,\cdots,l\}$$
	are precomputed, updating $\widehat{\mathcal{X}}^{t}$ to $\widehat{\mathcal{X}}^{t+1}$ requires 
	$$\min\{(2\tau-1+2s)r\zeta,(2\zeta-1+2r)\tau s\}\lceil\frac{l+1}{2}\rceil+rs(l-\lceil\frac{l+1}{2}\rceil)$$ 
	flops , i.e., $\mathcal{O}(\min\{\max\{\tau,s\}\zeta r,\max\{\zeta,r\}\tau s\}\lceil\frac{l+1}{2}\rceil)$ flops.

	\item  When $$\{(\widehat{\mathcal{C}}_{i})_{(k)}^{H}(\widehat{\mathcal{S}_{i}})_{(k)}^{H}\widehat{\mathcal{A}}_{(k)}\widehat{\mathcal{M}}^{-1}_{(k)}\widehat{\mathcal{A}}_{(k)}^{H}(\widehat{\mathcal{S}}_{v})_{(k)}(\widehat{\mathcal{C}}_{v})_{(k)} : i,v=1,\cdots,q_{\mathcal{ S}}, k=1,\cdots,l\}$$ 
	and
	$$\{(\widehat{\mathcal{D}}_{j})_{(k)}^{H}(\widehat{\mathcal{V}_{j}})_{(k)}^{H}\widehat{\mathcal{B}}_{(k)}^{H}\widehat{\mathcal{N}}^{-1}_{(k)}\widehat{\mathcal{B}}_{(k)}(\widehat{\mathcal{V}}_{w})_{(k)}(\widehat{\mathcal{D}}_{w})_{(k)}: j,w=1,\cdots,q_{\mathcal{ V}}, k=1,\cdots,l\}$$ 
	are precomputed, updating $\{\widehat{\mathcal{R}}^{t}_{i,j}:i=1,\cdots,q_{\mathcal{ S}}, j=1,\cdots,q_{\mathcal{ V}}\}$ to $\{\widehat{\mathcal{R}}^{t+1}_{i,j}: i=1,\cdots,q_{\mathcal{ S}}, j=1,\cdots,q_{\mathcal{ V}}\}$ requires 
	$$(2\tau^{2}\zeta+2\tau\zeta^{2}-\tau\zeta)q_{\mathcal{ S}}q_{\mathcal{ V}}\lceil\frac{l+1}{2}\rceil+\tau\zeta q_{\mathcal{ S}}q_{\mathcal{ V}}(l-\lceil\frac{l+1}{2}\rceil)$$
	 flops, i.e., $\mathcal{O}(\tau\zeta\max\{\tau,\zeta\}q_{\mathcal{ S}}q_{\mathcal{ V}}\lceil\frac{l+1}{2}\rceil)$ flops. Note that for the NTESP method, one only needs to compute the single sketched residual $\widehat{\mathcal{R}}_{i^{t},j^{t}}^{t}$, where $$(\widehat{\mathcal{R}}_{i^{t},i^{t}}^{t})_{(k)}=(\widehat{\mathcal{C}}_{i^{t}})_{(k)}^{H}\left((\widehat{\mathcal{S}}_{i^{t}})_{(k)}^{H}\left(\widehat{\mathcal{A}}_{(k)}\widehat{\mathcal{X}}_{(k)}^{t}\widehat{\mathcal{B}}_{(k)}-\widehat{\mathcal{C}}_{(k)}\right)(\widehat{\mathcal{V}_{j}})_{(k)}\right)(\widehat{\mathcal{D}}_{j})_{(k)},$$
  for $k=1,\cdots,l$.
	 When
	 $$\{(\widehat{\mathcal{C}}_{i})_{(k)}^{H}(\widehat{\mathcal{S}_{i}})_{(k)}^{H}\widehat{\mathcal{A}}_{(k)}: i=1,\cdots,q_{\mathcal{ S}}, k=1,\cdots,l\},$$ 
	 $$\{\widehat{\mathcal{B}}_{(k)}(\widehat{\mathcal{V}_{j}})_{(k)}(\widehat{\mathcal{D}}_{j})_{(k)}: j=1,\cdots,q_{\mathcal{ V}}, k=1,\cdots,l\},$$ 
	 and
	 $$\{(\widehat{\mathcal{C}}_{i})_{(k)}^{H}(\widehat{\mathcal{S}_{i}})_{(k)}^{H}\widehat{\mathcal{C}}_{(k)}(\widehat{\mathcal{V}_{j}})_{(k)}(\widehat{\mathcal{D}}_{j})_{(k)}: i=1,\cdots,q_{\mathcal{ S}},j=1,\cdots,q_{\mathcal{ V}}, k=1,\cdots,l\}$$ 
	 are precomputed,  computing sketched residual $\widehat{\mathcal{R}}_{i^{t},j^{t}}^{t}$ directly from the iterate $\mathcal{X}^{t}$ costs 
	 $$\min\{(2r-1+2\zeta)\tau s,(2s-1+2\tau)r\zeta\}\lceil\frac{l+1}{2}\rceil+\tau\zeta(l-\lceil\frac{l+1}{2}\rceil)$$
	 flops, i.e., $\mathcal{O}(\min\{\max\{\tau,s\}\zeta r,\max\{\zeta,r\}\tau s\}\lceil\frac{l+1}{2}\rceil)$ flops. Hence, when 
	 $$\tau\zeta\max\{\tau,\zeta\}q_{\mathcal{ S}}q_{\mathcal{ V}}>\min\{\max\{\tau,s\}\zeta r,\max\{\zeta,r\}\tau s\},$$ 
	 it is cheaper for the NTESP method to compute the sketched residual $\widehat{\mathcal{R}}_{i^{t},j^{t}}^{t}$ directly than using update  formula.
\end{enumerate}

Putting all the costs together, the overall leading order complexity per iteration of the NTESP method and the adaptive cases (ATESP-MD, ATESP-PR, ATESP-CS) are $$\mathcal{O}(\min\{\max\{\tau,s\}\zeta r,\max\{\zeta,r\}\tau s\}\lceil\frac{l+1}{2}\rceil)$$ and $$\mathcal{O}(\max\{\min\{\max\{\tau,s\}\zeta r,\max\{\zeta,r\}\tau s\}, \tau\zeta\max\{\tau,\zeta\}q_{\mathcal{ S}}q_{\mathcal{ V}}\}\lceil\frac{l+1}{2}\rceil+\tau\zeta lq_{\mathcal{S}}q_{\mathcal{V}}),$$ respectively. 

\subsection{Experimental results}
We will use three numerical examples to illustrate the empirical performance of the proposed TESP method and its adaptive variants for solving the tensor equation (\ref{tensorequation}). It should be noted that in the following specific experiments, we only consider the special cases of the TESP-type methods, namely the TERK-type methods, which have been discussed in Section \ref{secTERK}. All experiments are conducted on a computer with an Intel Xeon W-2255 3.7 GHz CPU and 256 GB memory, and all the algorithms have been implemented in the MATLAB R2020b environment and Tensor-Tensor Product Toolbox \cite{ttproduct}. All computations start from the initial point $\mathcal{ X}^{0}=O$, where $O$ is the zero tubal matrix, and terminate once the relative residual norm (RRN) at $\mathcal{ X}^{t}$, deﬁned by
$$\text{RRN}=\frac{\|\mathcal{ C}-\mathcal{ A}*\mathcal{ X}^{t}*\mathcal{ B}\|_{F}}{\|\mathcal{ C}-\mathcal{ A}*\mathcal{ X}^{0}*\mathcal{ B}\|_{F}}$$
is less than $10^{-4}$, or the number of iterations (IT) exceeds $10^{8}$, or the computing time in seconds (CPU) exceeds $5000$ s. Note that we do not consider the precomputational cost, but only
the costs spent at each iteration. All results are averaged over $10$ trails.
\begin{example}\label{EX1}\rm
Applying the algebraic properties of t-product, we can transform the tensor equation (\ref{tensorequation}) into a tensor linear system 
\begin{align}\label{linearsystem}
	(\mathcal{ B}^{ST}\otimes_{t}\mathcal{ A})*\text{vec}_{t}(\mathcal{ X})=\text{vec}_{t}(\mathcal{ C}),
\end{align}
or a matrix equation
\begin{align}\label{matrixequation}
	\text{bcirc}(\mathcal{ A})\text{bcirc}(\mathcal{ X})\text{bcirc}(\mathcal{ B})=	\text{bcirc}(\mathcal{ C}).
\end{align}
The former can be solved by the TRK method  \cite{ma2021randomized}, which is a specal case of the TSP method  \cite{tang2022sketch}, while the latter can be solved by the MERK methods, which are specal cases of the MESP method. In this example, we compare the empirical performance of the TERK methods (including the TERK-both, TERK-left and TERK-right methods) for the tensor equation (\ref{tensorequation}), the TRK \cite{ma2021randomized} method for the tensor linear system (\ref{linearsystem}), and the MERK methods (including the MERK-both  \cite{niu2022global,wu2022kaczmarz}, MERK-left and MERK-right methods) for the  matrix equation (\ref{matrixequation}). We generate the tubal matrices $\mathcal{ A}\in\mathbb{K}^{m\times r}_{l}$, $\mathcal{B}\in\mathbb{K}^{s\times n}_{l}$ and $\mathcal{X}\in\mathbb{K}^{r\times s}_{l}$ by using the MATLAB function \texttt{randn}, and construct a tensor equation by setting $\mathcal{C}=\mathcal{ A}*\mathcal{X}*\mathcal{ B}$. Table \ref{ITCPU_RK} list the average IT and CPU for various RK-type methods, from which we can see that, in all settings, 
the three TERK methods  outperform their corresponding matrix counterparts in terms of CPU time and the number of iterations. Except for MERK-both, these matrix methods in turn perform better than the TRK method. For the TERK-both and TRK methods, the former takes substantially less time even though it has a comparable number of iterations as the latter.  In addition, it is important to note that although the four methods TERK-left, TERK-right, MERK-left and MERK-right exhibit excellent performance in this experiment, they may be not suitable for very large-scale equations since they require calculating the pseudoinverse.

\begin{table}[tp]
	\centering
	\fontsize{8}{8} \selectfont
	\caption{Comparison of the average IT and CPU for the RK-type methods}\label{ITCPU_RK}
	\setlength{\tabcolsep}{0.5mm}{
		\begin{tabular}{ccccccccccccc}
				\hline
			$m $    & $r $  & $s $  & $n $ & $l $ &  & \begin{tabular}[c]{@{}c@{}}TERK\\  -both \end{tabular} & \begin{tabular}[c]{@{}c@{}}TERK\\  -left \end{tabular}&\begin{tabular}[c]{@{}c@{}}TERK\\  -right \end{tabular}&\begin{tabular}[c]{@{}c@{}}TRK\\  \cite{ma2021randomized} \end{tabular}& \begin{tabular}[c]{@{}c@{}}MERK-both\\ \cite{niu2022global,wu2022kaczmarz}\end{tabular} & \begin{tabular}[c]{@{}c@{}}MERK\\ -left\end{tabular}& \begin{tabular}[c]{@{}c@{}}MERK\\ -right\end{tabular}  \\
			\hline
			\multirow{2}{*}{70}   &  \multirow{2}{*}{50}  &  \multirow{2}{*}{50}  &  \multirow{2}{*}{70} &  \multirow{2}{*}{10}& IT &  279906.9 &  2317.4  & 2244.2  & 280269  & -- & 23539.1 &  25933.3  \\
			&  &   &  & & CPU &  152.0171 & 1.4959 & 1.4842   & 283.1842 &5000 & 61.7545 & 71.6118     \\
			\hline
			\multirow{2}{*}{100}   &  \multirow{2}{*}{50}  &  \multirow{2}{*}{50}  &  \multirow{2}{*}{100} &  \multirow{2}{*}{10}& IT & 123274.5 &  1104.9 &  1094    & 123396 & -- & 10977  & 12015.9  \\
			&  &   &  & &CPU & 69.1066 &  0.8528 & 0.8571  & 158.5778 & 5000 & 36.1942  & 41.8186    \\
			\hline
			\multirow{2}{*}{50}   &  \multirow{2}{*}{70}  &  \multirow{2}{*}{70}  &  \multirow{2}{*}{50} &  \multirow{2}{*}{10}& IT &  291300.6 &  2622.5  & 2252.2  & 291822.5  & -- & 25766.4 &  24454.3   \\
			&  &   &  & & CPU &  221.4358 & 2.2087  & 1.6733  &  360.9614 & 5000& 108.1670 & 49.9247     \\
			\hline
			\multirow{2}{*}{50}   &  \multirow{2}{*}{100}  &  \multirow{2}{*}{100}  &  \multirow{2}{*}{50} &  \multirow{2}{*}{10}& IT & 120277.8  &  1144.1  &  1039   & 120667.7 & -- &10682.4  & 10673.7  \\
			&  &   &  & & CPU & 138.0595 & 1.4167  &  1.1463   &  268.4936  & 5000 & 67.8701 &  32.3411    \\
			\hline
		\end{tabular}
	}
\end{table}
\end{example}


\begin{example}\label{EX2}\rm
In this example, we compare the empirical  performance of the nonadaptive and adaptive TERK methods using 
the tensor equations generated as in Example \ref{EX1}. Here, we only take into account the time spent by a single subsystem in the Fourier domain due to the fact that the TERK methods can be implemented in parallel. The numerical results of the TERK-both, TERK-left and TERK-right methods are provided in Tables \ref{ITCPU_NARK_both}, \ref{ITCPU_NARK_left} and \ref{ITCPU_NARK_right}, respectively. From the three tables, we can find that the nonadaptive TERK methods require more iteration steps  than their corresponding adaptive methods,  indicating that the proposed adaptive sampling strategies can indeed speed up the convergence of the nonadaptive methods. For CPU time, the nonadaptive TERK methods also take up more than the adaptive TERK methods, with the exception that the TERK-both-PR and TERK-left-PR methods consume more CPU time than the corresponding nonadaptive ones in the setting of $m>r$ and $n>s$. The primary reason for this would be that, in this case, compared to the nonadaptive ones, the number of iterations of the TERK-both-PR and TERK-left-PR  methods does not reduce considerably, but the computation cost per iteration increases, resulting in no reduction in the overall time, which is in line with the theory.

\begin{table}[tp]
	\centering
	\fontsize{8}{8} \selectfont
	\caption{Comparison of the average IT and CPU for the nonadaptive and adaptive TERK-both methods}\label{ITCPU_NARK_both}
	\setlength{\tabcolsep}{0.7mm}{
		\begin{tabular}{cccccccccc}
			\hline
			$m $    & $r $  & $s $  & $n $ & $l $ &   & \begin{tabular}[c]{@{}c@{}}NTERK\\  -both \end{tabular}& \begin{tabular}[c]{@{}c@{}}ATERK\\  -both-MD \end{tabular}&  \begin{tabular}[c]{@{}c@{}}ATERK\\  -both-PR \end{tabular} &\begin{tabular}[c]{@{}c@{}}ATERK\\  -both-CS \end{tabular} \\
			\hline
			\multirow{2}{*}{150}   &  \multirow{2}{*}{50}  &  \multirow{2}{*}{50}  &  \multirow{2}{*}{150} &  \multirow{2}{*}{10}& IT &  65289.7  &  27597  & 51829.3    &  28683.3   \\
			&  &   &  & & CPU & 31.0849 & 12.8789  & 42.5905  &   3.5250   \\
			\hline
			\multirow{2}{*}{300}   &  \multirow{2}{*}{50}  &  \multirow{2}{*}{50}  &  \multirow{2}{*}{300} &  \multirow{2}{*}{10}& IT &  35612.7   &   12405  &28886.2  & 13370.4  \\
			&  &   &  & & CPU &  21.8306  & 12.8423  & 54.7469  &   2.0754     \\
			\hline
			\multirow{2}{*}{50}   &  \multirow{2}{*}{150}  &  \multirow{2}{*}{150}  &  \multirow{2}{*}{50} &  \multirow{2}{*}{10}& IT & 64297    & 29161 &  37168   &   29426.4  \\
			&  &   &  & & CPU & 94.6512 & 35.1994  & 51.0602  & 28.2686    \\
			\hline
			\multirow{2}{*}{50}   &  \multirow{2}{*}{300}  &  \multirow{2}{*}{300}  &  \multirow{2}{*}{50} &  \multirow{2}{*}{10}& IT &  35096  & 13968   &  17082.6  & 13968 \\
			&  &   &  & & CPU &159.1867  & 60.0866 & 78.6140  & 51.2417      \\
			\hline
		\end{tabular}
	}
\end{table}

\begin{table}[tp]
	\centering
	\fontsize{8}{8} \selectfont
	\caption{Comparison of the average IT and CPU for the nonadaptive and adaptive TERK-left methods}\label{ITCPU_NARK_left}
	\setlength{\tabcolsep}{0.7mm}{
		\begin{tabular}{cccccccccc}
		\hline
			$m $    & $r $  & $s $  & $n $ & $l $ &   & \begin{tabular}[c]{@{}c@{}}NTERK\\  -left \end{tabular}& \begin{tabular}[c]{@{}c@{}}ATERK\\  -left-MD \end{tabular}&  \begin{tabular}[c]{@{}c@{}}ATERK\\  -left-PR \end{tabular} &\begin{tabular}[c]{@{}c@{}}ATERK\\  -left-CS \end{tabular} \\
			\hline
			\multirow{2}{*}{150}   &  \multirow{2}{*}{50}  &  \multirow{2}{*}{50}  &  \multirow{2}{*}{150} &  \multirow{2}{*}{10}& IT &  742.1 &  444  & 578.4    &  464.2   \\
			&  &   &  & & CPU & 0.6067 &  0.3298  & 0.5165  &  0.4239  \\
			\hline
			\multirow{2}{*}{300}   &  \multirow{2}{*}{50}  &  \multirow{2}{*}{50}  &  \multirow{2}{*}{300} &  \multirow{2}{*}{10}& IT &  547.2   &   380  & 484.2  & 393.2   \\
			&  &   &  & & CPU &0.5642  & 0.3602  & 0.5784  &   0.4776     \\
			\hline
			\multirow{2}{*}{50}   &  \multirow{2}{*}{150}  &  \multirow{2}{*}{150}  &  \multirow{2}{*}{50} &  \multirow{2}{*}{10}& IT & 718.3    & 300 & 370.7   &   309.5  \\
			&  &   &  & & CPU & 1.5813 & 0.5680 & 0.7775  &  0.6677    \\
			\hline
			\multirow{2}{*}{50}   &  \multirow{2}{*}{300}  &  \multirow{2}{*}{300}  &  \multirow{2}{*}{50} &  \multirow{2}{*}{10}& IT &  510.9   & 189   &  224.4  &  191.5 \\
			&  &   &  & & CPU & 2.5093  & 0.8215  & 0.9857 &  0.8496     \\
			\hline
		\end{tabular}
	}
\end{table}

\begin{table}[tp]
	\centering
	\fontsize{8}{8} \selectfont
	\caption{Comparison of the average IT and CPU for the nonadaptive and adaptive TERK-right methods}\label{ITCPU_NARK_right}
	\setlength{\tabcolsep}{0.7mm}{
		\begin{tabular}{cccccccccc}
			\hline
			$m $    & $r $  & $s $  & $n $ & $l $ &   & \begin{tabular}[c]{@{}c@{}}NTERK\\  -right \end{tabular}& \begin{tabular}[c]{@{}c@{}}ATERK\\  -right-MD \end{tabular}&  \begin{tabular}[c]{@{}c@{}}ATERK\\  -right-PR \end{tabular} &\begin{tabular}[c]{@{}c@{}}ATERK\\  -right-CS \end{tabular} \\
			\hline
			\multirow{2}{*}{150}   &  \multirow{2}{*}{50}  &  \multirow{2}{*}{50}  &  \multirow{2}{*}{150} &  \multirow{2}{*}{10}& IT &  723.4  & 450  & 586.7    &  465.1   \\
			&  &   &  & & CPU & 0.5127 &  0.1606 & 0.3015   &   0.2408   \\
			\hline
			\multirow{2}{*}{300}   &  \multirow{2}{*}{50}  &  \multirow{2}{*}{50}  &  \multirow{2}{*}{300} &  \multirow{2}{*}{10}& IT &  551.5   &   380  & 485.8 & 392.9   \\
			&  &   &  & & CPU &0.4928   & 0.2059  & 0.3576  &   0.2985    \\
			\hline
			\multirow{2}{*}{50}   &  \multirow{2}{*}{150}  &  \multirow{2}{*}{150}  &  \multirow{2}{*}{50} &  \multirow{2}{*}{10}& IT & 704.6    & 303 &  372.1   &   309.3  \\
			&  &   &  & & CPU & 0.9787  & 0.3426  & 0.4697   &  0.3962    \\
			\hline
			\multirow{2}{*}{50}   &  \multirow{2}{*}{300}  &  \multirow{2}{*}{300}  &  \multirow{2}{*}{50} &  \multirow{2}{*}{10}& IT &  542.1  &  186 &  218.8  &  188.4 \\
			&  &   &  & & CPU & 1.9569 & 0.6052  & 0.7548   &  0.6398     \\
			\hline
		\end{tabular}
	}
\end{table}
\end{example}
	\begin{example}\label{EX3}\rm
	In this example, we illustrate the eﬀectiveness of our proposed methods through a color image restoration problem \cite{el2021tensor}.
	Let $\mathcal{ X}^{\star}\in\mathbb{K}^{r\times s}_{l}$ and $\mathcal{C}\in\mathbb{K}^{r\times s}_{l}$ be the original blur-free and observed blurred color image, respectively. We consider the following full blurring model:
	\begin{align}\label{imagerestor}
		(H\otimes B\otimes A)\begin{bmatrix}
			\text{vec}(\mathcal{ X}_{(1)}^{\star})  \\
			\text{vec}(\mathcal{ X}_{(2)}^{\star}) \\
			\text{vec}(\mathcal{ X}_{(3)}^{\star}) \\
		\end{bmatrix}=\begin{bmatrix}
			\text{vec}(\mathcal{ C}_{(1)})  \\
			\text{vec}(\mathcal{ C}_{(2)}) \\
			\text{vec}(\mathcal{ C}_{(3)}) \\
		\end{bmatrix},
	\end{align}
	where $H\in\mathbb{R}^{3\times 3}$ is a matrix modeling the cross-channel blurring, and each row sums to one; $B\in\mathbb{R}^{n\times s}$ and $A\in\mathbb{R}^{m\times r}$
	are matrices modeling the horizontal within-blurring and the vertical within-blurring, respectively; for more details, see \cite{hansen2006deblurring}. Here, we consider a special case where $H$ is a circular matrix, i.e., 
	$$H=\begin{bmatrix}
		h_{1}& h_{3}& h_{2} \\
		h_{2} &  h_{1}& h_{3}\\
		h_{3} &h_{2}&   h_{1}\\
	\end{bmatrix},$$ then (\ref{imagerestor}) can be rewritten as
	\begin{align*}
		\begin{bmatrix}
			h_{1}(B\otimes A) & h_{3}(B\otimes A) & h_{2}(B\otimes A) \\
			h_{2}(B\otimes A) &  h_{1}(B\otimes A)& h_{3}(B\otimes A) \\
			h_{3}(B\otimes A) &h_{2}(B\otimes A)& h_{1}(B\otimes A)\\
		\end{bmatrix}\text{unfold}(\text{vet}_{t}(\mathcal{ X}^{\star}))=\text{unfold}(\text{vet}_{t}(\mathcal{ C})).
	\end{align*}
 Let $\mathcal{ A}\in\mathbb{K}_{l}^{m\times r}$ satisfy $\mathcal{ A}_{(k)}=h_{k}A$ for $k=1,2,3$ and $\mathcal{B}\in\mathbb{K}_{l}^{s\times n}$ satisfy $\mathcal{ B}_{(1)}=B^{T}$ and $\mathcal{ B}_{(k)}$ for $k=2,3$ are all zero matrices. Thus, we have 

\begin{align*}
&\begin{bmatrix}
h_{1}(B\otimes A) & h_{3}(B\otimes A) & h_{2}(B\otimes A) \\
h_{2}(B\otimes A) &  h_{1}(B\otimes A)  & h_{3}(B\otimes A) \\
h_{3}(B\otimes A) &h_{2}(B\otimes A) & h_{1}(B\otimes A)\\
\end{bmatrix}
=	\begin{bmatrix}
B\otimes h_{1}A & B\otimes h_{3}A & B\otimes h_{2}A \\
B\otimes h_{2}A & B\otimes  h_{1}A  & B\otimes h_{3}A \\
B\otimes h_{3}A & B\otimes h_{2}A  & B\otimes h_{1}A\\
\end{bmatrix}\\
=&\begin{bmatrix}
(\mathcal{ B}^{ST}\otimes_{t}\mathcal{ A})_{(1)}& 	(\mathcal{ B}^{ST}\otimes_{t}\mathcal{ A})_{(3)} &  	(\mathcal{ B}^{ST}\otimes_{t}\mathcal{ A})_{(2)} \\
(\mathcal{ B}^{ST}\otimes_{t}\mathcal{ A})_{(2)}&  	(\mathcal{ B}^{ST}\otimes_{t}\mathcal{ A})_{(1)}  & 	(\mathcal{ B}^{ST}\otimes_{t}\mathcal{ A})_{(3)} \\
(\mathcal{ B}^{ST}\otimes_{t}\mathcal{ A})_{(3)} &  	(\mathcal{ B}^{ST}\otimes_{t}\mathcal{ A})_{(2)}  &   	(\mathcal{ B}^{ST}\otimes_{t}\mathcal{ A})_{(1)}\\
\end{bmatrix}=\text{bcirc}(\mathcal{ B}^{ST}\otimes_{t}\mathcal{ A}),
	\end{align*}
	and hence 
	\begin{align*}
		\text{bcirc}(\mathcal{ B}^{ST}\otimes_{t}\mathcal{ A})\text{unfold}(\text{vet}_{t}(\mathcal{ X}^{\star}))=\text{unfold}(\text{vet}_{t}(\mathcal{ C})).
	\end{align*}
	Therefore, the blurring model (\ref{imagerestor}) can be represented by the following tensor equation:
	\begin{align*}
		\mathcal{ A}*\mathcal{ X}^{\star}*\mathcal{ B}=\mathcal{ C}.
	\end{align*}

Specifically, we consider a $192\times 128$ color image from the Corel5K dataset, and $A$ and $B$ are Gaussian Toeplitz matrices with dimensions $198\times 198$ and $168\times168$, respectively,  whose elements are defined as
 \begin{equation*}
A_{(i,j)}, B_{(i,j)}=	\left\{
	\begin{array}{lcl}
	\frac{1}{\sigma \sqrt{2\pi}}\text{exp}\left(-\frac{(i-j)^{2}}{2\sigma^{2}}\right), &  \vert i-j\vert\leq r, \\
		0, &  \text{otherwise}.
	\end{array} \right.
\end{equation*}
And we set $\sigma=7$, $r=3$ and
$$H=\begin{bmatrix}
	0.3& 0.4& 0.3 \\
	0.3 &  0.3& 0.4\\
	0.4 &0.3&   0.3\\
\end{bmatrix}.$$

We first compare the performance of the TERK-both, TERK-right and TERK-left methods with the Tensor T-Global GMRES method proposed in  \cite{el2021tensor}, and here we also consider the precomputation cost of the three TERK methods. It can be seen from Table \ref{re} that, our proposed methods outperform the GMERS method in terms of both CPU time and PSNR value, except that the TERK-both method takes the most time. It is worth mentioning that although the TERK-both method is the most time-consuming, it is more suitable for large-scale problems because it does not require to calculate large-dimensional matrix multiplications and pseudoinverses. In addition, the original clean image, its corresponding blurry observation and the images recovered from the Tensor T-Global GMRES, TERK-both, TERK-left and TERK-right methods are shown in Figure \ref{fig5}.
\begin{table}[tp]
	\centering
	\fontsize{8}{8} \selectfont
	\caption{Comparison of the average IT, CPU and PSNR for the tensor T-Global GMRES, TERK-both, TERK-left and TERK-right methods on real world color image data.}\label{re}
	
		\begin{tabular}{ccccc}
			\hline
			 & T-Global GMRES  \cite{el2021tensor} & TERK-both& TERK-left &TERK-right \\
			\hline
			IT &  599  & 22337  & 1401.2    &  997.7   \\
			 CPU & 30.6517 & 44.2455 & 4.6742   &   4.9833   \\
			  PSNR &  17.4053 & 17.8779 & 23.0036   &   22.3444   \\
			\hline
		\end{tabular}
\end{table}

\begin{figure}[htbp]
	\centering
	\includegraphics[width=1 \textwidth]{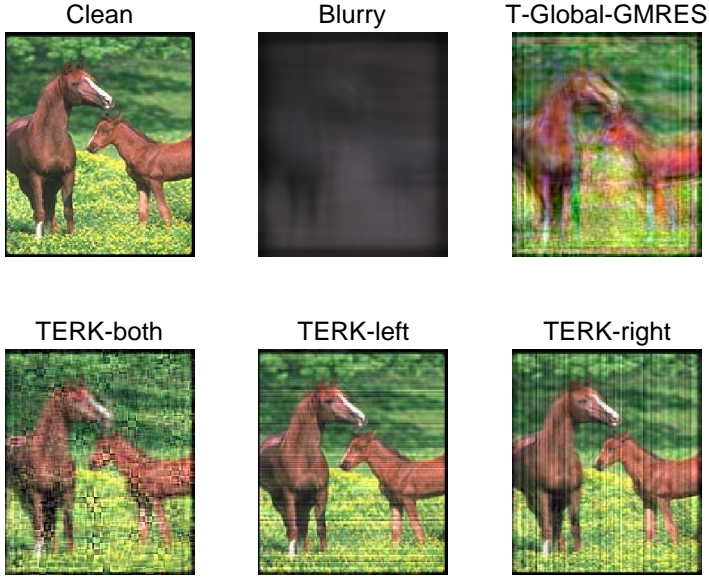}
	\caption{The original clean image, its corresponding blurry observation and the images recovered from the Tensor T-Global GMRES, TERK-both, TERK-left and TERK-right methods.}\label{fig5}
\end{figure}

We then compare the nonadaptive and adaptive TERK methods. Figure \ref{fig:1} shows that the nonadaptive TERK methods require more iteration steps and CPU time  than their corresponding adaptive methods. Again, we have 
that the adaptive sampling strategies can indeed accelerate the convergence of the nonadaptive ones.

\begin{figure}[t]
	
	\subfigure{\label{fig:1b}
		\includegraphics[width=0.5 \textwidth]{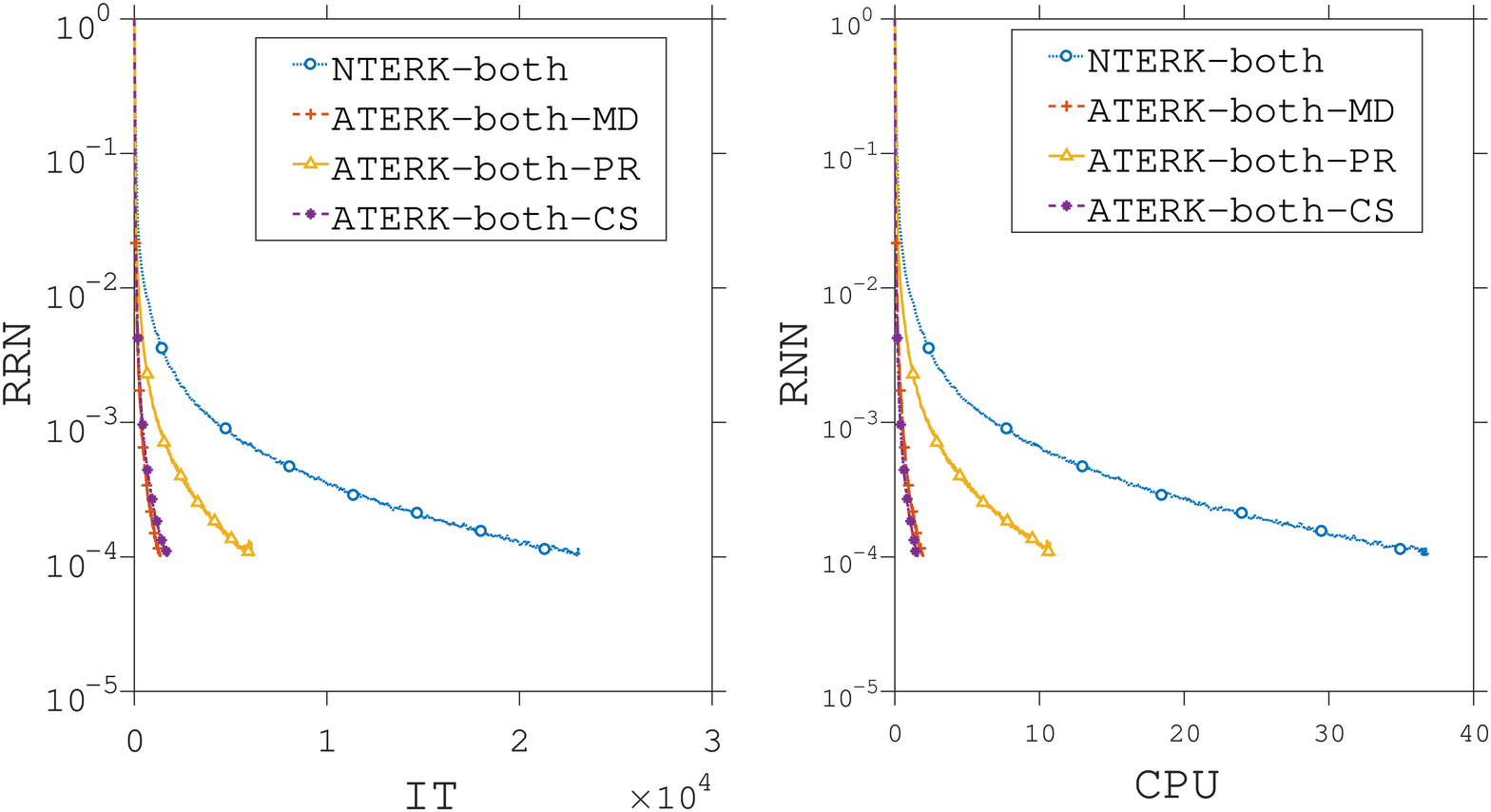}}
	\subfigure{\label{fig:1a}
		\includegraphics[width=0.5 \textwidth]{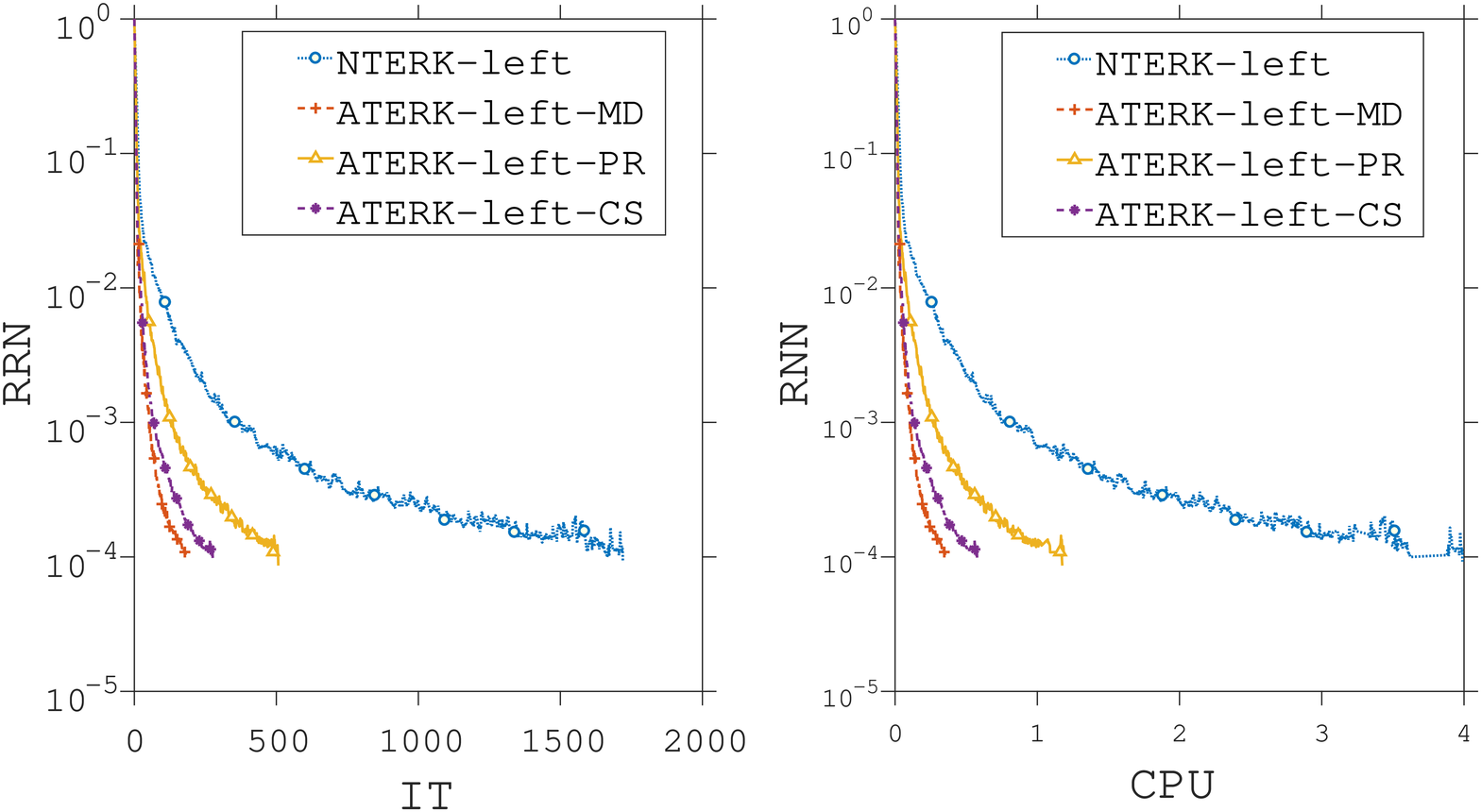}}
	
	\subfigure{\label{fig:1c}
		\includegraphics[width=0.5 \textwidth]{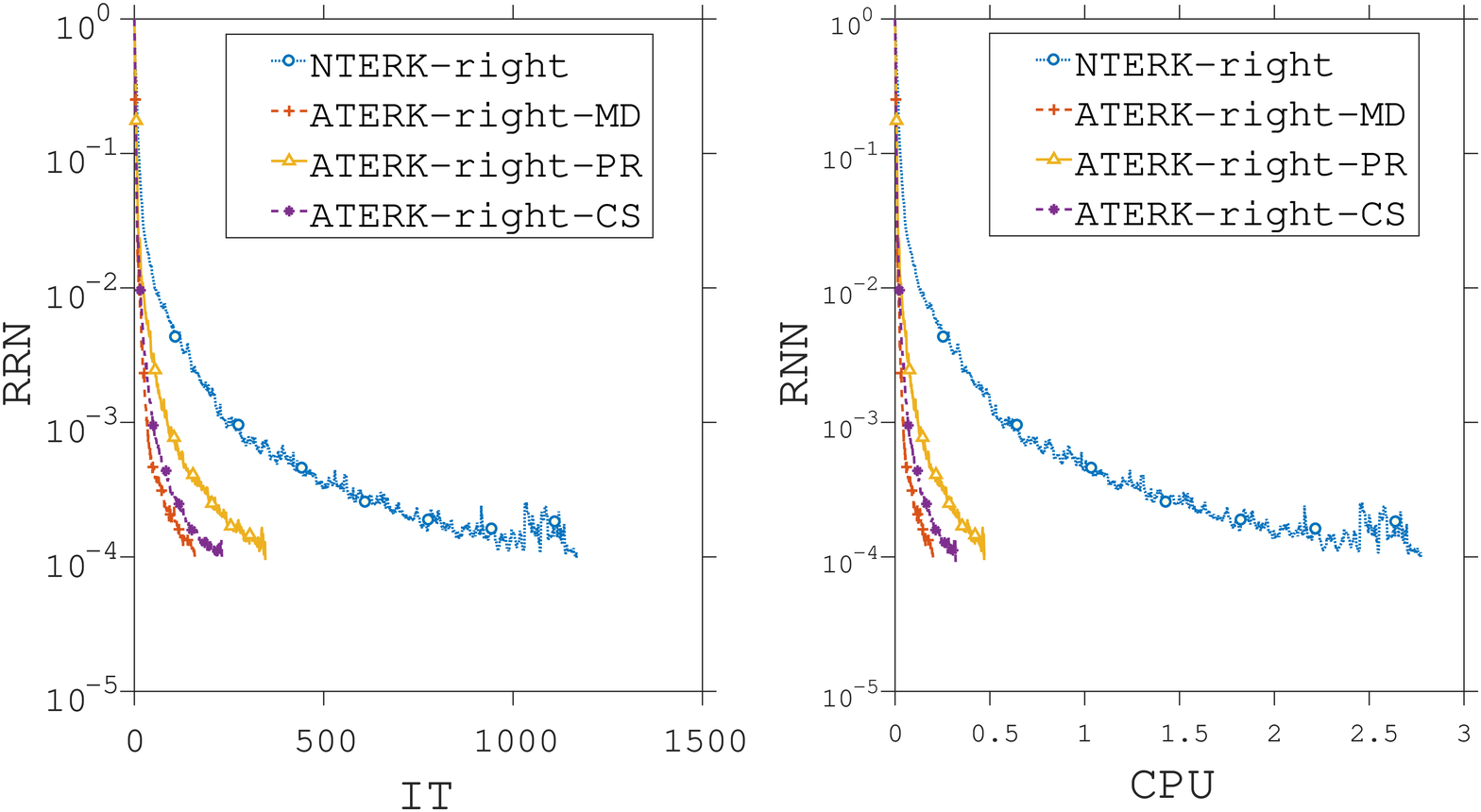}}
	
	\caption{RRN versus IT (left) and CPU (right) for the nonadaptive and adaptive TERK-both, TERK-left and TERK-right methods on real world color image data.}
	\label{fig:1}
\end{figure}
\end{example}
\section{Conclusion}
In this work, we propose the TESP method and its adaptive variants for linear tensor equations. We also discuss their efficient implementation in the Fourier domain. In theory, we analyze the convergence of all proposed methods in detail and provide the corresponding convergence factors. Numerical results show that our proposed methods are feasible and eﬀective for solving linear tensor equations and the adaptive sampling strategies can indeed accelerate the convergence of the nonadaptive ones.



\begin{appendices}
\section{Proofs for theoretical results}\label{app1}
\begin{proof}[Proof of Theorem~{\upshape\ref{thmTESP}}] \rm
According to the update formula (\ref{update}) and the fact $\mathcal{A}*\mathcal{X}^{\star}*\mathcal{B}=\mathcal{C}$, we have
\begin{align*}
	\mathcal{X}^{t+1}-\mathcal{X}^{\star}=&\mathcal{X}^{t}-\mathcal{M}^{-1}*\mathcal{A}^{T}*\mathcal{S}*(\mathcal{S}^{T}*\mathcal{A}*\mathcal{M}^{-1}*\mathcal{A}^{T}*\mathcal{S})^{\dag}*\mathcal{S}^{T}*(\mathcal{A}*\mathcal{X}^{t}*\mathcal{B}\\
	&-\mathcal{C})*\mathcal{V}*(\mathcal{V}^{T}*\mathcal{B}^{T}*\mathcal{N}^{-1}*\mathcal{B}*\mathcal{V})^{\dag}*\mathcal{V}^{T}*\mathcal{B}^{T}*\mathcal{N}^{-1}-\mathcal{X}^{\star}\\
	=&\mathcal{X}^{t}-\mathcal{X}^{\star}-\mathcal{M}^{-\frac{1}{2}}*\mathcal{Z}*\mathcal{M}^{\frac{1}{2}}*(\mathcal{X}^{t}-\mathcal{X}^{\star})*\mathcal{N}^{\frac{1}{2}}*\mathcal{W}*\mathcal{N}^{-\frac{1}{2}},
\end{align*}	
and multiply both $\mathcal{M}^{\frac{1}{2}}$ and $\mathcal{N}^{\frac{1}{2}}$ on its left and right sides to get
$$\mathcal{M}^{\frac{1}{2}}*(\mathcal{X}^{t+1}-\mathcal{X}^{\star})*\mathcal{N}^{\frac{1}{2}}=\mathcal{M}^{\frac{1}{2}}*(\mathcal{X}^{t}-\mathcal{X}^{\star})*\mathcal{N}^{\frac{1}{2}}-\mathcal{Z}*\mathcal{M}^{\frac{1}{2}}*(\mathcal{X}^{t}-\mathcal{X}^{\star})*\mathcal{N}^{\frac{1}{2}}*\mathcal{W}.$$
By setting $\Gamma^{t}=\mathcal{M}^{\frac{1}{2}}*(\mathcal{X}^{t}-\mathcal{X}^{\star})*\mathcal{N}^{\frac{1}{2}}$, the above equation can be rewritten as $$\Gamma^{t+1}=\Gamma^{t}-\mathcal{Z}*\Gamma^{t}*\mathcal{W}.$$
Taking the Frobenius norm on its both sides, we obtain
\begin{align*}
	\|\Gamma^{t+1}\|_{F}^{2}&=\|\Gamma^{t}-\mathcal{Z}*\Gamma^{t}*\mathcal{W}\|_{F}^{2}=\|\Gamma^{t}\|_{F}^{2}-\|\mathcal{Z}*\Gamma^{t}*\mathcal{W}\|_{F}^{2},
\end{align*}	
where the last equality follows from the Pythagorean theorem. Taking the expectation conditioned on $\mathcal{X}^{t}$ gives
\begin{align*}
	\mathbb{E}[\|\Gamma^{t+1}\|_{F}^{2}\mid\mathcal{X}^{t}]&=\|\Gamma^{t}\|_{F}^{2}-\mathbb{E}[\|\mathcal{Z}*\Gamma^{t}*\mathcal{W}\|_{F}^{2}].
\end{align*}	
Note that
\begin{align*}
	\mathbb{E}[\|\mathcal{Z}*\Gamma^{t}*\mathcal{W}\|_{F}^{2}]&=\mathbb{E}[\|\text{vec}_{t}(\mathcal{Z}*\Gamma^{t}*\mathcal{W})\|_{F}^{2}]
	=\mathbb{E}[\|(\mathcal{ W}^{ST}\otimes_{t}\mathcal{ Z})*\text{vec}_{t}(\Gamma^{t})\|_{F}^{2}]\\
	&=\mathbb{E}[\|\text{bcirc}(\mathcal{ W}^{ST}\otimes_{t}\mathcal{ Z})\text{unfold}(\text{vec}_{t}(\Gamma^{t}))\|_{F}^{2}]\\
	&=\left\langle\mathbb{E}[\text{bcirc}(\mathcal{ W}^{ST}\otimes_{t}\mathcal{ Z})]\text{unfold}(\text{vec}_{t}(\Gamma^{t})),\text{unfold}(\text{vec}_{t}(\Gamma^{t}))\right\rangle\\
	&\geq\lambda_{\min}(\mathbb{E}[\text{bcirc}(\mathcal{ W}^{ST}\otimes_{t}\mathcal{ Z})])\|\text{unfold}(\text{vec}_{t}(\Gamma^{t}))\|_{2}^{2}\\
	&=\lambda_{\min}(\mathbb{E}[\text{bcirc}(\mathcal{ W}\otimes_{t}\mathcal{ Z})])\|\Gamma^{t}\|_{F}^{2},
\end{align*}	
where the inequality is from that $\mathbb{E}[\text{bcirc}(\mathcal{ W}^{ST}\otimes_{t}\mathcal{ Z})]$ is symmetric positive deﬁnite with probability 1, which can be obtained by the assumption that $\mathbb{E}[\mathcal{Z}]$ and $\mathbb{E}[\mathcal{W}]$ are T-symmetric T-positive definite with probability $1$, and the last inequality is because
\begin{align*}
	&\lambda_{\min}(\mathbb{E}[\text{bcirc}(\mathcal{W}^{ST}\otimes_{t}\mathcal{Z})])=\min\limits_{k\in [l]}\lambda_{\min}(\mathbb{E}[\widehat{\mathcal{W}}_{(k)}^{T}\otimes\widehat{\mathcal{ Z}}_{(k)}])\\
 =&\min\limits_{k\in [l]}\lambda_{\min}(\mathbb{E}[\widehat{\mathcal{ W}}_{(k)}^{T}])\lambda_{\min}(\mathbb{E}[\widehat{\mathcal{Z}}_{(k)}])
	=\min\limits_{k\in [l]}\lambda_{\min}(\mathbb{E}[\widehat{\mathcal{ W}}_{(k)}])\lambda_{\min}(\mathbb{E}[\widehat{\mathcal{ Z}}_{(k)}])\\
 =&\min\limits_{k\in [l]}\lambda_{\min}(\mathbb{E}[\widehat{\mathcal{W}}_{(k)}\otimes\widehat{\mathcal{ Z}}_{(k)}])
	=\lambda_{\min}(\mathbb{E}[\text{bcirc}(\mathcal{W}\otimes_{t}\mathcal{Z})]).
\end{align*}
Therefore,
\begin{align*}
	\mathbb{E}[\|\Gamma^{t+1}\|_{F}^{2}\mid\mathcal{X}^{t}]&\leq(1-\lambda_{\min}(\mathbb{E}[\text{bcirc}(\mathcal{ W}\otimes_{t}\mathcal{ Z})]))\|\Gamma^{t}\|_{F}^{2},
\end{align*}	
that is,
\begin{align*}
	\mathbb{E}[\|\mathcal{X}^{t+1}-\mathcal{X}^{\star}\|_{F(\mathcal{M},\mathcal{ N})}^{2}\mid\mathcal{X}^{t}]&\leq\rho_{\text{TESP}}\|\mathcal{X}^{t}-\mathcal{X}^{\star}\|_{F(\mathcal{M},\mathcal{ N})}^{2},
\end{align*}	
where $\rho_{\text{TESP}}=1-\lambda_{\min}(\mathbb{E}[\text{\rm bcirc}(\mathcal{ W}\otimes_{t}\mathcal{ Z})])$.
Taking the  full  expectation and unrolling the recurrence give this theorem.
\end{proof}

\begin{proof}[Proof of Lemma~{\upshape\ref{defdelta}}]\rm
From $\text{vec}_{t}(\mathcal{X}^{0})\in \mathbf{Range}((\mathcal{N}^{-1})^{ST}*\mathcal{ B}^{C}\otimes_{t}\mathcal{M}^{-1}*\mathcal{A}^{T})$, we can get $\text{vec}_{t}(\mathcal{X}^{t}-\mathcal{X}^{\star})\in \mathbf{Range}((\mathcal{N}^{-1})^{ST}*\mathcal{ B}^{C}\otimes_{t}\mathcal{M}^{-1}*\mathcal{A}^{T})$, and then combining with
\begin{align*}
	f_{i,j}(\mathcal{X}^{t})&=\|\mathcal{Z}_{i}*\mathcal{M}^{\frac{1}{2}}*(\mathcal{X}^{t}-\mathcal{X}^{\star})*\mathcal{N}^{\frac{1}{2}}*\mathcal{W}_{j}\|_{F}^{2}\\
	&=\|((\mathcal{ W}_{j}^{ST}*(\mathcal{N}^{\frac{1}{2}})^{ST})\otimes_{t}(\mathcal{ Z}_{i}*\mathcal{M}^{\frac{1}{2}}))*\text{vec}_{t}(\mathcal{X}^{t}-\mathcal{X}^{\star})\|_{F}^{2}\\
	&=\|((\mathcal{N}^{\frac{1}{2}})^{ST}\otimes_{t} \mathcal{M}^{\frac{1}{2}})*\text{vec}_{t}(\mathcal{X}^{t}-\mathcal{X}^{\star})\|_{\mathcal{ W}_{j}^{ST}\otimes_{t}\mathcal{ Z}_{i}}^{2},
\end{align*}
we further obtain
\begin{align*}
&\frac{\mathop{\max}\limits_{i\in[q_{\mathcal{S}}], j\in[q_{\mathcal{V}}]}f_{i,j}(\mathcal{X}^{t})}{\|\mathcal{X}^{t}-\mathcal{X}^{\star}\|_{F(\mathcal{M},\mathcal{N})}^{2}}\\
=&\frac{\mathop{\max}\limits_{i\in[q_{\mathcal{S}}], j\in[q_{\mathcal{V}}]}\|((\mathcal{N}^{\frac{1}{2}})^{ST}\otimes_{t} \mathcal{M}^{\frac{1}{2}})*\text{vec}_{t}(\mathcal{X}^{t}-\mathcal{X}^{\star})\|_{\mathcal{ W}_{j}^{ST}\otimes_{t}\mathcal{ Z}_{i}}^{2}}{\|\text{vec}_{t}(\mathcal{X}^{t}-\mathcal{X}^{\star})\|_{\mathcal{N}^{ST}\otimes_{t} \mathcal{M}}^{2}}\\
\geq&\mathop{\min}_{\overrightarrow{\mathcal{Y}}\in \mathbf{Range}(((\mathcal{N}^{-1})^{ST}*\mathcal{ B}^{C})\otimes_{t}(\mathcal{M}^{-1}*\mathcal{A}^{T}))}\mathop{\max}\limits_{i\in[q_{\mathcal{ S}}], j\in[q_{\mathcal{ V}}]}
 \\
&\frac{\|((\mathcal{N}^{\frac{1}{2}})^{ST}\otimes_{t} \mathcal{M}^{\frac{1}{2}})*\overrightarrow{\mathcal{Y}}\|_{\mathcal{ W}_{j}^{ST}\otimes_{t}\mathcal{ Z}_{i}}^{2}}{\|\overrightarrow{\mathcal{Y}}\|_{\mathcal{N}^{ST}\otimes_{t} \mathcal{M}}^{2}}=\delta_{\infty}^{2}(\mathcal{M},\mathcal{ N},\boldsymbol{\mathcal{S}},\boldsymbol{\mathcal{V}}).
\end{align*}
Similarly, we have
\begin{align*}
&\frac{\mathbb{E}_{i\sim \mathbf{p}_{\mathcal{S}},j\sim \mathbf{p}_{\mathcal{V}}}[f_{i,j}(\mathcal{X}^{t})]}{\|\mathcal{X}^{t}-\mathcal{X}^{\star}\|_{F(\mathcal{M},\mathcal{N})}^{2}}\\
=&\frac{\mathbb{E}_{i\sim \mathbf{p}_{\mathcal{S}},j\sim \mathbf{p}_{\mathcal{V}}}[\|((\mathcal{N}^{\frac{1}{2}})^{ST}\otimes_{t} \mathcal{M}^{\frac{1}{2}})*\text{vec}_{t}(\mathcal{X}^{t}-\mathcal{X}^{\star})\|_{\mathcal{ W}_{j}^{ST}\otimes_{t}\mathcal{ Z}_{i}}^{2}]}{\|\text{vec}_{t}(\mathcal{X}^{t}-\mathcal{X}^{\star})\|_{\mathcal{N}^{ST}\otimes_{t} \mathcal{M}}^{2}}\\	\geq&\mathop{\min}_{\overrightarrow{\mathcal{Y}}\in \mathbf{Range}(((\mathcal{N}^{-1})^{ST}*\mathcal{ B}^{C})\otimes_{t}(\mathcal{M}^{-1}*\mathcal{A}^{T}))}
\\	
&\frac{\mathbb{E}_{i\sim \mathbf{p}_{\mathcal{S}},j\sim \mathbf{p}_{\mathcal{V}}}[\|((\mathcal{N}^{\frac{1}{2}})^{ST}\otimes_{t} \mathcal{M}^{\frac{1}{2}})*\overrightarrow{\mathcal{Y}}\|_{\mathcal{ W}_{j}^{ST}\otimes_{t}\mathcal{ Z}_{i}}^{2}]}{\|\overrightarrow{\mathcal{Y}}\|_{\mathcal{N}^{ST}\otimes_{t} \mathcal{M}}^{2}}\\
=&\mathop{\min}_{\overrightarrow{\mathcal{Y}}\in \mathbf{Range}(((\mathcal{N}^{-1})^{ST}*\mathcal{ B}^{C})\otimes_{t}(\mathcal{M}^{-1}*\mathcal{A}^{T}))}\\	
&\frac{\|((\mathcal{N}^{\frac{1}{2}})^{ST}\otimes_{t} \mathcal{M}^{\frac{1}{2}})*\overrightarrow{\mathcal{Y}}\|_{\mathbb{E}_{j\sim \mathbf{p}_{\mathcal{V}}}[\mathcal{ W}_{j}^{ST}] \otimes_{t} \mathbb{E}_{i\sim \mathbf{p}_{\mathcal{S}}}[\mathcal{ Z}_{i}]}^{2}}{\|\overrightarrow{\mathcal{Y}}\|_{\mathcal{N}^{ST}\otimes_{t} \mathcal{M}}^{2}}=\delta_{\mathbf{p}_{\mathcal{S}}, \mathbf{p}_{\mathcal{V}}}^{2}(\mathcal{M},\mathcal{ N},\boldsymbol{\mathcal{S}},\boldsymbol{\mathcal{V}}).
\end{align*}
\end{proof}

\begin{proof}[Proof of Lemma~{\upshape\ref{sizedelta}}]\rm
Since $\mathbb{E}_{i\sim \mathbf{p}_{\mathcal{ S}}}[\mathcal{ Z}_{i}]=\mathcal{M}^{-\frac{1}{2}}*\mathcal{A}^{T}*\mathbb{E}_{i\sim \mathbf{p}_{\mathcal{ S}}}[\mathcal{E}_{i}]*\mathcal{ A}*\mathcal{M}^{-\frac{1}{2}}$ and $\mathbb{E}_{j\sim \mathbf{p}_{\mathcal{ V}}}[\mathcal{ W}_{j}]=\mathcal{N}^{-\frac{1}{2}}*\mathcal{ B}*\mathbb{E}_{j\sim \mathbf{p_{\mathcal{ V}}}}[\mathcal{G}_{j}]*\mathcal{ B}^{T}*\mathcal{N}^{-\frac{1}{2}}$ are T-symmetric T-positive deﬁnite with probability 1, we have that
\begin{align*}
	&\mathbb{E}_{j\sim \mathbf{p}_{\mathcal{ V}}}[\mathcal{ W}_{j}^{ST}]\otimes_{t}\mathbb{E}_{i\sim \mathbf{p}_{\mathcal{ S}}}[\mathcal{ Z}_{i}]\\
=&\left((\mathcal{N}^{-\frac{1}{2}})^{ST}*\mathcal{ B}^{C}*\mathbf{E}_{j\sim \mathbf{p}_{\mathcal{ V}}}[\mathcal{G}_{j}^{ST}]*\mathcal{ B}^{ST}*(\mathcal{N}^{-\frac{1}{2}})^{ST}\right)
\\	
&\otimes_{t} \left(\mathcal{M}^{-\frac{1}{2}}*\mathcal{A}^{T}*\mathbf{E}_{i\sim \mathbf{p}_{\mathcal{ S}}}[\mathcal{E}_{i}]*\mathcal{ A}*\mathcal{M}^{-\frac{1}{2}}\right)\\
	=&\left(((\mathcal{N}^{-\frac{1}{2}})^{ST}*\mathcal{ B}^{C})\otimes_{t} (\mathcal{M}^{-\frac{1}{2}}*\mathcal{A}^{T})\right)*\left(\mathbf{E}_{j\sim \mathbf{p_{\mathcal{ V}}}}[\mathcal{G}_{j}^{ST}]\otimes \mathbf{E}_{i\sim \mathbf{p_{\mathcal{ S}}}}[\mathcal{E}_{i}]\right)
 \\	
&*\left((\mathcal{ B}^{ST}*(\mathcal{N}^{-\frac{1}{2}})^{ST})\otimes (\mathcal{ A}*\mathcal{M}^{-\frac{1}{2}})\right)
\end{align*}
is also T-symmetric T-positive deﬁnite with probability 1. Thus, we have
\begin{footnotesize}
\begin{align*}
\mathbf{Range}(((\mathcal{N}^{-\frac{1}{2}})^{ST}*\mathcal{ B}^{C})\otimes_{t} (\mathcal{M}^{-\frac{1}{2}}*\mathcal{A}^{T}))&=\mathbf{Range}\left(\mathbb{E}_{j\sim \mathbf{p}_{\mathcal{ V}}}[\mathcal{ W}_{j}^{ST}]\otimes_{t} \mathbb{E}_{i\sim \mathbf{p}_{\mathcal{ S}}}[\mathcal{ Z}_{i}]\right)=\mathbb{K}_{l}^{rs}.
\end{align*}
\end{footnotesize}
Therefore,
\begin{align*}
\delta_{\mathbf{p}_{\mathcal{S}}, \mathbf{p}_{\mathcal{V}}}^{2}(\mathcal{M},\mathcal{ N},\boldsymbol{\mathcal{S}},\boldsymbol{\mathcal{V}})=&\mathop{\min}_{\overrightarrow{\mathcal{Y}}\in \mathbf{Range}(((\mathcal{N}^{-1})^{ST}*\mathcal{ B}^{C})\otimes_{t}(\mathcal{M}^{-1}*\mathcal{A}^{T}))}\\
&\frac{\|((\mathcal{N}^{\frac{1}{2}})^{ST}\otimes_{t} \mathcal{M}^{\frac{1}{2}})*\overrightarrow{\mathcal{Y}}\|_{\mathbb{E}_{j\sim \mathbf{p}_{\mathcal{V}}}[\mathcal{ W}_{j}^{ST}] \otimes_{t} \mathbb{E}_{i\sim \mathbf{p}_{\mathcal{S}}}[\mathcal{ Z}_{i}]}^{2}}{\|\overrightarrow{\mathcal{Y}}\|_{\mathcal{N}^{ST}\otimes_{t} \mathcal{M}}^{2}}\\
=&\mathop{\min}_{\text{unfold}(((\mathcal{N}^{\frac{1}{2}})^{ST}\otimes_{t} \mathcal{M}^{\frac{1}{2}})*\overrightarrow{\mathcal{Y}})\in \mathbb{R}^{rsl}}\\&
\frac{\|\text{unfold}(((\mathcal{N}^{\frac{1}{2}})^{ST}\otimes_{t} \mathcal{M}^{\frac{1}{2}})*\overrightarrow{\mathcal{Y}})\|_{\text{bcirc}(\mathbb{E}_{j\sim \mathbf{p}_{\mathcal{V}}}[\mathcal{ W}_{j}^{ST}] \otimes_{t} \mathbb{E}_{i\sim \mathbf{p}_{\mathcal{S}}}[\mathcal{ Z}_{i}])}^{2}}{\|\text{unfold}(((\mathcal{N}^{\frac{1}{2}})^{ST}\otimes_{t} \mathcal{M}^{\frac{1}{2}})*\overrightarrow{\mathcal{Y}})\|_{2}^{2}}\\
	=&\lambda_{\min}\left(\text{bcirc}(\mathbb{E}_{j\sim \mathbf{p}_{\mathcal{V}}}[\mathcal{ W}_{j}^{ST}] \otimes_{t} \mathbb{E}_{i\sim \mathbf{p}_{\mathcal{S}}}[\mathcal{ Z}_{i}])\right)\\
	=&\lambda_{\min}\left(\mathbb{E}_{i\sim \mathbf{p}_{\mathcal{ S}},j\sim \mathbf{p}_{\mathcal{V}}}[\text{bcirc}(\mathcal{ W}_{j} \otimes_{t}\mathcal{ Z}_{i})]\right)> 0,
\end{align*}
and
\begin{align*}
\delta_{\mathbf{p}_{\mathcal{S}}, \mathbf{p}_{\mathcal{ V}}}^{2}(\mathcal{M},\mathcal{ N},\boldsymbol{\mathcal{S}},\boldsymbol{\mathcal{V}})&=\mathop{\min}_{\overrightarrow{\mathcal{Y}}\in \mathbf{Range}(((\mathcal{N}^{-1})^{ST}*\mathcal{ B}^{C})\otimes_{t}(\mathcal{M}^{-1}*\mathcal{A}^{T}))}\\
&\quad\frac{\|((\mathcal{N}^{\frac{1}{2}})^{ST}\otimes_{t} \mathcal{M}^{\frac{1}{2}})*\overrightarrow{\mathcal{Y}}\|_{\mathbb{E}_{j\sim \mathbf{p}_{\mathcal{V}}}[\mathcal{ W}_{j}^{ST}] \otimes_{t} \mathbb{E}_{i\sim \mathbf{p}_{\mathcal{S}}}[\mathcal{ Z}_{i}]}^{2}}{\|\overrightarrow{\mathcal{Y}}\|_{\mathcal{N}^{ST}\otimes_{t} \mathcal{M}}^{2}}\\
&=\mathop{\min}_{\overrightarrow{\mathcal{Y}}\in \mathbf{Range}(((\mathcal{N}^{-1})^{ST}*\mathcal{ B}^{C})\otimes_{t}(\mathcal{M}^{-1}*\mathcal{A}^{T}))}\\
&\quad\frac{\mathbb{E}_{i\sim \mathbf{p}_{\mathcal{S}}, j\sim \mathbf{p}_{\mathcal{V}}}[\|((\mathcal{N}^{\frac{1}{2}})^{ST}\otimes_{t} \mathcal{M}^{\frac{1}{2}})*\overrightarrow{\mathcal{Y}}\|_{\mathcal{ W}_{j}^{ST} \otimes_{t} \mathcal{ Z}_{i}}^{2}]}{\|\overrightarrow{\mathcal{Y}}\|_{\mathcal{N}^{ST}\otimes_{t} \mathcal{M}}^{2}}\\
&\leq\mathop{\min}_{\overrightarrow{\mathcal{Y}}\in \mathbf{Range}(((\mathcal{N}^{-1})^{ST}*\mathcal{ B}^{C})\otimes_{t}(\mathcal{M}^{-1}*\mathcal{A}^{T}))}\mathop{\max}_{i\in[q_{\mathcal{S}}], j\in[q_{\mathcal{V}}]}\\
&\quad\frac{\|((\mathcal{N}^{\frac{1}{2}})^{ST}\otimes_{t} \mathcal{M}^{\frac{1}{2}})*\overrightarrow{\mathcal{Y}}\|_{\mathcal{ W}_{j}^{ST}\otimes_{t}\mathcal{ Z}_{i}}^{2}}{\|\overrightarrow{\mathcal{Y}}\|_{\mathcal{N}^{ST}\otimes_{t} \mathcal{M}}^{2}}\\
	&=\delta_{\infty}^{2}(\mathcal{M},\mathcal{ N},\boldsymbol{\mathcal{S}},\boldsymbol{\mathcal{V}}).
\end{align*}
Finally, from the fact that the tubal matrices $ \mathcal{ Z}_{i}$ and $\mathcal{ W}_{j}$ are orthogonal projectors, we get
\begin{align*}
	\delta_{\infty}^{2}(\mathcal{M},\mathcal{ N},\boldsymbol{\mathcal{S}},\boldsymbol{\mathcal{V}})=& \mathop{\min}_{\overrightarrow{\mathcal{Y}}\in \mathbf{Range}(((\mathcal{N}^{-1})^{ST}*\mathcal{ B}^{C})\otimes_{t}(\mathcal{M}^{-1}*\mathcal{A}^{T}))}\mathop{\max}_{i\in[q_{\mathcal{S}}], j\in[q_{\mathcal{V}}]}\\
 &\frac{\|((\mathcal{N}^{\frac{1}{2}})^{ST}\otimes_{t} \mathcal{M}^{\frac{1}{2}})*\overrightarrow{\mathcal{Y}}\|_{\mathcal{ W}_{j}^{ST}\otimes_{t}\mathcal{ Z}_{i}}^{2}}{\|\overrightarrow{\mathcal{Y}}\|_{\mathcal{N}^{ST}\otimes_{t} \mathcal{M}}^{2}}\\
	\leq&\mathop{\max}_{i\in[q_{\mathcal{S}}], j\in[q_{\mathcal{V}}]}\frac{\|((\mathcal{N}^{\frac{1}{2}})^{ST}\otimes_{t} \mathcal{M}^{\frac{1}{2}})*\overrightarrow{\mathcal{Y}}\|_{\mathcal{ W}_{j}^{ST}\otimes_{t}\mathcal{ Z}_{i}}^{2}}{\|\overrightarrow{\mathcal{Y}}\|_{\mathcal{N}^{ST}\otimes_{t} \mathcal{M}}^{2}}\\
	\leq&\mathop{\max}_{i\in[q_{\mathcal{S}}], j\in[q_{\mathcal{V}}]}\frac{\|((\mathcal{N}^{\frac{1}{2}})^{ST}\otimes_{t} \mathcal{M}^{\frac{1}{2}})*\overrightarrow{\mathcal{Y}}\|_{F}^{2}}{\|((\mathcal{N}^{\frac{1}{2}})^{ST}\otimes_{t} \mathcal{M}^{\frac{1}{2}})*\overrightarrow{\mathcal{Y}}\|_{F}^{2}}=1.
\end{align*}
\end{proof}

\begin{proof}[Proof of Theorem~{\upshape\ref{thmNTESP}}]\rm
 Using {\ref{expectation}} and (\ref{deex}), we have
\begin{align*}
	&\mathbb{E}[\|\mathcal{X}^{t+1}-X^{\star}\|_{F(\mathcal{M},\mathcal{ N})}^{2}\mid\mathcal{X}^{t}]=\|\mathcal{X}^{t}-\mathcal{X}^{\star}\|_{F(\mathcal{ M},\mathcal{ N})}^{2}-\mathbb{E}_{i\sim \mathbf{p}_{\mathcal{ S}}, j\sim \mathbf{p}_{\mathcal{ V}}}[f_{i,j}(\mathcal{X}^{t})]\\
	\leq&\|\mathcal{ X}^{t}-\mathcal{ X}^{\star}\|_{F(\mathcal{M},\mathcal{ N})}^{2}-\delta_{\mathbf{p}_{\mathcal{ S}}, \mathbf{p}_{\mathcal{ V}}}^{2}(\mathcal{M},\mathcal{ N},\boldsymbol{\mathcal{S}},\boldsymbol{\mathcal{V}})\|\mathcal{ X}^{t}-\mathcal{ X}^{\star}\|_{F(\mathcal{M},\mathcal{ N})}^{2}\\
	=&(1-\delta_{\mathbf{p}_{\mathcal{ S}}, \mathbf{p}_{\mathcal{ V}}}^{2}(\mathcal{M},\mathcal{ N},\boldsymbol{\mathcal{S}},\boldsymbol{\mathcal{V}}))\|\mathcal{ X}^{t}-\mathcal{ X}^{\star}\|_{F(\mathcal{M},\mathcal{ N})}^{2}.
\end{align*}
Taking the  full  expectation and unrolling the recurrence give this theorem.
\end{proof}

\begin{proof}[Proof of Theorem~{\upshape\ref{thmATESPMD}}]\rm
	 Combining (\ref{secloss}) and (\ref{dein}) yields
\begin{align*}
	\|\mathcal{ X}^{t+1}-\mathcal{ X}^{\star}\|_{F(\mathcal{M},\mathcal{ N})}^{2}&=\|\mathcal{ X}^{t}-\mathcal{ X}^{\star}\|_{F(\mathcal{M},\mathcal{ N})}^{2}-\mathop{\max}\limits_{i\in[q_{\mathcal{ S}}], j\in[q_{\mathcal{ V}}]}f_{i,j}(\mathcal{ X}^{t})\\
	&\leq\|\mathcal{ X}^{t}-\mathcal{ X}^{\star}\|_{F(\mathcal{M},\mathcal{ N})}^{2}-\delta_{\infty}^{2}(\mathcal{M},\mathcal{ N},\boldsymbol{\mathcal{S}},\boldsymbol{\mathcal{V}})\|\mathcal{ X}^{t}-\mathcal{ X}^{\star}\|_{F(\mathcal{M},\mathcal{ N})}^{2}\\
	&=(1-\delta_{\infty}^{2}(\mathcal{M},\mathcal{ N},\boldsymbol{\mathcal{S}},\boldsymbol{\mathcal{V}}))\|\mathcal{ X}^{t}-\mathcal{ X}^{\star}\|_{F(\mathcal{M},\mathcal{ N})}^{2}.
\end{align*}
Taking the  full  expectation and unrolling the recurrence give this theorem.
\end{proof}

\begin{proof}[Proof of Theorem~{\upshape\ref{thmATESPRP}}]\rm
Since $i\sim \mathbf{u}_{\mathcal{ S}}$ and $j\sim \mathbf{u}_{\mathcal{ V}}$, we have  $(i,j)\sim \mathbf{u}_{\mathcal{ S},\mathcal{ V}}$, where $u_{i,j}=\frac{1}{q_{\mathcal{ S}}q_{\mathcal{ V}}}$ for $i=1,\cdots,q_{\mathcal{ S}}$ and $j=1,\cdots,q_{\mathcal{ V}}$. Hence,
\begin{align*}
	\textbf{Var}_{(i,j)\sim\mathbf{u}_{\mathcal{ S},\mathcal{ V}}}[f_{i,j}(\mathcal{ X}^{t})]&=\mathbb{E}_{(i,j)\sim\mathbf{u}_{\mathcal{ S},\mathcal{ V}}}[f_{i,j}(\mathcal{ X}^{t})^{2}]-	\mathbb{E}_{(i,j)\sim\mathbf{u}_{\mathcal{ S},\mathcal{ V}}}[f_{i,j}(\mathcal{ X}^{t})]^{2}\\
	&=\frac{1}{q_{\mathcal{ S}}q_{\mathcal{ V}}}\sum_{i=1}^{q_{\mathcal{ S}}}\sum_{j=1}^{q_{\mathcal{ V}}}f_{i,j}(\mathcal{ X}^{t})^{2}-\left(\frac{1}{q_{\mathcal{ S}}q_{\mathcal{ V}}}\sum_{i=1}^{q_{\mathcal{ S}}}\sum_{j=1}^{q_{\mathcal{ V}}}f_{i,j}(\mathcal{ X}^{t})\right)^{2}\\
	&=\frac{1}{q_{\mathcal{ S}}q_{\mathcal{ V}}}\sum_{i=1}^{q_{\mathcal{ S}}}\sum_{j=1}^{q_{\mathcal{ V}}}f_{i,j}(\mathcal{ X}^{t})^{2}-\frac{1}{q_{\mathcal{ S}}^{2}q_{\mathcal{ V}}^{2}}\left(\sum_{i=1}^{q_{\mathcal{ S}}}\sum_{j=1}^{q_{\mathcal{ V}}}f_{i,j}(\mathcal{ X}^{t})\right)^{2},
\end{align*}
which together with (\ref{deex}) and the definition of $\mathbf{p}_{\mathcal{ S},\mathcal{ V}}^{t}$ in the second case of Algorithm \ref{ATESP} implies
\begin{align*}
	&\mathbb{E}_{(i,j)\sim\mathbf{p}_{\mathcal{ S},\mathcal{ V}}^{t}}[f_{i,j}(\mathcal{ X}^{t})]
   =\sum_{i=1}^{q_{\mathcal{ S}}}\sum_{j=1}^{q_{\mathcal{ V}}}\frac{f_{i,j}(\mathcal{ X}^{t})}{\sum_{i=1}^{q_{\mathcal{ S}}}\sum_{j=1}^{q_{\mathcal{ V}}}f_{i,j}(\mathcal{ X}^{t})}f_{i,j}(\mathcal{ X}^{t})\\
	=&\frac{1}{\sum_{i=1}^{q_{\mathcal{ S}}}\sum_{j=1}^{q_{\mathcal{ V}}}f_{i,j}(\mathcal{ X}^{t})}\left(q_{\mathcal{ S}}q_{\mathcal{ V}}\text{Var}_{(i,j)\sim\mathbf{u}_{\mathcal{S},\mathcal{V}}}[f_{i,j}(\mathcal{ X}^{t})]+\frac{1}{q_{\mathcal{ S}}q_{\mathcal{ V}}}\left(\sum_{i=1}^{q_{\mathcal{ S}}}\sum_{j=1}^{q_{\mathcal{ V}}}f_{i,j}(\mathcal{ X}^{t})\right)^{2}\right)\\
	=&\left(1+q_{\mathcal{ S}}^{2}q_{\mathcal{ V}}^{2}\textbf{Var}_{(i,j)\sim\mathbf{u}_{\mathcal{ S},\mathcal{ V}}}\left[\frac{f_{i,j}(\mathcal{ X}^{t})}{\sum_{i=1}^{q_{\mathcal{ S}}}\sum_{j=1}^{q_{\mathcal{ V}}}f_{i,j}(\mathcal{ X}^{t})}\right]\right)\frac{1}{q_{\mathcal{ S}}q_{\mathcal{ V}}}\sum_{i=1}^{q_{\mathcal{ S}}}\sum_{j=1}^{q_{\mathcal{ V}}}f_{i,j}(\mathcal{ X}^{t})\\
	=&\left(1+q_{\mathcal{ S}}^{2}q_{\mathcal{ V}}^{2}\textbf{Var}_{(i,j)\sim\mathbf{u}_{\mathcal{ S},\mathcal{ V}}}\left[p_{i,j}^{t}\right]\right)\mathbb{E}_{(i,j)\sim\mathbf{u}_{\mathcal{ S},\mathcal{ V}}}[f_{i,j}(\mathcal{ X}^{t})]\\
	\geq&\left(1+q_{\mathcal{ S}}^{2}q_{\mathcal{ V}}^{2}\textbf{Var}_{i\sim\mathbf{u}_{\mathcal{ S}}, j\sim\mathbf{u}_{\mathcal{ V}}}\left[p_{i,j}^{t}\right]\right)\delta_{\mathbf{u}_{\mathcal{ S}}, \mathbf{u}_{\mathcal{ V}}}^{2}(\mathcal{M},\mathcal{ N},\boldsymbol{\mathcal{S}},\boldsymbol{\mathcal{V}})\|\mathcal{ X}^{t}-\mathcal{ X}^{\star}\|_{F(\mathcal{ M},\mathcal{ N})}^{2}.
\end{align*}
And then substitute it into (\ref{expectation}), we have
$$\mathbb{E}\left[\|\mathcal{ X}^{t+1}-\mathcal{ X}^{\star}\|_{F(\mathcal{ M},\mathcal{ N})}^{2}\mid\mathcal{ X}^{t}\right]\leq\rho_{\text{ATESP-PR}}\|\mathcal{ X}^{t}-\mathcal{ X}^{\star}\|_{F(\mathcal{ M},\mathcal{ N})}^{2},$$
where $\rho_{\text{ATESP-PR}}=1-(1+q_{\mathcal{ S}}^{2}q_{\mathcal{ V}}^{2}\mathbf{Var}_{i\sim \mathbf{u}_{\mathcal{ S}}, j\sim \mathbf{u}_{\mathcal{ V}}}[p_{i,j}^{t}])\delta_{\mathbf{u}_{\mathcal{ S}}, \mathbf{u}_{\mathcal{ V}}}^{2}(\mathcal{M},\mathcal{ N},\boldsymbol{\mathcal{S}},\boldsymbol{\mathcal{V}})$. Next, we shall derive a sharper bound for $\text{Var}_{i\sim\mathbf{u}_{\mathcal{ S}}, j\sim\mathbf{u}_{V}}\left[p_{i,j}^{t}\right]$. Since for any $(i,j)\in\Omega_{t}$, we have $f_{i,j}(\mathcal{X}^{t})=0$ 
which implies $p_{i,j}^{t}=0$, and hence
\begin{align*}
	\textbf{Var}_{(i,j)\sim\mathbf{u}_{\mathcal{ S},\mathcal{ V}}}\left[p_{i,j}^{t}\right]&=\frac{1}{q_{\mathcal{S}}q_{\mathcal{V}}}\sum_{i=1}^{q_{\mathcal{S}}}\sum_{j=1}^{q_{\mathcal{V}}}\left(p_{i,j}^{t}-\frac{1}{q_{\mathcal{S}}q_{\mathcal{V}}}\sum_{s=1}^{q_{\mathcal{\mathcal{S}}}}\sum_{r=1}^{q_{\mathcal{V}}}p_{s,r}^{t}\right)^{2}\\
	&=\frac{1}{q_{\mathcal{S}}q_{\mathcal{V}}}\sum_{i=1}^{q_{\mathcal{S}}}\sum_{j=1}^{q_{\mathcal{V}}}\left(p_{i,j}^{t}-\frac{1}{q_{\mathcal{S}}q_{\mathcal{V}}}\right)^{2}\\
	&\geq \frac{1}{q_{\mathcal{S}}q_{\mathcal{V}}}\sum_{(i,j)\in\Omega_{t}}\left(p_{i,j}^{t}-\frac{1}{q_{\mathcal{S}}q_{\mathcal{V}}}\right)^{2}= \frac{1}{q_{\mathcal{S}}^{3}q_{\mathcal{V}}^{3}}\vert\Omega_{t}\vert.
\end{align*}
Therefore, we get
$$\mathbf{E}\left[\|\mathcal{X}^{t+1}-\mathcal{X}^{\star}\|_{F(\mathcal{M},\mathcal{ N})}^{2}\mid\mathcal{X}^{t}\right]\leq\rho_{t}\|\mathcal{X}^{t}-\mathcal{X}^{\star}\|_{F(\mathcal{M},\mathcal{ N})}^{2}.$$
where $\rho_{t}=1-\left(1+ \frac{\vert\Omega_{t}\vert}{q_{\mathcal{S}}q_{\mathcal{V}}}\right)\delta_{\mathbf{u}_{\mathcal{ S}}, \mathbf{u}_{\mathcal{ V}}}^{2}(\mathcal{M},\mathcal{ N},\boldsymbol{\mathcal{S}},\boldsymbol{\mathcal{V}})$. Taking the  full  expectation and unrolling the recurrence give this theorem.
\end{proof}

\begin{proof}[Proof of Theorem~{\upshape\ref{thmATESPCS}}]\rm
Due to $$\mathop{\max}\limits_{i\in[q_{\mathcal{S}}], j\in[q_{\mathcal{V}}]}f_{i,j}(\mathcal{X}^{k})\geq\mathbb{E}_{i\sim \mathbf{p}_{\mathcal{S}},j\sim \mathbf{p_{\mathcal{V}}}}[f_{i,j}(\mathcal{X}^{t})],$$
we know that $\mathfrak{W}_{t}$ is not empty and $\mathop{\max}\limits_{i\in[q_{\mathcal{S}}], j\in[q_{\mathcal{V}}]}f_{i,j}(\mathcal{X}^{t})\in\mathfrak{W}_{t}$. From the deﬁnition of $\mathbf{p}_{\mathcal{ S},\mathcal{ V}}^{t}$, we have $p_{i,j}^{t}=0$ for all $(i,j)\notin \mathfrak{W}_{t}$, and thus
\begin{align*}
	\mathbb{E}[\|\mathcal{X}^{t+1}-\mathcal{X}^{\star}\|_{F(\mathcal{M},\mathcal{ N})}^{2}\mid\mathcal{X}^{t}]&=\|\mathcal{X}^{t}-\mathcal{X}^{\star}\|_{F(\mathcal{M},\mathcal{ N})}^{2}-\mathbb{E}_{(i,j)\sim \mathbf{p}_{\mathcal{ S},\mathcal{ V}}^{t}}[f_{i,j}(\mathcal{X}^{t})]\\
	&=\|\mathcal{X}^{t}-\mathcal{X}^{\star}\|_{F(\mathcal{M},\mathcal{ N})}^{2}-\sum_{(i,j)\in\mathfrak{W}_{t}}p_{i,j}^{t}f_{i,j}(\mathcal{X}^{t}).
\end{align*}
Note that
\begin{align*}	&\sum_{(i,j)\in\mathfrak{W}_{t}}p_{i,j}^{t}f_{i,j}(\mathcal{X}^{t})\\
\geq&\sum_{(i,j)\in\mathfrak{W}_{t}}\left(\theta\mathop{\max}\limits_{i\in[q_{\mathcal{S}}], j\in[q_{\mathcal{V}}]}f_{i,j}(\mathcal{X}^{t})+(1-\theta)\mathbb{E}_{i\sim \mathbf{p}_{\mathcal{S}},j\sim \mathbf{p}_{\mathcal{V}}}[f_{i,j}(\mathcal{X}^{t})]\right)p_{i,j}^{t}\\
=&\theta\mathop{\max}\limits_{i\in[q_{\mathcal{S}}], j\in[q_{\mathcal{V}}]}f_{i,j}(\mathcal{X}^{t})+(1-\theta)\mathbb{E}_{i\sim \mathbf{p}_{\mathcal{S}},j\sim \mathbf{p}_{\mathcal{V}}}[f_{i,j}(\mathcal{X}^{t})]\\
\geq&\left(\theta\delta_{\infty}^{2}(\mathcal{M},\mathcal{ N},\boldsymbol{\mathcal{S}},\boldsymbol{\mathcal{V}})+(1-\theta)\delta_{\mathbf{p}_{\mathcal{S}}, \mathbf{p}_{\mathcal{V}}}^{2}(\mathcal{M},\mathcal{ N},\boldsymbol{\mathcal{S}},\boldsymbol{\mathcal{V}})\right)\|\mathcal{X}^{t}-\mathcal{X}^{\star}\|_{F(\mathcal{M},\mathcal{ N})}^{2}.
\end{align*}
Hence,
\begin{align*}
	\mathbb{E}[\|\mathcal{X}^{t+1}-\mathcal{X}^{\star}\|_{F(\mathcal{M},\mathcal{ N})}^{2}\mid\mathcal{X}^{t}]\leq\rho_{\text{ATESP-CS}}\|\mathcal{X}^{t}-\mathcal{X}^{\star}\|_{F(\mathcal{M},\mathcal{ N})}^{2}.
\end{align*}
where $\rho_{\text{ATESP-CS}}=1-\theta\delta_{\infty}^{2}(\mathcal{\mathcal{M}},\mathcal{\mathcal{N}},\boldsymbol{\mathcal{S}},\boldsymbol{\mathcal{V}})-(1-\theta)\delta_{\mathbf{p}_{\mathcal{S}}, \mathbf{p}_{\mathcal{V}}}^{2}(\mathcal{\mathcal{M}},\mathcal{ N},\boldsymbol{\mathcal{S}},\boldsymbol{\mathcal{V}})$. Taking the  full  expectation and unrolling the recurrence give this theorem.
\end{proof}

\begin{proof}[Proof of Theorem~{\upshape\ref{thmTESPF}}]\rm
From the assumption that $\mathbb{E}[\text{ bdiag}(\widehat{\mathcal{Z}})]$ and $\mathbb{E}[\text{ bdiag}(\widehat{\mathcal{W}})]$ are Hermitian positive definite with probability $1$, we have that $\mathbb{E}[\mathcal{Z}]$ and $\mathbb{E}[\mathcal{W}]$ are T-symmetric T-positive definite with probability $1$.
Moreover,
\begin{align*}
	\lambda_{\min}(\mathbb{E}[\text{bcirc}(\mathcal{ W}\otimes_{t}\mathcal{ Z})])
	&=\mathop{\min}_{k\in[l]}\lambda_{\min}(\mathbb{E}[\widehat{\mathcal{W}}_{(k)}\otimes\widehat{\mathcal{Z}}_{(k)}])\\
	&=\mathop{\min}_{k\in[l]}\lambda_{\min}(\mathbb{E}[(\widehat{\mathcal{W}}_{(k)})\lambda_{\min}(\mathbb{E}[\widehat{\mathcal{Z}}_{(k)}]).
\end{align*}
Then, we conclude that
\begin{align*}
	&\mathbb{E}\left[\|\mathcal{X}^{t+1}-\mathcal{X}^{\star}\|_{F(\mathcal{M},\mathcal{ N})}^{2}\mid\mathcal{X}^{t}\right]\leq(1-	\lambda_{\min}(\mathbb{E}[\text{bcirc}(\mathcal{ W}\otimes_{t}\mathcal{ Z})]))\|\mathcal{X}^{t}-\mathcal{X}^{\star}\|_{F(\mathcal{M},\mathcal{ N})}^{2}\\
	=&\left(1-\mathop{\min}_{k\in[l]}\lambda_{\min}(\mathbb{E}[(\widehat{\mathcal{W}}_{(k)})])\lambda_{\min}(\mathbb{E}[\widehat{\mathcal{Z}}_{(k)}])\right)\|\mathcal{X}^{t}-\mathcal{X}^{\star}\|_{F(\mathcal{M},\mathcal{ N})}^{2}.
\end{align*}
Taking the  full  expectation and unrolling the recurrence give this theorem.
\end{proof}

\section{Fast version of the ATESP-PR method}\label{app2}
\begin{breakablealgorithm}
\caption{Fast ATESP-PR method in Fourier domain}\label{FATESPF}
\begin{algorithmic}[1]
 \State \textbf{Input:} $\mathcal{X}^{0}\in \mathbb{K}^{r\times s}_{l}$, $\mathcal{A}\in \mathbb{K}^{m\times r }_{l}$, $\mathcal{B}\in \mathbb{K}^{s\times n}_{l}$, and $\mathcal{C}\in \mathbb{K}^{m\times n }_{l}$
 \State \textbf{Parameters:} two finite sets of sketching tubal matrices $\boldsymbol{\mathcal{S}}=[\mathcal{S}_{1},\cdots,\mathcal{S}_{q_{\mathcal{ S}}}]$ and $\boldsymbol{\mathcal{V}}=[\mathcal{V}_{1},\cdots,\mathcal{V}_{q_{\mathcal{ V}}}]$, T-symmetric T-positive definite tubal matrix $\mathcal{M}\in \mathbb{K}^{r\times r}_{l}$ and $\mathcal{N}\in \mathbb{K}^{s\times s}_{l}$
		\State $\widehat{\mathcal{X}}^{0}= \texttt{fft}(\mathcal{X}^{0},[~],3)$, $\widehat{\mathcal{A}}=
		\texttt{fft}(\mathcal{A},[~],3)$, $\widehat{\mathcal{B}}=\texttt{fft}(\mathcal{B},[~],3)$, $\widehat{\mathcal{C}}=\texttt{fft}(\mathcal{C},[~],3)$,  $\widehat{\mathcal{M}}=\texttt{fft}(\mathcal{M},[~],3)$, $\widehat{\mathcal{N}}= \texttt{fft}(\mathcal{N},[~],3)$, $\widehat{\mathcal{S}_{i}}= \texttt{fft}(\mathcal{S}_{i},[~],3)$ for $i=1,\cdots,q_{\mathcal{ S}}$, $\widehat{\mathcal{V}_{i}}= \texttt{fft}(\mathcal{V}_{i},[~],3)$ for $i=1,\cdots,q_{\mathcal{ V}}$ %
		\For{$k=1,\cdots,\lceil\frac{l+1}{2}\rceil$} 
		\State Compute $(\widehat{\mathcal{C}}_{i})_{(k)}=\text{Cholesky}\left(\left((\widehat{\mathcal{S}}_{i})_{(k)}^{H}\widehat{\mathcal{A}}_{(k)}\widehat{\mathcal{Q}}^{-1}_{(k)}\widehat{\mathcal{A}}_{(k)}^{H}(\widehat{\mathcal{S}_{i}})_{(k)}\right)^{\dag}\right)$  and $(\widehat{\mathcal{D}}_{j})_{(k)}=\text{Cholesky}\left(\left((\widehat{\mathcal{V}}_{j})_{(k)}^{H}\widehat{\mathcal{B}}_{(k)}^{H}\widehat{\mathcal{R}}^{-1}_{(k)}\widehat{\mathcal{B}}_{(k)}(\widehat{\mathcal{V}_{j}})_{(k)}\right)^{\dag}\right)$ for $i=1,\cdots,q_{\mathcal{ S}}$ and  $j=1,\cdots,q_{\mathcal{ V}}$
		\State Compute
$\widehat{\mathcal{M}}^{-1}_{(k)}\widehat{\mathcal{A}}_{(k)}^{H}(\widehat{\mathcal{S}}_{i})_{(k)}(\widehat{\mathcal{C}}_{i})_{(k)}$ and $(\widehat{\mathcal{D}}_{j})^{H}_{(k)}(\widehat{\mathcal{V}}_{j})_{(k)}^{H}\widehat{\mathcal{B}}^{H}_{(k)}\widehat{\mathcal{N}}^{-1}_{(k)}$ for $i=1,\cdots,q_{\mathcal{ S}}$ and $j=1,\cdots,q_{\mathcal{ V}}$
		\State Compute $(\widehat{\mathcal{C}}_{i})_{(k)}^{H}(\widehat{\mathcal{S}_{i}})_{(k)}^{H}\widehat{\mathcal{A}}_{(k)}\widehat{\mathcal{M}}^{-1}_{(k)}\widehat{\mathcal{A}}_{(k)}^{H}(\widehat{\mathcal{S}}_{v})_{(k)}(\widehat{\mathcal{C}}_{v})_{(k)}$ for $i,v=1,\cdots,q_{\mathcal{ S}}$ and $(\widehat{\mathcal{D}}_{j})_{(k)}^{H}(\widehat{\mathcal{V}_{j}})_{(k)}^{H}\widehat{\mathcal{B}}_{(k)}^{H}\widehat{\mathcal{N}}^{-1}_{(k)}\widehat{\mathcal{B}}_{(k)}(\widehat{\mathcal{V}}_{w})_{(k)}(\widehat{\mathcal{D}}_{w})_{(k)}$ for $j,w=1,\cdots,q_{\mathcal{ V}}$
		\State Initialize $(\widehat{\mathcal{R}}_{i,j}^{0})_{(k)}=(\widehat{\mathcal{C}}_{i})_{(k)}^{H}(\widehat{\mathcal{S}_{i}})_{(k)}^{H}\left(\widehat{\mathcal{A}}_{(k)}\widehat{\mathcal{X}}_{(k)}^{0}\widehat{\mathcal{B}}_{(k)}-\widehat{\mathcal{C}}_{(k)}\right)(\widehat{\mathcal{V}_{j}})_{(k)}(\widehat{\mathcal{D}}_{j})_{(k)}$ for $i=1,\cdots,q_{\mathcal{S}}$ and $j=1,\cdots,q_{\mathcal{ V}}$
		\EndFor
		\For{$k=\lceil\frac{l+1}{2}\rceil+1,\cdots,l$}
		\State $(\widehat{\mathcal{R}}_{i,j}^{0})_{(k)}=\text{conj}((\widehat{\mathcal{R}}_{i,j}^{0})_{(l-k+2)})$ for $i=1,\cdots,q_{\mathcal{S}}$ and $j=1,\cdots,q_{\mathcal{V}}$
		\EndFor
		\For{$t=0,1,\cdots$}
		\State $f_{i,j}(\mathcal{X}^{t})=(1/l)\sum_{k=1}^{l}\|(\widehat{\mathcal{R}}_{i,j}^{t})_{(k)}\|_{F}^{2}$ for $i=1,\cdots,q_{\mathcal{ S}}$ and $j=1,\cdots,q_{\mathcal{ V}}$
		\State Sample $(i^{t},j^{t}) \sim \mathbf{p}_{\mathcal{ S},\mathcal{ V}}^{t}$, where $p_{i,j}^{t}=f_{i,j}(\mathcal{X}^{t})/(\sum_{i=1}^{q_{\mathcal{ S}}}\sum_{j=1}^{q_{\mathcal{ V}}}f_{i,j}(\mathcal{X}^{t}))$ for $i=1,\cdots,q_{\mathcal{ S}}$ and $j=1,\cdots,q_{\mathcal{ V}}$
		\For{$k=1,\cdots,\lceil\frac{l+1}{2}\rceil$}	
\State Update
  \begin{align*}
\widehat{\mathcal{X}}^{t+1}_{(k)}&=\widehat{\mathcal{X}}^{t}_{(k)}-\left(\widehat{\mathcal{M}}^{-1}_{(k)}\widehat{\mathcal{A}}_{(k)}^{H}(\widehat{\mathcal{S}}_{i^{t}})_{(k)}(\widehat{\mathcal{C}}_{i^{t}})_{(k)}\right)(\widehat{\mathcal{R}}_{i^{t},j^{t}}^{t})_{(k)}\\
&\quad\left((\widehat{\mathcal{D}}_{j^{t}})^{H}_{(k)}(\widehat{\mathcal{V}}_{j^{t}})_{(k)}^{H}\widehat{\mathcal{B}}^{H}_{(k)}\widehat{\mathcal{N}}^{-1}_{(k)}\right)
  \end{align*}
\State Update 
\begin{align*}
(\widehat{\mathcal{R}}_{i,j}^{t+1})_{(k)}&=(\widehat{\mathcal{R}}_{i,j}^{t})_{(k)}-\left((\widehat{\mathcal{C}}_{i})_{(k)}^{H}(\widehat{\mathcal{S}_{i}})_{(k)}^{H}\widehat{\mathcal{A}}_{(k)}\widehat{\mathcal{M}}^{-1}_{(k)}\widehat{\mathcal{A}}_{(k)}^{H}(\widehat{\mathcal{S}}_{i^{t}})_{(k)}(\widehat{\mathcal{C}}_{i^{t}})_{(k)}\right)\\
&\quad(\widehat{\mathcal{R}}_{i^{t},j^{t}}^{t})_{(k)}	\left((\widehat{\mathcal{D}}_{j^{t}})_{(k)}^{H}(\widehat{\mathcal{V}_{j^{t}}})_{(k)}^{H}\widehat{\mathcal{B}}_{(k)}^{H}\widehat{\mathcal{N}}^{-1}_{(k)}\widehat{\mathcal{B}}_{(k)}(\widehat{\mathcal{V}}_{j})_{(k)}(\widehat{\mathcal{D}}_{j})_{(k)}\right)
 \end{align*}
for $i=1,\cdots,q_{\mathcal{ S}}$ and $j=1,\cdots,q_{\mathcal{ V}}$
		\EndFor
		\For{$k=\lceil\frac{l+1}{2}\rceil+1,\cdots,l$}
		\State $\widehat{\mathcal{X}}^{t+1}_{(k)}=\text{conj}(\widehat{\mathcal{X}}^{t+1}_{(l-k+2)})$		
		\State $(\widehat{\mathcal{R}}_{i,j}^{t+1})_{(k)}=\text{conj}((\widehat{\mathcal{R}}_{i,j}^{t+1})_{(l-k+2)})$ for $i=1,\cdots,q_{\mathcal{ S}}$ and $j=1,\cdots,q_{\mathcal{ V}}$
		\EndFor
		\EndFor
		\State $\mathcal{X}^{t+1}= \texttt{ifft}\left(\widehat{\mathcal{X}}^{t+1},[~],3\right)$
  \State \textbf{Output:} last iterate $\mathcal{X}^{t+1}$
	\end{algorithmic}
\end{breakablealgorithm}




\end{appendices}


\bibliography{sn-bibliography}


\end{document}